\documentclass[a4paper,11pt]{amsart}%

\usepackage{amsmath}%
\usepackage{amssymb}%
\usepackage{amsfonts}%
\usepackage{amscd}%
\usepackage{mathrsfs}%
\usepackage[all]{xy}%
\usepackage{enumerate}%

\theoremstyle{plain}%
    \newtheorem{thm}{Theorem}[section]%
    \newtheorem{prop}[thm]{Proposition}%
    \newtheorem{lem}[thm]{Lemma}%
	\newtheorem{cor}[thm]{Corollary}%
    \newtheorem{con}[thm]{Conjecture}%

\theoremstyle{definition}%
    \newtheorem{ex}[thm]{Example}%

\theoremstyle{remark}%
    \newtheorem{rem}[thm]{Remark}%

\def\endpiece{xxx}%
\def\makeop[#1]{\xmakeop#1,xxx,}
\def\mkop#1{\expandafter\def\csname #1\endcsname{{\mathrm{#1}}}} %
\def\xmakeop#1,{\def\temp{#1}\ifx\temp\endpiece\else\mkop{#1}\expandafter\xmakeop\fi}%

\newcommand{\Ram}{\mathop{\mathrm {Ram}}\nolimits}
\newcommand{\Hom}{\mathop{\mathrm {Hom}}\nolimits}
\newcommand{\corank}{\mathop{\mathrm {corank}}\nolimits}
\newcommand{\rank}{\mathop{\mathrm {rank}}\nolimits}
\newcommand{\im}{\mathop{\mathrm {im}}\nolimits}

\newcommand{\Tr}{\mathop{\mathrm {Tr}}\nolimits}
\newcommand{\Frob}{\mathop{\mathrm {Frob}}\nolimits}
\newcommand{\If}{\mathop{\mathrm {if}}\nolimits}
\newcommand{\Gal}{\mathop{\mathrm {Gal}}\nolimits}

\newcommand{\Sel}{\mathop{\mathrm {Sel}}\nolimits}
\newcommand{\PD}{\mathop{\mathrm {PD}}\nolimits}

\def\A{{\mathbb{A}}}%
\def\Z{{\mathbb{Z}}}%
\def\Q{{\mathbb{Q}}}%
\def\Zp{{\mathbb{Z}_p}}%
\def\Qp{{\mathbb{Q}_p}}%
\def\C{{\mathbb{C}}}
\def\pa{{\mathrm{par}}}%
\def\Res{{\mathrm{Res}}}%
\def\ideal#1{{\mathfrak{#1}}}%
\def\integer#1{{\mathcal{O}_{#1}}}%
\def\line#1{{\overline{#1}}}%
\def\cal#1{{\mathcal{#1}}}%
\def\G_m^#1{{\mathbb{G}_m^{#1}}}%
\makeop[Tr,Nr,Ver,Ker,Aut,det,sgn,End,ev]%
\makeop[GL,SL,E,M,K,SK]%
    
\begin{document}

\setcounter{tocdepth}{1} 
\numberwithin{equation}{section}


\title[congruences between Hilbert modular forms and the Iwasawa $\lambda$-invariants]
{congruences between Hilbert modular forms of weight $2$, and the Iwasawa $\lambda$-invariants}
\author{Yuichi Hirano}
\address{Graduate School of Mathematical Sciences, The University of
Tokyo, 8-1 Komaba 3-chome, Meguro-ku, Tokyo, 153-8914, Japan}
\email{yhirano@ms.u-tokyo.ac.jp}
\date{\today}

\begin{abstract}
The purpose of this paper is to prove the equality between the algebraic Iwasawa $\lambda$-invariant and the analytic Iwasawa $\lambda$-invariant for a Hilbert cusp form of parallel weight $2$ at an ordinary prime $p$ when the associated residual Galois representation is reducible. 
This is a generalization of a result of R.\,Greenberg and V.\,Vatsal. 
\end{abstract}

\maketitle
\setcounter{section}{-1}

\section{Introduction}
%

%
\subsection{Introduction}
%

In this paper, we study a way to obtain the equality between the algebraic Iwasawa $\lambda$-invariant and the analytic Iwasawa $\lambda$-invariant from a congruence between a Hilbert cusp form of parallel weight $2$ and a Hilbert Eisenstein series of parallel weight $2$. 
Our result is a generalization of the work of R.\,Greenberg and V.\,Vatsal \cite{Gre--Vat} for elliptic modular forms, which is not covered by the advanced work of C.\,Skinner and E.\,Urban \cite{Ski--Ur} who assume absolute residual irreducibility. 

Let $F$ be a totally real number field of degree $n$ with narrow class number $h_F^+=1$. 
Let $\mathfrak{o}_F$ denote the ring of integers of $F$, and let $\Delta_F$ denote the discriminant of $F$. 
Let $\mathfrak{n}$ be a non-zero ideal of $\ideal{o}_F$ such that 
$\mathfrak{n}$ is prime to $6\Delta_F$. 
Let $p$ be a prime number such that $p\ge n+2$ and $p$ is prime to $6\mathfrak{n}\Delta_F$. 
We fix algebraic closures $\overline{\Q}$ of $\Q$ and $\overline{\Q}_p$ of $\Qp$, and embeddings $\overline{\Q} \hookrightarrow \overline{\Q}_p \hookrightarrow \C$. 
Let $\integer{}$ be the ring of integers of a finite extension $K$ of the field $\Phi_p$ defined in \S \ref{Hecke}, $\varpi$ a uniformizer of $\integer{}$, and $\kappa$ the residue field of $\integer{}$. 
Let $F_{\infty}$ denote the cyclotomic $\Zp$-extension of $F$. We put $\Gamma=\Gal(F_{\infty}/F)$. Let $\Lambda$ denote the Iwasawa algebra $\integer{}[[\Gamma]]$.
Let $S_2(\ideal{n},\integer{})$ denote the space of Hilbert cusp forms of parallel weight $2$ and level $\ideal{n}$ with coefficients in $\integer{}$ (see (\ref{integral MF})).
Let $Y(\ideal{n})$ be the Shimura variety $\Gamma_{1}(\ideal{d}_F[t_1],\ideal{n})\backslash \mathfrak{H}^{\mathrm{Hom}(F, \,\mathbb{R})}$ defined by $($\ref{adele HMV}$)$, and let $Y(\ideal{n})^{\mathrm{BS}}$ be the Borel\nobreakdash--Serre compactification of $Y(\ideal{n})$ $($cf. \S \ref{subsection:modular symbol}$)$. 
Let $C$ denote the set of all cusps of $Y(\ideal{n})$, and let $D_s$ denote the boundary of $Y(\ideal{n})^{\mathrm{BS}}$ at $s\in C$.
Let $C_{\infty}$ be the subset of $C$ consisting of cusps $\Gamma_0(\ideal{d}_F [t_1],\ideal{n})$\nobreakdash-equivalent to the cusp $\infty$ (where $\Gamma_0(\ideal{d}_F [t_1],\ideal{n})$ is the congruence subgroup defined in \S\ref{Analytic HMV}). 
Let $D_{C_{\infty}}(\ideal{n})$ denote the union of $D_s$ for all $s\in C_{\infty}$.

Let $\mathbf{f}\in S_2(\ideal{n},\integer{})$ be a normalized Hecke eigenform for all Hecke operators with character $\varepsilon$ and $p$-ordinary. 
We assume that the residual Galois representation $\bar{\rho}_{\mathbf{f}}\,(=\rho_{\mathbf{f}}\,(\bmod\, \varpi)): G_{F}\to \GL_2(\kappa)$ associated to $\mathbf{f}$ is reducible and of the form
\[
\bar{\rho}_{\mathbf{f}} \sim 
\overline
{\begin{pmatrix} 
\varphi&\ast\\ 
0&\psi
\end{pmatrix}}.
\]

\begin{thm}
\label{thm:main theorem}
Let $\mathbf{f}\in S_2(\ideal{n},\integer{})$ as above.
We assume that the above $\varphi$ $($resp. $\psi$$)$ is ramified $($resp. unramified$)$ at every prime ideal of $\ideal{o}_F$ lying above $p$, totally even $($resp. totally odd$)$, and the associated algebraic Iwasawa $\mu$-invariant is $0$ $($for the definition, see \S \ref{subsection:algebraic Iwasawa invariants} $(\mu=0)$$)$. 
Let $\ideal{n}'$ be the least common multiple of $\ideal{n}^2$ and $\ideal{m}_{\varepsilon} \ideal{m}_{\psi}^2$, where $\ideal{m}_{\ast}$ is the conductor of $\ast$. 
We put $\mathbf{g}=(\mathbf{f}\otimes \mathbf{1}_{\ideal{n}})\otimes \mathbf{1}_{\ideal{m}_{\psi}}\in S_2(\ideal{n}',\integer{})$ $($cf. \cite[Proposition 4.4, 4.5]{Shi}$)$, where, for a non-zero ideal $\ideal{a}$ of $\ideal{o}_F$, $\mathbf{1}_{\ideal{a}}$ is the trivial character modulo $\ideal{a}$. 
Let $\chi$ be an $\integer{}$-valued narrow ray class character of $F$, whose conductor is the level $\ideal{n}'$, such that $\chi$ is of order prime to $p$, totally even, and the algebraic Iwasawa $\mu$-invariants of $\varphi\chi$ and $\psi\chi$ are $0$. 
Let $\mathbf{f}\otimes \chi$ be the Hecke eigenform twisted by $\chi$ $($defined in \cite[Proposition 4.4, 4.5]{Shi}$)$. 
We assume the following two conditions$:$
\begin{enumerate}[$(a)$] 

\item the local components $H^{n}(\partial \left(Y(\ideal{n}')^{\mathrm{BS}}\right),\integer{})_{\ideal{m}}$ and $H_c^{n+1}(Y(\ideal{n}'),\integer{})_{\ideal{m}}$ are torsion-free, where $\ideal{m}$ is the maximal ideal of $\cal{H}_2(\ideal{n}',\integer{})$ defined at the beginning of \S \ref{subsection:application for alg and anal}$;$

\item the local component $H^{n}(D_{C_{\infty}}(\ideal{n}'),\integer{})_{\ideal{m}_{\mathbf{g}}'}$ is torsion-free, where $\ideal{m}_{\mathbf{g}}'$ is the maximal ideal of $\mathbb{H}_2(\ideal{n}',\integer{})'$ defined before Proposition \ref{prop:integral of relative class}. 

\end{enumerate}
Then we have the following$:$
\begin{enumerate}[$(1)$] 

\item \label{main thm:alg side}
$($Theorem \ref{thm:Selmer for 1 and 2}, \S\ref{subsection:application for algebraic}$)$ the Selmer group $\Sel(F_{\infty},A_{\mathbf{f}\otimes\chi})$ for $\mathbf{f}\otimes\chi$ $($for the definition, see \S \ref{subsection:algebraic Iwasawa invariants}, \S \ref{subsection:application for algebraic}$)$ is finitely generated $\Lambda$-cotorsion, and 
\[
\lambda_{\mathbf{f}\otimes \chi}^{\mathrm{alg}}=\lambda_{\varphi\chi,\Sigma_0}+\lambda_{\psi\chi,\Sigma_0}. 
\]
Here $\lambda_{\mathbf{f}\otimes \chi}^{\mathrm{alg}}$ is the algebraic Iwasawa $\lambda$-invariant for $\mathbf{f}\otimes\chi$ defined by $($\ref{lambda^alg for f and chi}$)$, and $\lambda_{\varphi\chi,\Sigma_0}$ and $\lambda_{\psi\chi,\Sigma_0}$ are the classical algebraic Iwasawa $\lambda$-invariants defined by $(\ref{classical lambda for varphichi})$ and $(\ref{classical lambda for psichi})$, respectively.

\item \label{main thm:an side}
$($Theorem \ref{thm:integral p-adic L}$)$ if $\chi$ is of type $S$ $($that is, $\overline{\Q}^{\ker(\chi)} \cap F_{\infty}=F$$)$, then there exists a $p$-adic $L$-function $\mathscr{L}_p(\mathbf{f},\chi,T)\in \integer{}[[T]]$ satisfying the following interpolation property$:$
for each finite order character $\rho$ of $\Gamma$ with conductor $p^{\nu_{\rho}}$, 
\[
\mathscr{L}_p(\mathbf{f},\chi,\rho(\gamma)-1)
=\alpha^{-\nu_{\rho}}\tau(\chi^{-1}\rho^{-1})\frac{D(1,\mathbf{f},\chi\rho)}{(-2\pi \sqrt{-1})^n\Omega_{\mathbf{g}}^{\mathbf{1}}}\in \integer{}(\rho).
\]
Here $\alpha^{-\nu_{\rho}}=\prod_{\ideal{p}|p}\alpha_{\ideal{p}}^{-\nu_{\rho,\ideal{p}}}$ for $p^{\nu_{\rho}}=\prod_{\ideal{p}|p}\ideal{p}^{\nu_{\rho,\ideal{p}}}$, where $\alpha_{\ideal{p}}$ is a unit root of the polynomial defined by $($\ref{unit root}$)$, $\tau(\chi^{-1}\rho^{-1})$ denotes the Gauss sum of $\chi^{-1}\rho^{-1}$ defined by \cite[(3.9)]{Shi}, $D(1,\mathbf{f},\chi\rho)$ is given by the Dirichlet series in the sense of G.\,Shimura $($see $($\ref{L-function}$)$$)$, $\Omega_{\mathbf{g}}^{\mathbf{1}}\in \C^{\times}$ is the canonical period $($defined in \S \ref{subsection:canonical period}$)$ of $\mathbf{g}\in S_2(\ideal{n}',\integer{})$ and the trivial character $\mathbf{1}$ of the Weyl group $W_G$, and $\integer{}(\rho)$ denotes the ring of integers of the field generated by $\im(\rho)$ over $K$.

\item \label{main thm:equality}
$($Theorem \ref{thm:equality of Iwasawa invariants}$)$ 
the algebraic Iwasawa $\lambda$-invariant $\lambda_{\mathbf{f}\otimes \chi}^{\mathrm{alg}}$ above is equal to the analytic Iwasawa $\lambda$-invariant $\lambda_{\mathbf{f}\otimes \chi}^{\mathrm{an}}$ for $\mathbf{f}\otimes\chi$ defined by $($\ref{lambda^an for f and chi}$)$$:$ 
\begin{align*}
\lambda_{\mathbf{f}\otimes \chi}^{\mathrm{alg}}&=\lambda_{\mathbf{f}\otimes \chi}^{\mathrm{an}}(=\lambda_{\varphi\chi,\Sigma_0}+\lambda_{\psi\chi,\Sigma_0}). 
\end{align*}
\end{enumerate}
\end{thm}

This result is a generalization of the result of Greenberg and Vatsal \cite{Gre--Vat} in the case where $F=\Q$ and weight $k=2$. 
However, the methods to prove Theorem \ref{thm:main theorem} (\ref{main thm:an side}) and (\ref{main thm:equality}) have some limitations, such as the need for the torsion-freeness on the compactly supported cohomology and the boundary cohomology. 
In the case where $F$ is a real quadratic field, the torsion-freeness is satisfied under some conditions (Proposition \ref{prop:ab torsion-free} and \ref{prop:bou torsion-free}).
We also give examples satisfying all the assumptions of the main theorem (Example \ref{Example cong}).

\begin{rem}\label{rem:Fe--Wa}
The assumption on the algebraic Iwasawa $\mu$-invariant of a narrow ray class character $\eta$ is satisfied if $\overline{\Q}^{\ker (\eta)}$ is an abelian extension over $\Q$ by Ferrero\nobreakdash--Washington theorem (see, for example, \cite[\S 7.5, Theorem 7.15]{Was}).
\end{rem}

Our proof based on the method of Greenberg and Vatsal requires new ingredients; a criterion for the $\Lambda$\nobreakdash-cotorsionness of Selmer groups (\S \ref{subsection:algebraic Iwasawa invariants}), a construction of a $\C$-valued distribution interpolating special values of the $L$-functions (\S \ref{subsection:const of p-adic L}), a criterion for the integrality of special values of the ($p$-adic) $L$-functions divided by the canonical period defined in terms of parabolic cohomology (\S \ref{subsection:construction of p-adic L}), and a congruence between our $p$-adic $L$-function and the product of two Deligne--Ribet $p$-adic $L$-functions (\S \ref{subsection:proof of main theorem}). 
We give an outline of the proof of the main theorem 
(Theorem \ref{thm:main theorem}) below in order to clarify its complicated structure, the methods used, and the places where the assumptions are necessary.
The proof consists of three steps.

\textbf{Step 1.} To prove Theorem \ref{thm:main theorem} (\ref{main thm:alg side}).

Since $\overline{\rho}_{\mathbf{f}\otimes\chi}$ is reducible, we have an exact sequence of residual Galois representations 
\[
0 \rightarrow A_{\varphi\chi}[\varpi] \rightarrow A_{\mathbf{f}\otimes\chi}[\varpi] \rightarrow A_{\psi\chi}[\varpi] \rightarrow 0.
\]
Here $A_{\ast}$ denotes a torsion quotient of the Galois representation associated to $\ast\in\{\mathbf{f}\otimes\chi,\varphi\chi,\psi\chi\}$ (for the definition, see \S \ref{subsection:def of Selmer}), which is an $\integer{}$-module, and $A_{\ast}[\varpi]$ denotes the $\varpi$\nobreakdash-torsion of $A_{\ast}$. 
By the parity conditions on $\varphi\chi$ and $\psi\chi$, this sequence induces an exact sequence of the $\varpi$-torsion of (non-primitive) Selmer groups 
\begin{align}\label{pi-torsion parts}
0 \rightarrow \Sel^{\Sigma_0}(F_{\infty},A_{\varphi\chi}) [\varpi] 
\rightarrow \Sel(F_{\infty},A_{\mathbf{f}\otimes\chi})[\varpi] 
\rightarrow \Sel^{\Sigma_0}(F_{\infty},A_{\psi\chi})[\varpi] \rightarrow 0
\end{align}
(Theorem \ref{thm:Selmer for 1 and 2} (\ref{key-seq})).
It follows from the isomorphism $\Sel(F_{\infty},A_{\mathbf{f}\otimes\chi})=\Sel^{\Sigma_0}(F_{\infty},A_{\mathbf{f}\otimes\chi})$ (cf. Proposition \ref{prop:generator}), the vanishing results of $H^0(F_{\infty},A_{\psi\chi}[\varpi])$ (Lemma \ref{lem:H^0=0}) and $H^2(F_{\infty},A_{\varphi\chi}[\varpi])$ (Lemma \ref{lem:H^2=0}), and the isomorphism $\Sel^{\Sigma_0}(F_{\infty},A[\varpi])\simeq \Sel^{\Sigma_0}(F_{\infty},A)[\varpi]$ for $A=A_{\mathbf{f}\otimes\chi}$, $A_{\varphi\chi}$, or $A_{\psi\chi}$ (Proposition \ref{prop:Selmer mod pi}). 
Now, by the assumption on the algebraic Iwasawa $\mu$-invariants of $\varphi\chi$ and $\psi\chi$, the group $\Sel(F_{\infty},A_{\mathbf{f}\otimes\chi})[\varpi]$ is finite, and hence $\Sel(F_{\infty},A_{\mathbf{f}\otimes\chi})$ is finitely generated $\Lambda$-cotorsion (Theorem \ref{thm:Selmer for 1 and 2}). 

By Proposition \ref{prop:corank=dim}, we have 
$\corank_{\integer{}}(\Sel(F_{\infty},A_{\mathbf{f}\otimes\chi}))=\dim_{\kappa}(\Sel(F_{\infty},A_{\mathbf{f}\otimes\chi}[\varpi]))$ and $\corank_{\integer{}}(\Sel^{\Sigma_0}(F_{\infty},A))=\dim_{\kappa}(\Sel^{\Sigma_0}(F_{\infty},A[\varpi]))$ for $A=A_{\varphi\chi}$ and $A_{\psi\chi}$. 
It follows from the fact obtained by T.\,Weston \cite{We} that, for $A=A_{\mathbf{f}\otimes\chi}$, $A_{\varphi\chi}$, and $A_{\psi\chi}$, $\Sel(F_{\infty},A)$ has no proper $\Lambda$-submodules of finite index. 
Now, again using the exact sequence (\ref{pi-torsion parts}), we obtain the equality $\lambda_{\mathbf{f}\otimes\chi}^{\text{alg}}=\lambda_{\varphi\chi,\Sigma_0}+\lambda_{\psi\chi,\Sigma_0}$ (Theorem \ref{thm:Selmer for 1 and 2} (\ref{thm:Selmer for 1 and 2:corank})).
The proof is based on the method of Greenberg and Vatsal \cite[\S 2]{Gre--Vat}.

\textbf{Step 2.} To prove Theorem \ref{thm:main theorem} (\ref{main thm:an side}). 

We first construct a $\C$-valued distribution attached to a Hilbert cusp form $\mathbf{f}$ of parallel weight $2$ (Proposition \ref{prop:distribution}) interpolating special values of the associated $L$-functions (Proposition \ref{prop:interpolation}). 
This is a generalization of results of Y.\,Amice and J.\,V\'elu \cite{Ami--Ve}, M.\,Vishik \cite{Vi}, and B.\,Mazur, J.\,Tate, and J.\,Teitelbaum \cite{MTT}.
The proof is based on Mellin transforms for the compactly supported cohomology class $[\omega_{\mathbf{f}}]_c$ in $H_c^n(Y(\ideal{n}),\C)$ (Proposition \ref{Mellin hol} and \ref{Mellin anti-hol}), which implies that the special values of the associated $L$-functions are expressed as a linear combination of the images of $[\omega_{\mathbf{f}}]_c$ under the evaluation maps with integral coefficients. 

Next we prove the integrality of the $p$-adic $L$-function.
Let $[\omega_{\mathbf{g}}]$ (resp. $[\omega_{\mathbf{g}}]_{\mathrm{rel}}$) denote the image of $[\omega_{\mathbf{g}}]_c$ in the parabolic cohomology $H_{\pa}^n(Y(\ideal{n}'),\C)$ 
(resp. the relative cohomology $H^{n}(Y(\ideal{n}')^{\mathrm{BS}},D_{C_{\infty}}(\ideal{n}');\C)$).
By the Eichler--Shimura\nobreakdash--Harder isomorphism (\ref{+,+ decomp}) and the $q$-expansion principle over $\C$,
there exists $\Omega_{\mathbf{g}}\in \C^{\times}$ such that the class $[\omega_{\mathbf{g}}]/\Omega_{\mathbf{g}}$ belongs to the torsion-free part $H_{\pa}^{n}(Y(\ideal{n}'),\integer{})/(\text{$\integer{}$-torsion})$ of the parabolic cohomology, and its reduction modulo $\varpi$ does not vanish. 
By using a vanishing result on $H^{n-1}(D_{C_{\infty}}(\ideal{n}'),\C)$ (\cite[Proposition 5.3]{Hira}), we prove that the class $[\omega_{\mathbf{g}}]_{\mathrm{rel}}/\Omega_{\mathbf{g}}$ belongs to the torsion\nobreakdash-free part $H^{n}(Y(\ideal{n}')^{\mathrm{BS}},D_{C_{\infty}}(\ideal{n}');\integer{})/(\text{$\integer{}$-torsion})$ of the relative cohomology (Proposition \ref{prop:integral of relative class}), where we use the assumptions (b), $h_F^+=1$, and weight $k=2$.
This is a generalization of an argument of Manin\nobreakdash--Drinfeld in the case where $F=\Q$ (see, for example, \cite[Lemma 6.9 (b)]{Gre--Ste}).
Now the integrality of the $p$-adic $L$-function follows from the fact $D(s,\mathbf{f},\chi\rho)=D(s,\mathbf{g},\chi\rho)$, the Mellin transforms for the relative cohomology class $[\omega_{\mathbf{g}}]_{\mathrm{rel}}$ (\cite[Proposition 2.5 and 2.6]{Hira}), and the integrality of $[\omega_{\mathbf{g}}]_{\mathrm{rel}}/\Omega_{\mathbf{g}}$ (Theorem \ref{thm:integral p-adic L}), where we use the assumption that the conductor of $\chi$ is the level $\ideal{n}'$.

\textbf{Step 3.} To prove Theorem \ref{thm:main theorem} (\ref{main thm:equality}).

Since $\lambda_{\mathbf{f}\otimes \chi}^{\mathrm{alg}}=\lambda_{\varphi\chi,\Sigma_0}+\lambda_{\psi\chi,\Sigma_0}$ (mentioned in Step 1), Theorem \ref{thm:main theorem} (\ref{main thm:equality}) follows from the equality $\lambda_{\mathbf{f}\otimes \chi}^{\mathrm{an}}=\lambda_{\varphi\chi,\Sigma_0}+\lambda_{\psi\chi,\Sigma_0}$ (Theorem \ref{congruence of p-adic L} (\ref{equality of the analytic Iwasawa invariants})). 
The equality is obtained by the Iwasawa main conjecture for totally real number fields (proved by A.\,Wiles \cite{Wil90}) and a congruence between our $p$\nobreakdash-adic $L$-function and the product of two Deligne--Ribet $p$\nobreakdash-adic $L$-functions (Theorem \ref{congruence of p-adic L} (\ref{congruence of p-adic L})). 
The latter follows from congruences between special values of $L$-functions obtained from a congruence between the Hilbert cusp form $\mathbf{g}\in S_2(\ideal{n}',\integer{})$ and a Hilbert Eisenstein series $\mathbf{G}\in M_2(\ideal{n}',\integer{})$ induced by $\psi$ and $\varepsilon\psi^{-1}$ (by \cite[Theorem 6.1]{Hira}), where we use the assumptions (a), (b). 
The proof is based on the method of Greenberg and Vatsal \cite[Theorem 3.11]{Gre--Vat} by using the result \cite[Theorem 6.1]{Hira} instead of the result of Vatsal \cite[Theorem 2.10]{Vat}.\\

The organization of this paper is as follows. 

In \S \ref{Hilbert}, 
we summarize results on Hilbert modular varieties and Hilbert modular forms in the analytic and algebraic settings.

In \S \ref{section:Selmer}, we generalize results of Greenberg and Vatsal \cite{Gre--Vat} on the algebraic side. 
We prove that the Selmer group for a Hilbert cusp form is finitely generated $\Lambda$-cotorsion, and the associated Iwasawa $\lambda$-invariant is equal to the sum of classical Iwasawa $\lambda$-invariants under some assumptions (Theorem \ref{thm:Selmer for 1 and 2}).

In \S \ref{section:p-adic L}, we construct a $\C$-valued distribution attached to a Hilbert cusp form (Proposition \ref{prop:distribution}), which interpolates special values of the associated $L$-functions (Proposition \ref{prop:interpolation}). 

In \S \ref{section:integrality of p-adic L},
we prove the integrality of the $p$-adic $L$-function attached to a Hilbert cusp form divided by the canonical period (Theorem \ref{thm:integral p-adic L}). 

In \S \ref{section:Equality between the Iwasawa invariants}, we generalize results of Greenberg and Vatsal \cite{Gre--Vat} on the analytic side. 
We prove that the analytic Iwasawa $\lambda$-invariant for a Hilbert cusp form is equal to the sum of classical Iwasawa $\lambda$-invariants under some assumptions (Theorem \ref{congruence of p-adic L} (\ref{equality of the analytic Iwasawa invariants})). 
The key ingredient of our proof is the congruence between special values of the $L$\nobreakdash-functions obtained from a congruence between a Hilbert cusp form of parallel weight $2$ and a Hilbert Eisenstein series of parallel weight $2$, which is the main theorem of \cite{Hira}. 

In \S \ref{section:Modularity}, we give examples satisfying all the assumptions of the main theorem (Example \ref{Example cong}).

\tableofcontents

%
\subsection{Notation}
%

Let $\widehat{\Z}$ denote $\prod_{l}\Z_l$, where $l$ runs over all rational primes. 
We abbreviate $\A_{\Q}$, the ring of adeles of $\Q$, to $\A$. 
We fix a rational prime number $p>3$.
We fix algebraic closures $\line{\Q}$ of $\Q$ and $\line{\Q}_p$ of the field of $p$-adic numbers $\Qp$, 
and embeddings $\iota_p:\line{\Q} \to \line{\Q}_p$ and $\line{\Q}_p \rightarrow \C$,
where $\mathbb{C}$ denotes the field of complex numbers.

Let $F$ be a totally real number field unramified at $p$, $n$ the degree $[F\colon\Q]$ of the extension $F/\Q$, and $\mathfrak{o}_F$ the ring of integers of $F$.
For a place $v$ of $F$ (resp. a non-zero prime ideal $\ideal{q}$ of $\ideal{o}_F$), let $F_v$ (resp. $F_{\ideal{q}}$) denote the completion at $v$ (resp. the $\ideal{q}$-adic completion) of $F$.
Let $\ideal{o}_{F_{\ideal{q}}}$ denote the ring of integers of $F_{\ideal{q}}$, and $\widehat{\ideal{o}}_F$ the product of $\ideal{o}_{F_{\ideal{q}}}$ over all non-zero prime ideals $\ideal{q}$ of $\ideal{o}_F$.
Let $J_F$ denote the set of embeddings of $F$ into $\mathbb{R}$.
For $a\in F$ and $\iota\in J_F$, let $a^{\iota}$ denote $\iota(a)$. 
We have $F\otimes_{\Q}\mathbb{R} \simeq \mathbb{R}^{J_F}$, and write $(F\otimes_{\Q}\mathbb{R})_+^{\times}$ for the subgroup of $(F\otimes_{\Q}\mathbb{R})^{\times}$ corresponding to $(\mathbb{R}_+^{\times})^{J_F}$, where $\mathbb{R}_+^{\times}$ denotes the multiplicative group of positive real numbers.
As usual, $\A_F$ denotes the ring of adeles of $F$, which is the product of the finite part $\A_{F,f}(\simeq \widehat{\ideal{o}}_F \otimes_{\ideal{o}_F}F)$ and the infinite part $\A_{F,\infty}(\simeq F\otimes_{\Q}\mathbb{R})$.
For $x\in \A_F$ and a place $v$ of $F$, $x_0$, $x_{\infty}$, and $x_v$ denote the finite component $\in \A_{F,f}$, the infinite component $\in \A_{F,\infty}$, and the $v$-component $\in F_v$ of $x$, respectively. 
For $x \in \A_F$, a subset $X$ of $\A_F$, and a non-zero ideal $\ideal{n}$ of $\ideal{o}_F$, we write $x_{\ideal{n}}$ and $X_{\ideal{n}}$ for the images of $x$ and $X$ in $\prod_{\ideal{q}|\ideal{n}} F_{\ideal{q}}$, where
$\ideal{q}$ denotes a non-zero prime ideal of $\ideal{o}_F$. 
Let $N$ denote the norm map $\mathrm{Nr}_{F/\Q}$ of the extension $F/\Q$, 
$\ideal{d}_F \subset \ideal{o}_F$ the different of $F$, and $\Delta_F$ the discriminant $N(\ideal{d}_F)$ of $F$, which is prime to $p$ by assumption. 
Let $\mathrm{Cl}_F^+$ denote the narrow ideal class group of $F$. 
We have an isomorphism 
$F^{\times}\backslash \A_F^{\times}/\widehat{\ideal{o}}_F^{\times}(F\otimes_{\Q}\mathbb{R})_+^{\times} \xrightarrow{\simeq} \mathrm{Cl}_F^+$ sending the class of $x\in \A_F^{\times}$ to the class of the fractional ideal
$[x]:=\prod_{\ideal{q}} \ideal{q}^{\mathrm{ord}_{\ideal{q}}(x_{\ideal{q}})}$, where $\ideal{q}$ runs over the set of all non-zero prime ideals of $\ideal{o}_F$. Let $D$ be an element of $\A_F^{\times}$ such that $[D] = \ideal{d}_F$ and $D_{\infty}=1$.

For a non-zero ideal $\ideal{b}$ of $\ideal{o}_F$, 
let $\mathrm{Cl}_F^+(\ideal{b})$ denote the narrow ray class group of $F$ modulo $\ideal{b}$.
By a narrow ray class character of $F$ modulo $\ideal{b}$, we mean a homomorphism $\chi : \mathrm{Cl}_F^+(\ideal{b}) \to \C^{\times}$.
The conductor of $\chi$ is the smallest divisor $\ideal{m}_{\chi}$ of $\ideal{b}$ such that $\chi$ factors through $\mathrm{Cl}_F^+(\ideal{m}_{\chi})$.
For a narrow ray class character $\chi$ of $F$ modulo $\ideal{b}$, there exists an $r = (r_{\iota})_{\iota\in J_F}\in (\Z/2\Z)^{J_F}$ such that
\[
\chi((\alpha))=\sgn(\alpha)^r \ \ \ \  \text{for all}\ \ \ \ \alpha\in F^{\times}\ \ \ \ \text{satisfying}\ \ \ \ \alpha\equiv 1\ (\bmod \ideal{b}).
\]
Here $\sgn(x)$ for $x \in \mathbb{R}^{\times}$ denotes the sign of $x$ and $\sgn(\alpha)^r=\prod_{\iota\in J_F}\sgn(\alpha^{\iota})^{r_{\iota}}$, where we identify $J_F$ with the set of infinite places of $F$. 
We call $r$ the sign of $\chi$. 
We say that $\chi$ is totally even (resp. totally odd) if $r_{\iota} = 0$ (resp. $r_{\iota} = 1$) for all $\iota \in J_F$.

For an algebraic group $H$ defined over $\Q$, 
$H(\mathbb{R})$ is abbreviated to $H_{\infty}$ and 
$H_{\infty,+}$ denotes the connected component of $H_{\infty}$ containing the unit. 
Let $G$ denote the reductive algebraic group $\Res_{F/\Q}\mathrm{GL}_2$ over $\Q$, where $\Res_{F/\Q}$ denotes the Weil restriction of scalars. 
We have $G_{\infty}= \GL_2(\mathbb{R})^{J_F}$, $G_{\infty,+}= \GL_2(\mathbb{R})_+^{J_F}$, and $G(\A)=\GL_2(\A_F)$. 
Let $B$ denote the Borel subgroup of $G$ consisting of upper triangular matrices, and let $U$ denote its unipotent radical.

%
\subsection{Acknowledgement}
%

I would like to thank Professor Takashi Taniguchi (Kobe University) for patiently providing comments and suggestions on the work in \S \ref{subsection:Mellin transform}. 
I am grateful to Professor Takeshi Tsuji (Tokyo University) for helpful discussions on Proposition \ref{prop:bou torsion-free}. 
I am also grateful to Professor Iwao Kimura (Toyama University) for helpful comments and suggestions on Example \ref{Example cong}.

%
\section{Hilbert modular varieties and Hilbert modular forms}\label{Hilbert}
%

%
\subsection{Analytic Hilbert modular varieties}\label{Analytic HMV}
%

In this subsection, we recall the definition of analytic Hilbert modular varieties. 
For more detail, refer to \cite[\S 1.1]{Dim}. 

Let $\mathfrak{H}$ be the upper half plane $\{ z \in \mathbb{C} \mid \mathrm{Im}(z) > 0 \}$. 
The group $\GL_2(\mathbb{R})_+$ acts on $\mathfrak{H}$ by linear fractional transformations. 
We can extend the action to $\GL_2(\mathbb{R})$ by defining the action of $\begin{pmatrix}-1&0\\ 0&1\\\end{pmatrix}$ on $\mathfrak{H}$ by $z\mapsto -\bar{z}$. 
We define the action of $G_{\infty}=\GL_2(\mathbb{R})^{J_F}$ on $\mathfrak{H}^{J_F}$ by $(g_{\iota})_{\iota\in J_F}\cdot(z_{\iota})_{\iota\in J_F}=(g_{\iota}z_{\iota})_{\iota\in J_F}$.
Let $\underline{\textbf{i}}=(\sqrt{-1},\cdots,\sqrt{-1}) \in \mathfrak{H}^{J_F}$. 
Let $K_{\infty}$ and $K_{\infty,+}$ be the stabilizers of $\underline{\textbf{i}}$ in $G_{\infty}$ and $G_{\infty,+}$, respectively. 

For a non-zero ideal $\mathfrak{n}$ of $\mathfrak{o}_F$, we define the open compact subgroup $K_1(\ideal{n})$ of $G(\mathbb{A}_f)$ by 
\begin{align*}
&K_1(\mathfrak{n})=\left\{\begin{pmatrix}a&b\\c&d\\\end{pmatrix}\in G(\widehat{\Z})\bigg\vert c\in \mathfrak{n},d-1\in \mathfrak{n}\right\}. 
\end{align*}
The adelic Hilbert modular variety $Y(\ideal{n})$ of level $K_1(\ideal{n})$ is defined by
\begin{align}\label{adele HMV}
Y(\ideal{n})&=G(\Q)\backslash G(\mathbb{A})/K_1(\ideal{n}) K_{\infty,+}\\
&\nonumber=G(\Q)_+\backslash G(\mathbb{A})_+/K_1(\ideal{n}) K_{\infty,+},
\end{align}
where $G(\A)_+=G(\A_f) G_{\infty,+}$ and $G(\Q)_+=G(\Q)\cap G_{\infty,+}$. 
Then $Y(\ideal{n})$ is a disjoint union of finitely many arithmetic quotients of $\mathfrak{H}^{J_F}$ as follows. 
Let $\ideal{a}$ be a fractional ideal of $F$.
We consider the following congruence subgroups of $G(\Q)_+$:
\begin{align}\label{cong subgroup}
\Gamma_0(\ideal{a},\mathfrak{n})
&= \left\{ \begin{pmatrix}
a&b\\
c&d
\end{pmatrix} \in \begin{pmatrix}\mathfrak{o}_F&\ideal{a}^{-1}\\ \ideal{a}\mathfrak{n}&\mathfrak{o}_F\end{pmatrix} \bigg| ad-bc \in \mathfrak{o}_{F,+}^{\times} \right\},\\
\nonumber \Gamma_1(\ideal{a},\mathfrak{n})
&=\left\{ \begin{pmatrix}
a&b\\
c&d
\end{pmatrix} \in \Gamma_0(\ideal{a},\mathfrak{n}) \bigg| d\equiv  1 (\bmod \mathfrak{n}) \right\},\\
\nonumber \Gamma_1^1(\ideal{a},\mathfrak{n})
&=\Gamma_1(\ideal{a},\mathfrak{n})\cap \SL_2(F), 
\end{align}
where $\mathfrak{o}_{F,+}^{\times}$ denotes the subgroup of $\mathfrak{o}_F^{\times}$ consisting of totally positive units. 
Let $\mathrm{Cl}_F^+$ be the narrow ideal class group of $F$ and 
$h_F^+$ the narrow class number $\sharp\mathrm{Cl}_F^+$ of $F$. 
Choose and fix $t_1,\cdots,t_{h_F^+} \in \A_F^{\times}$ such that $t_{i,\infty}=1$ and the corresponding fractional ideals $[t_1],\cdots,[t_{h_F^+}]$ form a complete set of representatives of $\mathrm{Cl}_F^+$. 
Throughout the paper, we assume that 
\begin{align}\label{prime to p}
\text{$[t_i]$ is prime to $p$ for each $i$. }
\end{align}
We put 
\[
x_i=\begin{pmatrix}D^{-1}t_i^{-1}&0\\0&1\\\end{pmatrix}\in G(\A)_+. 
\]
We define the analytic Hilbert modular varieties $Y_i$ by 
\begin{align}\label{analytic HMV}
Y_i=\overline{\Gamma_1(\ideal{d}_F[t_i],\mathfrak{n})}\backslash \mathfrak{H}^{J_F}, 
\end{align}
where  $\overline{\Gamma}$ denotes $\Gamma/(\Gamma\cap F^{\times})$ for a congruence subgroup $\Gamma$ of $G(\Q)_+$. 
Then, by the strong approximation theorem, we have the following description of $Y(\ideal{n})$:
\begin{align}\label{adele var}
Y(\ideal{n})\simeq \coprod_{1\le i \le h_F^+} Y_i
\end{align}
given by sending the class of $x_i g\in Y(\ideal{n})$ to the class of $g \underline{\mathbf{i}}\in Y_i$ for $g\in G_{\infty,+}$.

%
\subsection{Analytic Hilbert modular forms}\label{Analytic HMF}
%

In this subsection, we fix notation concerning the spaces of Hilbert modular forms, following \cite[\S 1.2]{Dim}, \cite[\S 2]{Shi}. 

Let $k$ be an integer $\ge 2$ and $\ideal{n}$ a non-zero ideal of $\ideal{o}_F$. 
Let $t=\sum_{\iota\in J_F}\iota \in \Z[J_F]$.
For each subset $J$ of $J_F$, let $M_{k,J}(\ideal{n},\C)$ (resp. $S_{k,J}(\ideal{n},\C)$) denote the $\C$-vector space $G_{kt,J}(K_1(\ideal{n}))$ (resp. $S_{kt,J}(K_1(\ideal{n}))$) of Hilbert modular (resp. cusp) forms of weight $kt$, level $K_1(\ideal{n})$, and type $J$ at infinity defined in \cite[Definition 1.2]{Dim}. 
Let $\chi$ be a Hecke character of $F$ of type $-(k-2)t$ whose conductor divides $\mathfrak{n}$. 
Let $M_{k,J}(\ideal{n},\chi,\C)$ (resp. $S_{k,J}(\ideal{n},\chi,\C)$) denote the subspace $G_{kt,J}(K_1(\ideal{n}),\chi)$ (resp. $S_{kt,J}(K_1(\ideal{n}),\chi)$) of $G_{kt,J}(K_1(\ideal{n}))$ (resp. $S_{kt,J}(K_1(\ideal{n}))$) of elements with character $\chi$ defined in \cite[Definition 1.3]{Dim}. 
If $J=J_F$, then we simply write 
$M_{k}(\ideal{n},\C)$ and $M_{k}(\ideal{n},\chi,\C)$ 
(resp. $S_{k}(\ideal{n},\C)$ and $S_{k}(\ideal{n},\chi,\C)$)
for $M_{k,J_F}(\ideal{n},\C)$ and $M_{k,J_F}(\ideal{n},\chi,\C)$
(resp. $S_{k,J_F}(\ideal{n},\C)$ and $S_{k,J_F}(\ideal{n},\chi,\C)$), respectively.
We note that the spaces $M_{k}(\ideal{n},\C)$ and $M_{k}(\ideal{n},\chi,\C)$ are embedded into the space of Hilbert modular forms defined in \cite[\S 2]{Shi} (see, for example, \cite[\S 5.5]{Ge--Go}).

For a fractional ideal $\ideal{a}$ of $F$, 
let $M_{k}(\Gamma_{1}(\ideal{a},\mathfrak{n}),\C)$ (resp. $S_{k}(\Gamma_{1}(\ideal{a},\mathfrak{n}),\C)$) denote the space $G_{kt,J_F}(\Gamma_{1}(\ideal{a},\mathfrak{n});\C)$ (resp. $S_{kt,J_F}(\Gamma_{1}(\ideal{a},\mathfrak{n});\C)$) of holomorphic Hilbert modular (resp. cusp) forms of weight $kt$ and of level $\Gamma_{1}(\ideal{a},\mathfrak{n})$ defined in \cite[Definition 1.4]{Dim}. 
Then we have canonical isomorphisms (cf. \cite[p.323]{Hida91} and \cite[(2.6a)]{Hida88}):
\begin{align}\label{isomorphism between spaces of HMF}
M_{k}(\mathfrak{n},\C) \simeq \bigoplus_{1\le i \le h_F^+}M_{k}(\Gamma_{1}(\ideal{d}_F [t_i],\mathfrak{n}),\C),\ \ \ \ \ \ 
S_{k}(\mathfrak{n},\C) \simeq \bigoplus_{1\le i \le h_F^+}S_{k}(\Gamma_{1}(\ideal{d}_F [t_i],\mathfrak{n}),\C).
\end{align}

%
\subsection{Dirichlet series associated to Hilbert modular forms}\label{Diri}
%

In this subsection, we recall the definition and properties of the Dirichlet series associated to Hilbert modular forms, following \cite[\S 2]{Shi}.

Let $\textbf{h} \in M_k(\ideal{n},\C)$ and $h_i\in M_k(\Gamma_1(\ideal{d}_F[t_i],\ideal{n}),\C)$ such that $\textbf{h}=(h_i)_{1\le i \le h_F^+}$ under the isomorphism (\ref{isomorphism between spaces of HMF}). 
Then $h_i$ has the Fourier expansion of the form 
\begin{align}\label{i-th Fourier exp}
h_i(z)
=c_{\infty}([t_i]^{-1},\textbf{h})N([t_i])^{k/2}
+\sum_{0\ll \xi \in [t_i]}c(\xi[t_i]^{-1},\bold{h})N(\xi)^{k/2}e_F(\xi z)
\end{align}
given by \cite[(2.18)]{Shi} and \cite[Proposition 4.1]{Hida88}. 
Here the notion $\gg 0$ means totally positive, $\ideal{m} \mapsto c(\ideal{m},\textbf{h})$ is a function on the set of all fractional ideals of $F$ vanishing outside the set of integral ideals, and 
$e_F$ denotes the additive character of $F\backslash \A_F$ characterized by $e_F(x_{\infty})=\exp(2\pi \sqrt{-1}x_{\infty})$ for $x_{\infty}\in \A_{F,\infty}$.
We put 
\[
\text{$a_{\infty}(0,h_i)=c_{\infty}([t_i]^{-1},\textbf{h})N([t_i])^{k/2}$ and $a_{\infty}(\xi,h_i)=c(\xi[t_i]^{-1},\bold{h})N(\xi)^{k/2}$}
\]
for any $0\ll \xi\in [t_i]$. 
We also put 
\begin{align}
\label{Hecke constant}
C_{\infty,i}(0,\textbf{h})&=N([t_i])^{-k/2}a_{\infty}(0,h_i),\\
\label{Hecke eigenvalue}
C(\mathfrak{m},\textbf{h})&=N(\ideal{m})^{k/2}c(\mathfrak{m},\textbf{h})
\end{align}
for all non-zero ideals $\ideal{m}$ of $\ideal{o}_F$. 
Let $\eta$ be a narrow ray class character of $F$.
The Dirichlet series in the sense of G.\,Shimura \cite[(2.25)]{Shi} is defined by 
\begin{align}\label{L-function}
\sum_{\mathfrak{m}}C(\mathfrak{m},\textbf{h})
\eta(\mathfrak{m})\mathrm{N}(\mathfrak{m})^{-s} \ \ \text{for}\ \ s\in \C,
\end{align}
where $\mathfrak{m}$ runs over the set of all non-zero ideals of $\ideal{o}_F$. 
It converges absolutely if $\mathrm{Re}(s)$ is sufficiently large and extends to a meromorphic function on the complex plane (see, for example, \cite[\S 2.3]{Hira}). 
For each $\textbf{h}\in M_{k}(\ideal{n},\C)$, let $D(s,\textbf{h},\eta)$ denote this meromorphic function. 
If $\eta$ is the trivial character, we simply write $D(s,\textbf{h})$ for $D(s,\textbf{h},\eta)$.

%
\subsection{Hecke operators on analytic Hilbert modular forms}\label{Hecke}
%

Let $\ideal{n}$ be a non-zero ideal of $\ideal{o}_F$. 
In this subsection, we recall the definition of the Hecke operators acting on $M_2(\mathfrak{n},\C)$ and $S_2(\mathfrak{n},\C)$, following \cite[\S 1.10]{Dim}. 

Let $\Delta(\mathfrak{n})$ be the following semigroup: 
\begin{align*}
\Delta(\mathfrak{n})&=
G(\A_f)\cap 
\left\{\begin{pmatrix}a&b\\c&d\\\end{pmatrix} \in M_2(\widehat{\mathfrak{o}}_F) \bigg\vert c \in \mathfrak{n}\widehat{\mathfrak{o}}_F,\  d_{\ideal{q}} \in \ideal{o}_{F_{\ideal{q}}}^{\times}\ \,\text{whenever}\ \ideal{q} \vert \mathfrak{n}  \right\},
\end{align*}
where $\ideal{q}$ is a non-zero prime ideal of $\ideal{o}_F$.
For $y \in \Delta(\mathfrak{n})$, we define the action of the double coset $K_1(\mathfrak{n}) y K_1(\mathfrak{n})$ on $M_2(\ideal{n},\C)$ (resp. $S_2(\ideal{n},\C)$) by 
\begin{align}\label{Hecke op analytic}
\textbf{f}\vert [K_1(\mathfrak{n}) y K_1(\mathfrak{n})](x)=\sum_{i} \textbf{f}(xy_i^{-1}),
\end{align}
where $K_1(\mathfrak{n}) y K_1(\mathfrak{n})=\coprod_i K_1(\mathfrak{n}) y_i$. 
By the definition of $M_2(\ideal{n},\C)$ and $S_2(\ideal{n},\C)$, the right\nobreakdash-hand side is independent of  the choice of the representative set $\{y_i\}_i$. 

We define the Hecke operator $T(\varpi_{\ideal{q}}^e)$ (resp. $S(\varpi_{\ideal{q}}^e)$) for a non-negative integer $e$, a non-zero prime ideal $\ideal{q}$ of $\mathfrak{o}_F$ (resp. prime ideal $\ideal{q}$ of $\mathfrak{o}_F$ prime to $\ideal{n}$), and a uniformizer $\varpi_{\ideal{q}}$ of $\mathfrak{o}_{F_{{}_{\ideal{q}}}}$ by the action of the double coset
$K_1(\ideal{n})\begin{pmatrix}\varpi_{\ideal{q}}^e&0\\0&1\\\end{pmatrix} K_1(\ideal{n})$ (resp.
$K_1(\ideal{n})\begin{pmatrix}\varpi_{\ideal{q}}^e&0\\0&\varpi_{\ideal{q}}^e\end{pmatrix} K_1(\ideal{n})$). 
We note that these operators are independent of the choice of $\varpi_{\ideal{q}}$.
We put $T(\ideal{q}^e)=T(\varpi_{\ideal{q}}^e)$ and $S(\ideal{q}^e)=S(\varpi_{\ideal{q}}^e)$ (resp. $U(\ideal{q}^e)=T(\varpi_{\ideal{q}}^e)$) for a non-negative integer $e$ and a non-zero prime ideal $\ideal{q}$ prime to $\ideal{n}$ (resp. prime ideal $\ideal{q}$ dividing $\ideal{n}$). 
We define 
$T(\ideal{m})=\prod_{\ideal{q}\nmid\ideal{n}}T(\ideal{q}^{e(\ideal{q})})$ and 
$S(\ideal{m})=\prod_{\ideal{q}\nmid\ideal{n}}S(\ideal{q}^{e(\ideal{q})})$ for all non-zero ideal 
$\ideal{m}=\prod_{\ideal{q}\nmid\ideal{n}}\ideal{q}^{e(\ideal{q})}$ of $\ideal{o}_F$ prime to $\ideal{n}$ and 
$U(\ideal{m})=\prod_{\ideal{q}|\ideal{n}}U(\ideal{q}^{e(\ideal{q})})$ for all non-zero ideal $\ideal{m}=\prod_{\ideal{q}|\ideal{n}}\ideal{q}^{e(\ideal{q})}$ of $\ideal{o}_F$, where $\ideal{q}$ is a non-zero prime ideal. 

The definition of the Hecke operators acting on $M_2(\Gamma_1(\ideal{a},\ideal{n}),\C)$ and $S_2(\Gamma_1(\ideal{a},\ideal{n}),\C)$ and their relation (via (\ref{isomorphism between spaces of HMF})) to the adelic ones recalled above are explicitly given in \cite[\S 2]{Shi}.

By \cite[(2.23)]{Shi}, we have a relation between the Hecke operators and the Fourier expansion of the following form: 
for $\mathbf{f} \in M_2(\ideal{n},\C)$ and $V(\ideal{m}')=T(\ideal{m}')$ or $U(\ideal{m}')$, 
\begin{align}\label{Hecke and Fourier}
C(\ideal{m},\textbf{f}\vert V(\ideal{m}'))=\sum_{\ideal{m}+\ideal{m}'\subset \ideal{c}}N(\ideal{c})C(\ideal{c}^{-2}\ideal{m}\ideal{m}',\textbf{f}|S(\ideal{c})). 
\end{align}

For a subalgebra $A$ of $\C$, let 
$\mathbb{H}_2(\mathfrak{n},A)$ (resp. $\mathscr{H}_2(\mathfrak{n},A)$) be the commutative $A$-subalgebra of $\End_{\C}(M_2(\ideal{n},\C))$ (resp. $\End_{\C}(S_2(\ideal{n},\C))$) generated by $T(\ideal{m})$, $S(\ideal{m})$ for all non-zero ideal $\ideal{m}=\prod_{\ideal{q}\nmid\ideal{n}}\ideal{q}^{e(\ideal{q})}$ of $\ideal{o}_F$ prime to $\ideal{n}$ and $U(\ideal{m})$ for all non-zero ideal $\ideal{m}=\prod_{\ideal{q}|\ideal{n}}\ideal{q}^{e(\ideal{q})}$ of $\ideal{o}_F$.

Let $\widetilde{F}$ be the Galois closure of $F$ in $\overline{\Q}$, and let $F'$ be the field generated by elements $\varepsilon^{t/2}$ for all $\varepsilon\in\ideal{o}_{F,+}^{\times}$ over $\widetilde{F}$. 
Let $\Phi_p$ be the composite field of $\iota_p((F')^{\iota}(\sqrt{-1}))$ in $\overline{\Q}_p$ for all $\iota\in J_F$, where $\iota_p: \overline{\Q} \to \overline{\Q}_p$ is the fixed embedding. 
Let $\integer{}$ be the ring of integers of a finite extension of $\Phi_p$.
For $\textbf{h}=(h_i)_{1\le i \le h_F^+} \in M_2(\ideal{n},\C)$, $h_i\in M_2(\Gamma_1(\ideal{d}_F[t_i],\ideal{n}),\C)$ has the Fourier expansion of the form $(\ref{i-th Fourier exp})$. 
For a subalgebra $A$ of $\C$, we put
\begin{align}
\nonumber M_2(\Gamma_1(\ideal{d}_F[t_i],\ideal{n}),A)&=M_2(\Gamma_1(\ideal{d}_F[t_i],\ideal{n}),\C) \cap A[[e_F(\xi z):\text{$\xi=0$ or $0\ll \xi \in F$}]],\\
\nonumber S_2(\Gamma_1(\ideal{d}_F[t_i],\ideal{n}),A)&=S_2(\Gamma_1(\ideal{d}_F[t_i],\ideal{n}),\C) \cap A[[e_F(\xi z):\text{$\xi=0$ or $0\ll \xi \in F$}]],\\
\label{integral MF}
M_2(\ideal{n},A)=\bigoplus_{1\le i \le h_F^+} &M_2(\Gamma_1(\ideal{d}_F[t_i],\ideal{n}),A),\ \ \ 
S_2(\ideal{n},A)=\bigoplus_{1 \le i \le h_F^+} S_2(\Gamma_1(\ideal{d}_F[t_i],\ideal{n}),A).
\end{align}
The space $M_2(\ideal{n},\integer{})$ (resp. $S_2(\ideal{n},\integer{})$) is stable under $\mathbb{H}_2(\ideal{n},\integer{})$ (resp. $\mathscr{H}_2(\ideal{n},\integer{})$) (see, for example, \cite[Theorem 4.11]{Hida88}, \cite[Theorem 2.2 (ii)]{Hida91}).

\section{Selmer groups}\label{section:Selmer}
%

%
\subsection{Definition of Selmer groups}\label{subsection:def of Selmer}
%

In this subsection, we recall the notation and definitions of Selmer groups for a nearly ordinary Galois representation with Selmer weights in the sense of T.\,Weston \cite{We}.

Let $F$ be a finite extension of $\Q$. 
Let $F_{\infty}$ denote the cyclotomic $\Zp$-extension of $F$.
Let $G_F$ denote the absolute Galois group of $F$. 
We put $\Gamma=\Gal(F_{\infty}/F)$. 
Let $K_p$ be a finite extension of $\Qp$ and $V$ a finite dimensional $K_p$\nobreakdash-vector space endowed with a continuous $K_p$\nobreakdash-linear action of $G_F$. 
We put $n=\dim_{K_p}V$. 
For a real place $v$ of $F$, let $d_v^{\pm}(V)$ denote the dimension of the $(\pm 1)$\nobreakdash-eigenspaces of a complex conjugation in $G_{F_v}$ acting on $V$, respectively. 
Let $\cal{O}$ denote the ring of integers of $K_p$.
We choose a $G_F$-stable $\integer{}$-lattice $T$ in $V$.
We put $A=V/T$. 
We call $A$ a torsion quotient of $V$. 
Then $A$ is a discrete $\integer{}$\nobreakdash-module with the action of $G_F$, which is isomorphic to $(K_p/\integer{})^n$ as an $\integer{}$\nobreakdash-module.  
We say that $V$ is a nearly ordinary Galois representation of $G_F$ (\cite[\S 1.2]{We}) if, for every place $v$ of $F$ lying above $p$, 
$V$ has a $G_{F_v}$-stable complete flag of $K_p$-subspaces
\[
0=V_v^0 \varsubsetneq V_v^1 \varsubsetneq \cdots \varsubsetneq V_v^n=V.
\]
A set $\cal{W}$ of Selmer weights for $V$ in the sense of Weston (\cite[\S 1.2]{We}) is a choice of integers $c_v(V)$ for every place $v$ of $F$ lying above $p$ such that $0\le c_v(V) \le n$ and 
\[
\sum_{v|p}[F_v:\Qp]\cdot c_v(V)
=\sum_{\text{$v$\,:\,real places}}d_v^-(V)+\sum_{\text{$v$\,:\,complex places}}n.
\]
For every place $w$ of $F_{\infty}$ lying above $p$, 
let $A_w^{\cal{W}}$ denote the image of $V_v^{n-c_v(V)}$ in $A$ under the canonical morphism $V\twoheadrightarrow A$ with $v$ the restriction of $w$ to $F$:
\[\xymatrix{
0 \ar[r] & V_v^{n-c_v(V)} \ar[d] \ar[r] & V \ar[d] \\
0 \ar[r] & A_w^{\cal{W}} \ar[r] & A=V/T \ar[r] & A/A_w^{\cal{W}} \ar[r] & 0.
}\]
For every place $w$ of $F_{\infty}$, the local Galois cohomology $H_{s,\cal{W}}^1(F_{\infty,w},A)$ is defined by 
\begin{eqnarray*}
H_{s,\cal{W}}^1(F_{\infty,w},A) =
\left\{\begin{array}{ll}
H^1(F_{\infty,w},A)&
\If w \nmid p,  \\[1mm]
\im\left(H^1(F_{\infty,w},A) \rightarrow H^1(I_{F_{\infty,w}},A/A_w^{\cal{W}})\right)& 
\If w|p.   
\end{array} \right.
\end{eqnarray*}
Here $I_{F_{\infty,w}}$ denotes the inertia subgroup of $G_{F_{\infty,w}}$.

Let $\Sigma(F)$ denote the set of all places of $F$. 
Let $\Sigma_p(F)$ denote the set of all places of $F$ lying above $p$, and 
let $\Sigma_{\infty}(F)$ denote the set of all infinite places of $F$. 
We put 
\[
\Ram(A)=\{ v \in \Sigma(F)\mid 
\text{$v\notin \Sigma_p(F) \cup \Sigma_{\infty}(F)$ and the action of $G_{F_{v}}$ on $A$ is ramified}\}. 
\]
We say that a finite subset $\Sigma$ of $\Sigma(F)$ is sufficient large for $A$ (\cite[\S 1.3]{We}) if $\Sigma$ contains $\Sigma_p(F) \cup \Sigma_{\infty}(F) \cup \Ram(A)$. 

Let $A$ be a torsion quotient of a nearly ordinary Galois representation $V$ with Selmer weights $\cal{W}$. 
Let $\Sigma$ be a finite subset of $\Sigma(F)$ such that $\Sigma$ is sufficient large for $A$. 
We define the Selmer group $\Sel_{\cal{W}}(F_{\infty},A)$ of $A$ in the sense of Weston \cite[\S 1.3]{We} by 
\begin{align*}
\Sel_{\cal{W}}(F_{\infty},A)=\ker\left( H^1(F_{\Sigma}/F_{\infty},A)
\rightarrow \prod_{w|v\in \Sigma}H_{s,\cal{W}}^1(F_{\infty,w},A) \right).
\end{align*}
Here $F_{\Sigma}$ denotes the maximal extension of $F$ which is unramified outside $\Sigma$. 

Next we define the non-primitive Selmer groups in the sense of R.\,Greenberg \cite[\S 2]{Gre--Vat}.
Let $\Sigma_0$ be a subset of $\Sigma\backslash \{\Sigma_p(F) \cup \Sigma_{\infty}(F)\}$. 
We define the non-primitive Selmer group $\Sel_{\cal{W}}^{\Sigma_0}(F_{\infty},A)$ of $A$ and $\Sigma_0$ by
\begin{align*}
\Sel_{\cal{W}}^{\Sigma_0}(F_{\infty},A)=\ker\left( H^1(F_{\Sigma}/F_{\infty},A)
\rightarrow \prod_{w|v\in \Sigma\backslash \Sigma_0} H_{s,\cal{W}}^1(F_{\infty,w},A) \right).
\end{align*}
We have $\Sel_{\cal{W}}(F_{\infty},A) \subset \Sel_{\cal{W}}^{\Sigma_0}(F_{\infty},A)$ by the definition.

Let $\Lambda$ denote the Iwasawa algebra $\integer{}[[\Gamma]]$. 
We know that the groups $H^1(F_{\Sigma}/F_{\infty},A)$, $H^2(F_{\Sigma}/F_{\infty},A)$, 
$\prod_{w|v} H_{s,\cal{W}}^1(F_{\infty,w},A)$, and $\Sel_{\cal{W}}(F_{\infty},A)$ are discrete 
$\integer{}$\nobreakdash-modules with a natural continuous action of $\Gamma$. 
Hence these groups are regarded as $\Lambda$\nobreakdash-modules and 
are known to be cofinitely generated, 
that is, their Pontryagin duals are finitely generated $\Lambda$\nobreakdash-modules. 
The following is conjectured by \cite[Conjecture 1.7]{We}:
\begin{con}\label{conj:cotorsion}
For every nearly ordinary Galois representation $V$ with Selmer weights $\cal{W}$, the Selmer group $\Sel_{\cal{W}}(F_{\infty},A)$ of a torsion quotient $A$ of $V$ is finitely generated $\Lambda$\nobreakdash-cotorsion. 
\end{con}

In order to confirm Conjecture \ref{conj:cotorsion} in some special cases, we will need the Selmer groups and the non-primitive Selmer groups of the residual Galois representation $A[\varpi]$ defined as follows. 
Let $\varpi$ be a uniformizer of $\integer{}$. Let $A[\varpi]$ denote the $\varpi$-torsion of $A$. 
For every place $w$ of $F_{\infty}$, the local Galois cohomology $H_{s,\cal{W}}^1(F_{\infty,w},A[\varpi])$ is defined by 
\begin{eqnarray*}
H_{s,\cal{W}}^1(F_{\infty,w},A[\varpi])=
\left\{\begin{array}{ll}
H^1(F_{\infty,w},A[\varpi])&
\If w \nmid p, \\[1mm]
\im\left(H^1(F_{\infty,w},A[\varpi]) \rightarrow 
H^1(I_{F_{\infty,w}},(A/A_w^{\cal{W}})[\varpi])\right)& 
\If w|p.   
\end{array} \right.
\end{eqnarray*}
We define the Selmer group $\Sel_{\cal{W}}(F_{\infty},A[\varpi])$ of $A[\varpi]$ by 
\begin{align*}
\Sel_{\cal{W}}(F_{\infty},A[\varpi])=\ker\left( H^1(F_{\Sigma}/F_{\infty},A[\varpi])
\rightarrow \prod_{w|v\in \Sigma}H_{s,\cal{W}}^1(F_{\infty,w},A[\varpi]) \right).
\end{align*}
We also define the non-primitive Selmer group $\Sel_{\cal{W}}^{\Sigma_0}(F_{\infty},A[\varpi])$ of $A[\varpi]$ and $\Sigma_0$ by
\begin{align*}
\Sel_{\cal{W}}^{\Sigma_0}(F_{\infty},A[\varpi])=\ker\left( H^1(F_{\Sigma}/F_{\infty},A[\varpi])
\rightarrow \prod_{w|v\in \Sigma\backslash \Sigma_0} H_{s,\cal{W}}^1(F_{\infty,w},A[\varpi]) \right).
\end{align*}
We have $\Sel_{\cal{W}}(F_{\infty},A[\varpi]) \subset \Sel_{\cal{W}}^{\Sigma_0}(F_{\infty},A[\varpi])$ by the definition.

%
\subsection{Structure of Selmer groups}\label{subsection:St of Selmer}
%

Let $A$ be a torsion quotient of a nearly ordinary Galois representation $V$ with Selmer weights $\cal{W}$. 
Let $\Sigma$ be a finite subset of $\Sigma(F)$ such that $\Sigma$ is sufficient large for $A$. 
Let $\Sigma_0$ be a subset of $\Sigma\backslash \{\Sigma_p(F) \cup \Sigma_{\infty}(F)\}$ such that $\Sigma_0$ contains $\Ram(A)$.
We put $A^{\ast}=\Hom(T,\mu_{p^{\infty}})$. This is also a discrete 
$\integer{}$-module equipped with the action of $\Gal(F_{\Sigma}/F)$.
For a locally compact $\Z_p$-module $M$, let $M^{\mathrm{PD}}$ denote the Pontryagin dual of $M$. 
\begin{prop}\label{prop:surj}
Assume that $H^0(F_{\infty},A^{\ast}\otimes_{\integer{}} K_p/\integer{})=0$ and $\Sel_{\cal{W}}(F_{\infty},A)$ is finitely generated $\Lambda$-cotorsion. 
Then we have 
\[
\Sel_{\cal{W}}^{\Sigma_0}(F_{\infty},A) /\Sel_{\cal{W}}(F_{\infty},A) \simeq \prod_{w|v \in \Sigma_0}H_{s,\cal{W}}^1(F_{\infty,w},A). 
\]
In particular, $\Sel_{\cal{W}}^{\Sigma_0}(F_{\infty},A)$ is finitely generated $\Lambda$-cotorsion, 
and the following equalities hold$:$ 
\begin{align*}
\corank_{\integer{}}(\Sel_{\cal{W}}^{\Sigma_0}(F_{\infty},A))
&=\corank_{\integer{}}(\Sel_{\cal{W}}(F_{\infty},A))
+\sum_{w|v\in \Sigma_0} \corank_{\integer{}}(H_{s,\cal{W}}^1(F_{\infty,w}, A)),\\
\mu(\Sel_{\cal{W}}^{\Sigma_0}(F_{\infty},A)^{\PD})
&=\mu(\Sel_{\cal{W}}(F_{\infty},A)^{\PD}).
\end{align*}
\end{prop}
\begin{proof}
By our assumptions and \cite[\S 1.4, Proposition 1.8]{We}, the sequence 
\[
0 \rightarrow \Sel_{\cal{W}}(F_{\infty},A)\rightarrow H^1(F_{\Sigma}/F_{\infty},A)\rightarrow \prod_{w|v \in \Sigma}H_{s,\cal{W}}^1(F_{\infty,w},A) \rightarrow 0
\]
is exact. Then our first assertion follows from the snake lemma for 
\begin{align*}
\xymatrix{
0\ar[r]&\Sel_{\cal{W}}(F_{\infty},A)\ar[r] \ar@{^{(}->}[d]&H^1(F_{\Sigma}/F_{\infty},A) \ar[r]\ar@{=}[d] &\displaystyle\prod_{w|v \in \Sigma}H_{s,\cal{W}}^1(F_{\infty,w},A)\ar@{->>}[d]\ar[r] &0 \\
0\ar[r]&\Sel_{\cal{W}}^{\Sigma_0}(F_{\infty},A)\ar[r] & H^1(F_{\Sigma}/F_{\infty},A) \ar[r]& \displaystyle\prod_{w|v \in \Sigma\backslash \Sigma_0}H_{s,\cal{W}}^1(F_{\infty,w},A)\ar[r] &0. \\
}
\end{align*}
We fix a place $v \in \Sigma_0$. 
For every place $w$ of $F_{\infty}$ lying above $v$, 
by \cite[Chapter VII, Theorem 7.1.8 (i)]{NSW}, 
the exact sequence $0\to A[\varpi] \to A \xrightarrow{\times \varpi} A\to 0$ implies that $H^1(F_{\infty,w},A)$ is divisible. 
Moreover, since $w\nmid p$, by \cite[\S 3, Proposition 2]{Gre89}, 
$\prod_{w|v}H_{s,\cal{W}}^1(F_{\infty,w},A)=\prod_{w|v} H^1(F_{\infty,w},A)$ is finitely generated $\Lambda$-cotorsion. 
Now $\prod_{w|v}H_{s,\cal{W}}^1(F_{\infty,w},A)$ is a finitely generated $\integer{}$-comodule and hence we obtain the equalities as desired. 
\end{proof}

\begin{prop}\label{prop:Selmer mod pi}
Let $p$ be an odd prime number. 
Assume that, for every place $w$ of $F_{\infty}$ lying above $p$, $I_{F_{\infty,w}}$ acts trivially on $A/A_w^{\cal{W}}$ and $H^0(F,A[\varpi])=0$. 
Then we have
\[
\Sel_{\cal{W}}^{\Sigma_0}(F_{\infty},A[\varpi])\simeq \Sel_{\cal{W}}^{\Sigma_0}(F_{\infty},A)[\varpi].
\]
\end{prop}
\begin{proof}
The assumption $H^0(F,A[\varpi])=0$ implies that $H^0(F_{\infty},A[\varpi])=0$ because $\Gamma$ is a pro-$p$ group. Then we have $H^0(F_{\infty},A)=0$ and hence the exact sequence $0\to A[\varpi]\to A \xrightarrow{\times \varpi} A \to 0$ induces an isomorphism 
\[
H^1(F_{\Sigma}/F_{\infty},A[\varpi])\simeq H^1(F_{\Sigma}/F_{\infty},A)[\varpi].
\]
We have the following commutative diagram:
\begin{align*}
\xymatrix{
\Sel_{\cal{W}}^{\Sigma_0}(F_{\infty},A[\varpi])\ar@{^{(}->}[r] \ar@{^{(}->}[d]&\Sel_{\cal{W}}^{\Sigma_0}(F_{\infty},A)[\varpi] \ar@{^{(}->}[r]\ar@{^{(}->}[d] &\Sel_{\cal{W}}^{\Sigma_0}(F_{\infty},A)\ar@{^{(}->}[d] \\
H^1(F_{\Sigma}/F_{\infty},A[\varpi])\ar[r]^{\simeq}\ar[d] & H^1(F_{\Sigma}/F_{\infty},A)[\varpi] \ar@{^{(}->}[r]& H^1(F_{\Sigma}/F_{\infty},A) \ar[d]\\
\displaystyle \prod_{w|v\in \Sigma\backslash \Sigma_0} H_{s,\cal{W}}^1(F_{\infty,w},A[\varpi]) \ar[rr]^{\bigstar}& & \displaystyle \prod_{w|v\in \Sigma\backslash \Sigma_0} H_{s,\cal{W}}^1(F_{\infty,w},A). 
}
\end{align*}
Thus, for the proof, it suffices to show the injectivity of the morphism $\bigstar$. 
In order to do it, we consider the following commutative diagram:
\begin{align*}
\xymatrix{
\displaystyle \prod_{w|v\in \Sigma\backslash \Sigma_0} H_{s,\cal{W}}^1(F_{\infty,w},A[\varpi])\ar@{^{(}->}[r] \ar[d]^{\bigstar} & \displaystyle \prod_{w|v\nmid p\in \Sigma\backslash \Sigma_0} H^1(I_{F_{\infty,w}},A[\varpi])\times \prod_{w|v\in \Sigma_p(F)} H^1(I_{F_{\infty,w}},D[\varpi]) \ar[d]^{\bigstar\bigstar}\\
\displaystyle \prod_{w|v\in \Sigma\backslash \Sigma_0} H_{s,\cal{W}}^1(F_{\infty,w},A) \ar@{^{(}->}[r] & \displaystyle \prod_{w|v\nmid p\in \Sigma\backslash \Sigma_0} H^1(I_{F_{\infty,w}},A)\times \prod_{w|v\in \Sigma_p(F)} H^1(I_{F_{\infty,w}},D),
}
\end{align*}
where we write $D$ for $A/A_w^{\cal{W}}$ to simplify the notation. 
Here the injectivity of the horizontal morphisms follows from that, for every place $w$ of $F_{\infty}$ such that $w\nmid p$, $G_{F_{\infty,w}}/I_{F_{\infty,w}}$ has profinite order prime to $p$.
Thus, it suffices to show the injectivity of the morphism $\bigstar\bigstar$. 
For a place $w\mid v\in \Sigma\backslash \Sigma_0$ such that $w\nmid p$, by the definition of $\Sigma_0$, we have $H^0(I_{F_{\infty,w}},A)=A$, which is divisible, and hence the exact sequence $0\to A[\varpi]\to A \xrightarrow{\times \varpi} A \to 0$ implies that 
\[
\text{$H^1(I_{F_{\infty,w}},A[\varpi]) \to H^1(I_{F_{\infty,w}},A)$ is injective.}
\]
For a place $w|v\in \Sigma_p(F)$, by the assumption on $D$, we have $H^0(I_{F_{\infty,w}},D)=D$, which is divisible, and hence the exact sequence $0\to D[\varpi]\to D \xrightarrow{\times \varpi} D \to 0$ implies that 
\[
\text{$H^1(I_{F_{\infty,w}},D[\varpi]) \to H^1(I_{F_{\infty,w}},D)$ is injective}
\]
as desired. 
\end{proof}

\begin{prop}\label{prop:Selmer f.g}
Assume that $\Sel_{\cal{W}}^{\Sigma_0}(F_{\infty},A[\varpi])$ is finite. 
Then, under the same assumptions as Proposition \ref{prop:Selmer mod pi},
$\Sel_{\cal{W}}^{\Sigma_0}(F_{\infty},A)$ and $\Sel_{\cal{W}}(F_{\infty},A)$ are finitely generated $\Lambda$\nobreakdash-cotorsion. 
\end{prop}
\begin{proof}
By our assumptions and Proposition \ref{prop:Selmer mod pi}, $\Sel_{\cal{W}}^{\Sigma_0}(F_{\infty},A)[\varpi]$ is finite. 
Since we have $(\Sel_{\cal{W}}^{\Sigma_0}(F_{\infty},A)[\varpi])^{\PD}\simeq \Sel_{\cal{W}}^{\Sigma_0}(F_{\infty},A)^{\PD}/\varpi$, 
which is finite, the structure theorem of finitely generated $\Lambda$-modules (see, for example, \cite[Chapter 13, Theorem 13.12]{Was}) implies that $\Sel_{\cal{W}}^{\Sigma_0}(F_{\infty},A)$ is finitely generated $\Lambda$-cotorsion. 
Since $\Sel_{\cal{W}}(F_{\infty},A)\subset \Sel_{\cal{W}}^{\Sigma_0}(F_{\infty},A)$, we have $\Sel_{\cal{W}}^{\Sigma_0}(F_{\infty},A)^{\PD}\twoheadrightarrow \Sel_{\cal{W}}(F_{\infty},A)^{\PD}$ and hence $\Sel_{\cal{W}}(F_{\infty},A)$ is also finitely generated $\Lambda$\nobreakdash-cotorsion. 
\end{proof}

Let $\kappa$ denote the residue field of $\integer{}$. 
\begin{prop}\label{prop:corank=dim}
Assume that $H^0(F_{\infty},A^{\ast}\otimes_{\integer{}} K_p/\integer{})=0$.
Then, under the same assumptions as Proposition \ref{prop:Selmer mod pi} and Proposition \ref{prop:Selmer f.g}, we have
\[
\corank_{\integer{}}(\Sel_{\cal{W}}^{\Sigma_0}(F_{\infty},A))=\dim_{\kappa}(\Sel_{\cal{W}}^{\Sigma_0}(F_{\infty},A[\varpi])).
\]
\end{prop}
\begin{proof}
By Proposition \ref{prop:Selmer f.g}, $\Sel_{\cal{W}}^{\Sigma_0}(F_{\infty},A)^{\mathrm{PD}}$ is a finitely generated $\integer{}$-module and hence, by Proposition \ref{prop:Selmer mod pi}, it suffices to show that $\Sel_{\cal{W}}^{\Sigma_0}(F_{\infty},A)$ has no proper $\Lambda$-submodules of finite index. 
By Proposition \ref{prop:Selmer f.g}, $\Sel_{\cal{W}}(F_{\infty},A)$ is finitely generated $\Lambda$\nobreakdash-cotorsion and hence Proposition \ref{prop:surj} implies that 
$\Sel_{\cal{W}}^{\Sigma_0}(F_{\infty},A)/\Sel_{\cal{W}}(F_{\infty},A) \simeq \prod_{w|v \in \Sigma_0}H_{s,\cal{W}}^1(F_{\infty,w},A)$,
which is divisible (by the proof of Proposition \ref{prop:surj}). 
Now the assertion follows from the fact obtained by \cite[Proposition 1.8 (2)]{We} that $\Sel_{\cal{W}}(F_{\infty},A)$ has no proper $\Lambda$\nobreakdash-submodule of finite index.
\end{proof}

We can compare $\Sel_{\cal{W}}^{\Sigma_0}(F_{\infty},A)$ with $\Sel_{\cal{W}}(F_{\infty},A)$ by the following proposition:
\begin{prop}\label{prop:generator}
Let $v$ be a finite place of $F$ prime to $p$, and let $P_v(X)=\det((1-\Frob_v X) \vert_{V_{I_{F_v}}})\in \integer{}[X]$. 
Let $\cal{P}_v=P_v(N(\ideal{q}_v)^{-1}\gamma_v) \in \Lambda$, 
where $\ideal{q}_v$ denotes the prime ideal of $\ideal{o}_F$ corresponding to $v$, and $\gamma_v$ denotes the Frobenius automorphism 
corresponding to $v$ in $\Gamma$. 
Then, for each place $w$ of $F_{\infty}$ lying above $v$, the characteristic ideal of the $\Lambda$\nobreakdash-module 
$\prod_{w|v}H_{s,\cal{W}}^1(F_{\infty,w},A)^{\mathrm{PD}}$ is generated by $\cal{P}_v$.
\end{prop}
\begin{proof}
We fix a place $w$ of $F_{\infty}$ lying above $v$. 
Let $\Gamma_v$ denote the decomposition subgroup of $\Gamma$ for $w|v$. 
Since $H_{s,\cal{W}}^1(F_{\infty,w},A)\otimes_{\integer{}[[\Gamma_v]]}\Lambda\simeq \prod_{w'|v}H_{s,\cal{W}}^1(F_{\infty,w'},A)$, it suffices to show that 
the characteristic ideal of the $\integer{}[[\Gamma_v]]$\nobreakdash-module 
$H_{s,\cal{W}}^1(F_{\infty,w},A)^{\mathrm{PD}}$ is generated by $\cal{P}_v$.
Let $R_{F_v}$ denote the ramification subgroup of $G_{F_v}$, which has profinite order prime to $p$. 
Since $I_{F_v}/R_{F_v}\simeq \prod_{(\ell,\ideal{q}_v)=1}\mathbb{Z}_{\ell}(1)$ 
(see, for example, \cite[Chapter IV, \S 2, Exercise 2]{SE}),
there exists a unique subgroup $J_{F_v}$ of $I_{F_v}$ such that 
$J_{F_v}$ has profinite order prime to $p$ and $I_{F_v}/J_{F_v}\simeq \Zp(1)$. 
We put $\overline{I}_{F_v}=I_{F_v}/J_{F_v}$ and $G=G_{F_{\infty,w}}/J_{F_v}$. 
Therefore we obtain 
\begin{align*}
H_{s,\cal{W}}^1(F_{\infty,w},A) &\simeq H^1(G,A^{J_{F_v}})\\
\nonumber &\simeq H^1(G,A_{J_{F_v}})\\
\nonumber &\simeq H^1(\overline{I}_{F_v},A_{J_{F_v}})^{G/\overline{I}_{F_v}}.
\end{align*}
Here the first isomorphism is obtained from the inflation-restriction sequence for $1\to J_{F_v} \to G_{F_{\infty,w}} \to G \to 1$ because $J_{F_v}$ has profinite order prime to $p$, 
the second isomorphism is obtained from that the canonical morphism $A^{J_{F_v}} \hookrightarrow A \twoheadrightarrow A_{J_{F_v}}$ induces an isomorphism $A^{J_{F_v}} \simeq A_{J_{F_v}}$ because $J_{F_v}$ has profinite order prime to $p$, 
and the last isomorphism is obtained from the inflation-restriction sequence for $1\to \overline{I}_{F_v} \to G \to G/\overline{I}_{F_v}\to 1$ because 
$G/\overline{I}_{F_v}$ has profinite order prime to $p$. 
Note that 
\[
H^1(\overline{I}_{F_v},A_{J_{F_v}})\simeq A_{I_{F_v}}(-1)
\]
as $G_{F_v}/I_{F_v}$-modules. 
Indeed, if $\epsilon_v$ is a topological generator of $\overline{I}_{F_v}\simeq \Zp(1)$, then we have $H^1(\overline{I}_{F_v},A_{J_{F_v}}) \simeq H^1(\overline{I}_{F_v},A_{J_{F_v}}/(\epsilon_v-1)A_{J_{F_v}})\simeq \Hom (\Zp(1),A_{I_{F_v}}) \simeq A_{I_{F_v}}(-1)$.
Let $\alpha_1,\cdots,\alpha_e$ denote the eigenvalues of $\Frob_v\in G_{F_v}/I_{F_v}$ acting on $A_{I_{F_v}}$. 
Then the eigenvalues of $\Frob_v$ acting on $(A_{I_{F_v}}(-1))^{\mathrm{PD}}$ are $N(\ideal{q}_v)\alpha_1^{-1},\cdots,N(\ideal{q}_v)\alpha_e^{-1}$ and hence, again using the fact that $G/\overline{I}_{F_v}$ has profinite order prime to $p$, those eigenvalues $N(\ideal{q}_v)\alpha_i^{-1}$ which are principal units are the eigenvalues of $\gamma_v$ acting on $(A_{I_{F_v}}(-1)^{G/\overline{I}_{F_v}})^{\mathrm{PD}}$. 
Therefore, our assertion follows from that 
$1-N(\ideal{q}_v)\alpha_i^{-1}\gamma_v\in \integer{}[[\Gamma_v]]$ is a unit if and only if $N(\ideal{q}_v)\alpha_i^{-1}$ is a principal unit. 
\end{proof}

%
\subsection{Selmer groups for $1$-dimensional representations}\label{subsection:Selmer for 1}
%

In this subsection, we compute the Selmer groups 
for $1$-dimensional representations under some assumptions. 

We assume that $F$ is a totally real number field. 
Let $\Sigma$ be a finite subset of $\Sigma(F)$ such that 
$\Sigma $ contains $\Sigma_p(F)\cup \Sigma_{\infty}(F)$. 
Let $\theta:\Gal(F_{\Sigma}/F)\to \integer{}^{\times}$ be a continuous homomorphism. 
Let $V_{\theta}$ denote the space $\theta\otimes_{\Zp}\Qp$ with the action of $\Gal(F_{\Sigma}/F)$ via $\theta$. 
Let $A_{\theta}$ denote the cofree $\integer{}$-module $\theta\otimes_{\Zp}(\Qp/\Zp)$ of corank $1$ with the action of $\Gal(F_{\Sigma}/F)$ via $\theta$. 
Let $K$ denote the extension of $F$ such that $\ker(\theta)=\Gal(F_{\Sigma}/K)$, and let $K_{\infty}$ denote the cyclotomic $\Zp$-extension of $K$.
We put $G=\Gal(K_{\infty}/F)$ and $\Delta=\Gal(K_{\infty}/F_{\infty})$. 
We assume that 
\[
(p,\sharp\Delta)=1.
\]
Then we can identify $\Gamma$ with a subgroup of $G$ such that 
$G=\Delta \times \Gamma$. 
Moreover, the restriction map induces an isomorphism:
\[
H^1(F_{\Sigma}/F_{\infty},A_{\theta}) \simeq H^1(F_{\Sigma}/K_{\infty},A_{\theta})^{\Delta}.
\]
Since $\Gal(F_{\Sigma}/K_{\infty})$ acts trivially on $A_{\theta}$, we have 
\[
H^1(F_{\Sigma}/K_{\infty},A_{\theta})=\Hom(\Gal(M_{\infty}^{\Sigma}/K_{\infty}),A_{\theta}),
\]
where $M_{\infty}^{\Sigma}$ denotes the maximal abelian pro-$p$ extension of $K_{\infty}$ unramified outside $\Sigma$. 
Let $M_{\infty}^{p,\infty}$ denote the maximal abelian pro-$p$ extension of $K_{\infty}$ unramified outside $\Sigma_p(F)\cup \Sigma_{\infty}(F)$. 
Let $L_{\infty}$ denote the maximal abelian pro-$p$ extension of $K_{\infty}$ unramified everywhere. 
We put $X_{\infty}=\Gal(M_{\infty}^{p,\infty}/K_{\infty})$ and $Y_{\infty}=\Gal(L_{\infty}/K_{\infty})$. 

If $\theta$ is totally even (resp. totally odd), 
then, for every place $v\in \Sigma_{\infty}(F)$, we have $d_v^{-}(V_{\theta})=0$ (resp. $d_v^{-}(V_{\theta})=1$). 
Then $V_{\theta}$ is a nearly ordinary Galois representation with Selmer weights $c_v(V_{\theta})=0$ (resp. $c_v(V_{\theta})=1$) for every place $v\in \Sigma_p(F)$. 
Now, for $B=A_{\theta}$ and $A_{\theta}[\varpi]$, 
we can define the Selmer groups $\Sel(F_{\infty},B)$ and the non-primitive Selmer groups $\Sel^{\Sigma_0}(F_{\infty},B)$ as \S \ref{subsection:def of Selmer}.

\begin{prop}\label{prop:Selmer for 1}
Let $\xi$ be the restriction $\theta\vert_{\Delta}$ of $\theta$ to $\Delta$ and $\Sigma_0=\Sigma \backslash \{ \Sigma_p(F)\cup \Sigma_{\infty}(F)\}$. 
Assume that $p$ is odd and $(p, \sharp\Delta)=1$. 
\begin{enumerate}[$(1)$]
\item \label{prop:Selmer for 1 str}
\begin{eqnarray*}
\Sel(F_{\infty},A_{\theta})\simeq
\left\{ \begin{array}{ll}
\Hom_{\integer{}}((X_{\infty}\otimes_{\Zp}\integer{})^{\xi}, A_{\theta})&
\If \theta\ \mathrm{is \ totally \ even},  \\[1mm]
\Hom_{\integer{}}((Y_{\infty}\otimes_{\Zp}\integer{})^{\xi}, A_{\theta})& 
\If \theta\ \mathrm{is \ totally \ odd}.   \\
\end{array} \right.
\end{eqnarray*}
In particular, $\Sel(F_{\infty},A_{\theta})$ is finitely generated $\Lambda$-cotorsion.

\item \label{prop:Selmer for 1 str Sigma0}
For $B=A_{\theta}$ and $A_{\theta}[\varpi]$, 
\begin{eqnarray*}
\ \ \ \,\Sel^{\Sigma_0}(F_{\infty},B)\simeq
\left\{ \begin{array}{ll}
H^1(F_{\Sigma}/F_{\infty},B)&
\If \theta\ \mathrm{is \ totally \ even},  \\[1mm]
\ker\left(H^1(F_{\Sigma}/F_{\infty},B) \to \displaystyle \prod_{w|v\in \Sigma_p(F)}H^1(I_{F_{\infty},w},B)\right)& 
\If \theta\ \mathrm{is \ totally \ odd}.   \\
\end{array} \right.
\end{eqnarray*}

\item \label{prop:Selmer for 1 vanishing}
Assume that $\xi$ is non-trivial if $\theta$ is totally even, and $\xi\neq\omega$ if $\theta$ is totally odd. 
Then 
\begin{align*}
H^0(F,A_{\theta}^{\ast}\otimes_{\integer{}} K_p/\integer{})=0.
\end{align*}
Moreover, $\Sel^{\Sigma_0}(F_{\infty},A_{\theta})$ is finitely generated $\Lambda$-cotorsion. 
\end{enumerate}
\end{prop}
\begin{proof}
(\ref{prop:Selmer for 1 str}), (\ref{prop:Selmer for 1 str Sigma0}) 
First we treat the case where $\theta$ is totally even. 
Let $\mathbf{0}$ denote the set of the Selmer weights for $V_{\theta}$ as explained before this proposition. 
Then we have $A_{\theta,w}^{\mathbf{0}}=A_{\theta}$ for every place $w$ of $F_{\infty}$ lying above $v\in \Sigma_p(F)$ and hence, for $B=A_{\theta}$ and $A_{\theta}[\varpi]$, we have
\[
H_{s,\mathbf{0}}^1(F_{\infty,w},B)=0,
\]
which proves (\ref{prop:Selmer for 1 str Sigma0}). 
Moreover the assumption $(p,\sharp \Delta)=1$ implies (\ref{prop:Selmer for 1 str}):
\[
\Sel(F_{\infty}, A_{\theta})
=\Hom(X_{\infty},A_{\theta})^{\Delta}=\Hom_{\integer{}}((X_{\infty}\otimes_{\Zp}\integer{})^{\xi},A_{\theta}).
\]

Next we treat the case where $\theta$ is totally odd. 
Let $\mathbf{1}$ denote the set of the Selmer weights for $V_{\theta}$ as explained before this proposition. 
Then we have 
$A_{\theta,w}^{\mathbf{1}}=0$ for every place $w$ of $F_{\infty}$ lying above $v\in \Sigma_p(F)$ and hence, for $B=A_{\theta}$ and $A_{\theta}[\varpi]$, we have
\begin{align*}
H_{s,\mathbf{1}}^1(F_{\infty,w},B)&=\im\left( H^1(F_{\infty,w},B)\to H^1(I_{F_{\infty,w}},B)\right),
\end{align*}
which proves (\ref{prop:Selmer for 1 str Sigma0}). 
Moreover the assumption $(p,\sharp \Delta)=1$ implies (\ref{prop:Selmer for 1 str}): 
\[
\Sel(F_{\infty}, A_{\theta})
=\Hom(Y_{\infty},A_{\theta})^{\Delta}=\Hom_{\integer{}}((Y_{\infty}\otimes_{\Zp}\integer{})^{\xi},A_{\theta}).
\]

According to a well-known theorem of K.\,Iwasawa (\cite{Iwa59}, \cite{Iwa73}), $X_{\infty}$ and $Y_{\infty}$ are finitely generated $\Lambda$-torsion, and hence $\Sel(F_{\infty},A_{\theta})$ is finitely generated $\Lambda$-cotorsion as desired. 

(\ref{prop:Selmer for 1 vanishing}) 
We note that the vanishing of $H^0(F,A_{\theta}^{\ast}\otimes_{\integer{}}K_p/\integer{})$ is equivalent to the vanishing of its $\varpi$-torsion $H^0\left(F,\Hom(A_{\theta}[\varpi],\kappa(1))\right)$. 
Thus, the first assertion follows from our assumption on $\xi$. 
Now, by combining with Proposition \ref{prop:surj} and (\ref{prop:Selmer for 1 str}), $\Sel^{\Sigma_0}(F_{\infty},A_{\theta})$ is finitely generated $\Lambda$-cotorsion as desired. 
\end{proof}

%
\subsection{Selmer groups for Hilbert modular forms}\label{subsection:algebraic Iwasawa invariants}
%

In this subsection, we apply the general theory above to the Iwasawa main conjecture of Hilbert modular forms. 

Let $\mathbf{f}\in S_2(\ideal{n},\integer{})$ be a normalized Hecke eigenform for all $T(\ideal{q})$ and $U(\ideal{q})$ with character $\varepsilon$. 
We assume that $\mathbf{f}$ is $p$-ordinary, that is, for every prime ideal $\ideal{p}$ of $\ideal{o}_F$ lying above $p$, the Hecke eigenvalue $C(\ideal{p},\mathbf{f})$ of $T(\ideal{p})$ is prime to $p$. Let 
\[
\rho_{\mathbf{f}} \colon G_{F} \rightarrow \GL(T_{\mathbf{f}})\simeq \GL_2(\integer{})
\]
denote the associated Galois representation, which satisfies the following (1), (2), (3), (4):
\begin{enumerate}[$(1)$]
\item $\rho_{\mathbf{f}}$ is unramified at every prime ideal $\ideal{q}$ of $\ideal{o}_F$ such that $\ideal{q} \nmid\ideal{n} p$;
\item $\Tr(\rho_{\mathbf{f}}(\Frob_{\ideal{q}}))=C(\ideal{q},\mathbf{f})$ for every prime ideal $\ideal{q}$ of $\ideal{o}_F$ such that $\ideal{q} \nmid\ideal{n} p$;
\item $\det(\rho_{\mathbf{f}}(\Frob_{\ideal{q}}))=\varepsilon(\ideal{q})N(\ideal{q})$ for every prime ideal $\ideal{q}$ of $\ideal{o}_F$ such that $\ideal{q} \nmid\ideal{n} p$;
\item $\rho_{\mathbf{f}}$ is totally odd. 
\end{enumerate}

Let $V_{\mathbf{f}}$ denote the space $T_{\mathbf{f}}\otimes_{\Z_p} \Qp$ with the action of $G_{F}$ via $\rho_{\mathbf{f}}$. 
Let $A_{\mathbf{f}}$ denote the cofree $\integer{}$-module $T_{\mathbf{f}}\otimes_{\Z_p}(\Q_p/\Z_p)$ of corank $2$ with the action of $G_{F}$ via $\rho_{\mathbf{f}}$. 

From the result of \cite[Theorem 2.1.4]{Wil88}, for every prime ideal $\ideal{p}$ of $\ideal{o}_F$ lying above $p$, the restriction of $\rho_{\mathbf{f}}$ to the decomposition group $G_{F_{\ideal{p}}}$ is of the form 
\begin{align}\label{rho:decom}
\rho_{\mathbf{f}} \big\vert_{G_{F_{\ideal{p}}}} \sim 
\begin{pmatrix} 
\rho_1&\ast\\ 
0&\rho_2\\
\end{pmatrix},
\end{align}
where 
$\rho_2 \colon G_{F_{\ideal{p}}} \rightarrow \integer{}^{\times}$ is unramified such that $\rho_2$ sends the arithmetic Frobenius to a unit\nobreakdash-root of $X^2-C(\ideal{p},\mathbf{f})X+\varepsilon(\ideal{p})N(\ideal{p})=0$. 
Then, for every place $w$ of $F_{\infty}$ lying above $p$, $A_{\mathbf{f},w}$ is defined by the following exact sequence of $\integer{}[G_{F_{\ideal{p}_v}}]$\nobreakdash-modules: 
\begin{align*}
0 \rightarrow A_{\mathbf{f},w} \rightarrow A_{\mathbf{f}} \rightarrow A_{\mathbf{f}}/A_{\mathbf{f},w} \rightarrow 0,
\end{align*}
where $v$ denotes the restriction of $w$ to $F$ and $\ideal{p}_v$ denotes the prime ideal of $\ideal{o}_F$ corresponding to $v$. 
Here $G_{F_{\ideal{p}_v}}$ acts on $A_{\mathbf{f},w}$ via the character 
$\rho_1 \colon G_{F_{\ideal{p}_v}} \rightarrow \integer{}^{\times}$, 
and on $A_{\mathbf{f}}/A_{\mathbf{f},w}$ via the character 
$\rho_2 \colon G_{F_{\ideal{p}_v}} \rightarrow \integer{}^{\times}$. 
Thus, $V_{\mathbf{f}}$ is a nearly ordinary Galois representation with Selmer weights $c_v(V_{\mathbf{f}})=1$ for every place $v \in \Sigma_p(F)$. 

Let $\Sigma_0$ be the set of all finite places $v$ of $F$ such that $\ideal{q}_v$ divides $\ideal{n}$, where $\ideal{q}_v$ denotes the prime ideal of $\ideal{o}_F$ corresponding to $v$.
We put $\Sigma=\Sigma_0\cup\Sigma_p(F)\cup\Sigma_{\infty}(F)$. 
Now, for $B=A_{\mathbf{f}}$ and $A_{\mathbf{f}}[\varpi]$, we can define the Selmer groups $\Sel(F_{\infty},B)$ and the non\nobreakdash-primitive Selmer groups $\Sel^{\Sigma_0}(F_{\infty},B)$ as \S \ref{subsection:def of Selmer}.

Let $\kappa$ denote the residue field of $\integer{}$. 
We assume that 
\begin{align*}
\mbox{(RR)}\ \,&&\mbox{the residual representation } \bar{\rho}_{\mathbf{f}}: G_{F}\to \GL_2(\kappa)\ \mbox{is reducible and of the form}&& 
\end{align*}
\[
\bar{\rho}_{\mathbf{f}} \sim 
\begin{pmatrix} 
\bar{\varphi}&\ast\\ 
0&\bar{\psi}\\
\end{pmatrix},
\]
that is, there exists an exact sequence of $\kappa[G_{F}]$-modules
\begin{align}\label{rho:bar}
0 \rightarrow \Phi \rightarrow A_{\mathbf{f}}[\varpi] \rightarrow \Psi \rightarrow 0,
\end{align}
where $G_{F}$ acts on $\Phi$ via the character 
$\bar{\varphi} \colon G_{F} \rightarrow \kappa^{\times}$, 
and on $\Psi$ via the character 
$\bar{\psi}\colon G_{F} \rightarrow \kappa^{\times}$. 
Let $(A_{\varphi},\varphi)$ (resp. $(A_{\psi},\psi)$) be an $\integer{}$-module $A_{\varphi}\simeq K_p/\integer{}$ (resp. $A_{\psi}\simeq K_p/\integer{}$) with the action of $G_F$ via the character $\varphi=\chi_{\text{cyc}}\varepsilon \psi^{-1}:G_F \to \kappa^{\times} \hookrightarrow \integer{}^{\times}$ (resp. $\psi:G_F \xrightarrow{\bar{\psi}} \kappa^{\times} \hookrightarrow \integer{}^{\times}$), where $\chi_{\text{cyc}}$ is the $p$-adic cyclotomic character of $G_F$. 
Note that $A_{\varphi}[\varpi]=\Phi$ and $A_{\psi}[\varpi]=\Psi$.

We assume the following parity condition: 
\begin{align*}
\mbox{(Parity)} \ \ \ 
&\varphi \ \mbox{is ramified at every prime ideal $\ideal{p}$ of $\ideal{o}_F$ lying above $p$ and totally even, and}\\
&\psi \ \mbox{is unramified at every prime ideal $\ideal{p}$ of $\ideal{o}_F$ lying above $p$ and totally odd.}
\end{align*}
Hence, for every prime ideal $\ideal{p}$ of $\ideal{o}_F$ lying above $p$, we have $\psi(\Frob_{\ideal{p}})\equiv C(\ideal{p},\mathbf{f})\,(\bmod\,{\varpi})$ by (\ref{rho:decom}).
We also assume that 
\begin{align*}
\mbox{($\mu=0$)}&&\ \ \ \ \ \ \ \ \ \ \ \ \ \ \mu(\Sel(F_{\infty},A_{\varphi})^{\PD})=0\ \text{and}\  \mu(\Sel(F_{\infty},A_{\psi})^{\PD})=0.\ \ \ \ \ \ \ \ \ \ \ \ \ \ \ \ \ \ \ && 
\end{align*}

In order to prove Theorem \ref{thm:Selmer for 1 and 2} below,
we need the following three lemmas.

\begin{lem}\label{lem:H^0=0}
Assume that $p$ is odd. Then we have 
\begin{align*}
&H^0(F,\Phi)=0,\ H^0(F,\Psi)=0,\\
&H^0(F_{\infty},\Phi)=0,\ H^0(F_{\infty},\Psi)=0.
\end{align*}
\end{lem}
\begin{proof}
Since $\Gamma$ is a pro-$p$ group, it suffices to show that 
$H^0(F,\Phi)=0$ and $H^0(F,\Psi)=0$.
The condition (Parity) implies that $H^0 (F,\Phi)=0$ (because $\varphi$ is ramified at every prime ideal $\ideal{p}$ of $\ideal{o}_F$ lying above $p$) and $H^0 (F,\Psi)=0$ (because $\psi$ is totally odd and $p$ is odd) as desired. 
\end{proof}

\begin{lem}\label{lem:A_f=0}
Assume that $p$ is odd. Then we have 
\[
H^0 (F,A_{\mathbf{f}}[\varpi])=0. 
\]
\end{lem}
\begin{proof}
The sequence (\ref{rho:bar}) induces an exact sequence 
\[
0\to H^0 (F,\Phi)\to H^0 (F,A_{\mathbf{f}}[\varpi])\to H^0 (F,\Psi).
\]
Now our assertion follows from Lemma \ref{lem:H^0=0}. 
\end{proof}

\begin{lem}\label{lem:H^2=0}
Assume that $p$ is odd. Then we have 
\[
H^2(F_{\Sigma}/F_{\infty},\Phi)=0.
\]
\end{lem}
\begin{proof}
The exact sequence $0 \rightarrow \Phi \rightarrow A_{\varphi} \xrightarrow{\times\varpi} A_{\varphi} \rightarrow 0$ induces a long exact sequence 
\[
H^1(F_{\Sigma}/F_{\infty},A_{\varphi})\xrightarrow{\times\varpi} H^1(F_{\Sigma}/F_{\infty},A_{\varphi}) 
\rightarrow H^2(F_{\Sigma}/F_{\infty},\Phi) \rightarrow H^2(F_{\Sigma}/F_{\infty},A_{\varphi}).
\]
Thus, for the proof, it is enough to show the following:
\begin{align*}
&\text{(i)} \ H^1(F_{\Sigma}/F_{\infty},A_{\varphi}) \ \text{is divisible;}\\
&\text{(ii)} \ H^2(F_{\Sigma}/F_{\infty},A_{\varphi})=0. 
\end{align*}

First we prove (i). 
Since $\varphi$ is totally even, we have $H^1(F_{\Sigma}/F_{\infty},A_{\varphi})=\Sel^{\Sigma_0}(F_{\infty},A_{\varphi})$ by Proposition \ref{prop:Selmer for 1} (\ref{prop:Selmer for 1 str Sigma0}). 
Since $\varphi$ is non-trivial, Proposition \ref{prop:Selmer for 1} (\ref{prop:Selmer for 1 vanishing}) implies that $\Sel^{\Sigma_0}(F_{\infty},A_{\varphi})$ is finitely generated $\Lambda$-cotorsion, and $\mu(\Sel^{\Sigma_0}(F_{\infty},A_{\varphi})^{\PD})=0$ by Proposition \ref{prop:surj} and the assumption $(\mu=0)$.
Moreover, $\Sel^{\Sigma_0}(F_{\infty},A_{\varphi})$ has no proper $\Lambda$-submodules of finite index by the proof of Proposition \ref{prop:corank=dim} and the assumption $(\mu=0)$. 
Therefore, by combining them, $H^1(F_{\Sigma}/F_{\infty},A_{\varphi})$ is divisible as desired. 

Next we prove (ii). By the above (i), we know $\corank_{\Lambda}(H^1(F_{\Sigma}/F_{\infty},A_{\varphi}))=0$. 
Then, by \cite[\S 4, Proposition 3]{Gre89}, we have $\corank_{\Lambda}(H^2(F_{\Sigma}/F_{\infty},A_{\varphi}))=0$ and hence 
$H^2(F_{\Sigma}/F_{\infty},A_{\varphi})$ is finitely generated $\Lambda$-cotorsion. 
Since $H^2(F_{\Sigma}/F_{\infty},A_{\varphi})$ is $\Lambda$-cofree by \cite[\S 4, Proposition 4]{Gre89}, we obtain $H^2(F_{\Sigma}/F_{\infty},A_{\varphi})=0$ as desired. 
\end{proof}

\begin{thm}\label{thm:Selmer for 1 and 2}
Assume that $p$ is odd. Then, under the above assumptions $\mathrm{(RR)}$, $\mathrm{(Parity)}$, and $(\mu=0)$, 
$\Sel(F_{\infty},A_{\mathbf{f}})$ and $\Sel^{\Sigma_0}(F_{\infty},A_{\mathbf{f}})$ are finitely generated $\Lambda$-cotorsion, and the following equality holds$:$
\begin{align}\label{thm:Selmer for 1 and 2:corank}
\displaystyle
\corank_{\integer{}}(\Sel^{\Sigma_0}(F_{\infty},A_{\mathbf{f}}))
=\corank_{\integer{}}(\Sel^{\Sigma_0}(F_{\infty},A_{\varphi}))+\corank_{\integer{}}(\Sel^{\Sigma_0}(F_{\infty},A_{\psi})).
\end{align}
\end{thm}
\begin{proof}
For the proof of our first assertion, it suffices to check that $\Sel^{\Sigma_0}(F_{\infty},A_{\mathbf{f}}[\varpi])$ is finite by applying Proposition \ref{prop:Selmer f.g} with the help of Lemma \ref{lem:A_f=0}. 
By the exact sequence (\ref{rho:bar}) with the help of Lemma \ref{lem:H^0=0} and Lemma \ref{lem:H^2=0}, we have an exact sequence 
\[
0\rightarrow H^1(F_{\Sigma}/F_{\infty},\Phi) 
\rightarrow H^1(F_{\Sigma}/F_{\infty},A_{\mathbf{f}}[\varpi]) 
\rightarrow H^1(F_{\Sigma}/F_{\infty},\Psi) 
\rightarrow 0. 
\]
Then this sequence induces an exact sequence
\begin{align}\label{key-seq}
0\to \Sel^{\Sigma_0}(F_{\infty},\Phi) \to \Sel^{\Sigma_0}(F_{\infty},A_{\mathbf{f}}[\varpi])\to \Sel^{\Sigma_0}(F_{\infty},\Psi)\to 0. 
\end{align}
Indeed, 
by the definition of $\Sel^{\Sigma_0}(F_{\infty},A_{\mathbf{f}}[\varpi])$ and the condition (Parity), we have 
\[
\Sel^{\Sigma_0}(F_{\infty},A_{\mathbf{f}}[\varpi])=\ker\left(H^1(F_{\Sigma}/F_{\infty},A_{\mathbf{f}}[\varpi]) \rightarrow \prod_{w|v\in \Sigma_p(F)}H^1(I_{F_{\infty,w}},\Psi)\right), 
\]
and by Proposition \ref{prop:Selmer for 1} (\ref{prop:Selmer for 1 str Sigma0}), we have $\Sel^{\Sigma_0}(F_{\infty},\Phi)=H^1(F_{\Sigma}/F_{\infty},\Phi)$ 
and 
$\Sel^{\Sigma_0}(F_{\infty},\Psi)=\ker\left(H^1(F_{\Sigma}/F_{\infty},\Psi) 
\rightarrow \prod_{w|v\in \Sigma_p(F)}H^1(I_{F_{\infty,w}},\Psi)\right)$.
Hence, for the proof, it suffices to check that $\Sel^{\Sigma_0}(F_{\infty},\Phi)$ and $\Sel^{\Sigma_0}(F_{\infty},\Psi)$ are finite. 
Proposition \ref{prop:Selmer for 1} (\ref{prop:Selmer for 1 vanishing}) implies that $\Sel^{\Sigma_0}(F_{\infty},A_{\varphi})$ and $\Sel^{\Sigma_0}(F_{\infty},A_{\psi})$ are finitely generated $\Lambda$\nobreakdash-cotorsion. 
Moreover, by the assumption $(\mu=0)$, Proposition \ref{prop:surj} implies that $\mu(\Sel^{\Sigma_0}(F_{\infty},A_{\varphi})^{\PD})=0$ and $\mu(\Sel^{\Sigma_0}(F_{\infty},A_{\psi})^{\PD})=0$.
Now, by combining them and Proposition \ref{prop:Selmer mod pi}, $\Sel^{\Sigma_0}(F_{\infty},\Phi)$ and $\Sel^{\Sigma_0}(F_{\infty},\Psi)$ are finite as desired. 

Next we prove (\ref{thm:Selmer for 1 and 2:corank}).
Since $H^0(F,A_{\varphi}^{\ast}\otimes_{\integer{}} K_p/\integer{})=0$ and $H^0(F,A_{\psi}^{\ast}\otimes_{\integer{}} K_p/\integer{})=0$ by Proposition \ref{prop:Selmer for 1} (\ref{prop:Selmer for 1 vanishing}), we have 
$H^0(F,A_{\mathbf{f}}^{\ast}\otimes_{\integer{}} K_p/\integer{})=0$. 
Therefore we obtain 
\begin{align*}
\corank_{\integer{}}(\Sel^{\Sigma_0}(F_{\infty},A_{\mathbf{f}}))
&=\dim_{\kappa}(\Sel^{\Sigma_0}(F_{\infty},A_{\mathbf{f}}[\varpi]))\\
\nonumber&=\dim_{\kappa}(\Sel^{\Sigma_0}(F_{\infty},\Phi))
+\dim_{\kappa}(\Sel^{\Sigma_0}(F_{\infty},\Psi))\\
\nonumber&=\corank_{\integer{}}(\Sel^{\Sigma_0}(F_{\infty},A_{\varphi}))
+\corank_{\integer{}}(\Sel^{\Sigma_0}(F_{\infty},A_{\psi}))
\end{align*}
as desired. 
Here the first equality follows from Proposition \ref{prop:corank=dim} with the help of Lemma \ref{lem:A_f=0}, the second equality follows from the exact sequence (\ref{key-seq}), and the last equality follows from Proposition \ref{prop:corank=dim} with the help of Lemma \ref{lem:H^0=0}.
\end{proof}

%
\subsection{Applications to the algebraic Iwasawa invariants}\label{subsection:application for algebraic}
%

In this subsection, we fix $\mathbf{f}\in S_2(\ideal{n},\integer{})$ satisfying the conditions (RR) and (Parity) in \S \ref{subsection:algebraic Iwasawa invariants}. 
Let $\ideal{n}'$ be the least common multiple of $\ideal{n}^2$ and $\ideal{m}_{\varepsilon} \ideal{m}_{\psi}^2$. 
Let $\Sigma_0$ be the set of all finite places $v$ of $F$ such that $\ideal{q}_v$ divides $\ideal{n}'$, where $\ideal{q}_v$ denotes the prime ideal of $\ideal{o}_F$ corresponding to $v$.
We put $\Sigma=\Sigma_0\cup\Sigma_p(F)\cup\Sigma_{\infty}(F)$. 
Let $\chi$ be an $\integer{}$-valued narrow ray class character of $F$, whose conductor is $\ideal{n}'$, such that $\chi$ is of order prime to $p$, totally even, and 
\[
\mu(\Sel(F_{\infty},A_{\varphi\chi})^{\PD})=0, \ \ \mu(\Sel(F_{\infty},A_{\psi\chi})^{\PD})=0.
\]

The assumption (RR) implies that the residual representation $\bar{\rho}_{\mathbf{f}\otimes \chi}$ is reducible and of the form 
\begin{align*}
\bar{\rho}_{\mathbf{f}\otimes\chi} \sim 
\begin{pmatrix} 
\overline{\varphi\chi}&\ast\\ 
0&\overline{\psi\chi}
\end{pmatrix}.
\end{align*}
Then the triple $(\mathbf{f}\otimes \chi,\varphi\chi,\psi\chi)$ satisfies the conditions (RR), (Parity), and $(\mu=0)$ in \S \ref{subsection:algebraic Iwasawa invariants}. 
Hence, by applying Theorem \ref{thm:Selmer for 1 and 2} to $(\mathbf{f}\otimes \chi,\varphi\chi,\psi\chi)$ instead of $(\mathbf{f},\varphi,\psi)$, 
the Selmer group $\Sel(F_{\infty},A_{\mathbf{f}\otimes\chi})$ is finitely generated $\Lambda$-cotorsion. 
Here we note that $\Sel(F_{\infty},A_{\mathbf{f}\otimes\chi})=\Sel^{\Sigma_0}(F_{\infty},A_{\mathbf{f}\otimes\chi})$ by the assumption on $\chi$ (cf. Proposition \ref{prop:generator}).
Then we can define the algebraic Iwasawa $\lambda$-invariant $\lambda_{\mathbf{f}\otimes\chi}^{\mathrm{alg}}$ by 
\begin{align}\label{lambda^alg for f and chi}
\lambda_{\mathbf{f}\otimes\chi}^{\mathrm{alg}}
&=\lambda(\Sel(F_{\infty},A_{\mathbf{f}\otimes\chi})^{\mathrm{PD}})
=\corank_{\integer{}}(\Sel(F_{\infty},A_{\mathbf{f}\otimes\chi})).
\end{align}

By Proposition \ref{prop:Selmer for 1}, we can also define the algebraic Iwasawa $\lambda$-invariants $\lambda_{\varphi\chi,\Sigma_0}$ and $\lambda_{\psi\chi,\Sigma_0}$ by 
\begin{align}
\label{classical lambda for varphichi}
\lambda_{\varphi\chi,\Sigma_0}
&=\lambda(\Sel^{\Sigma_0}(F_{\infty},A_{\varphi\chi})^{\mathrm{PD}})=\corank_{\integer{}}(\Sel^{\Sigma_0}(F_{\infty},A_{\varphi\chi})),\\
\label{classical lambda for psichi}
\lambda_{\psi\chi,\Sigma_0}&=\lambda(\Sel^{\Sigma_0}(F_{\infty},A_{\psi\chi})^{\mathrm{PD}})=\corank_{\integer{}}(\Sel^{\Sigma_0}(F_{\infty},A_{\psi\chi})). 
\end{align}
Therefore, by the equality (\ref{thm:Selmer for 1 and 2:corank}), we obtain 
\begin{align}
\label{alg.lambda}
\lambda_{\mathbf{f}\otimes\chi}^{\mathrm{alg}}
&=\lambda_{\varphi\chi,\Sigma_0}+\lambda_{\psi\chi,\Sigma_0}.
\end{align}

\section{Construction of distributions}\label{section:p-adic L}
%

%
\subsection{Mellin transform}\label{subsection:Mellin transform}
%

In this subsection, we give a Mellin transform for a Hilbert cusp form.

Let $\ideal{n}$ be a non-zero ideal of $\ideal{o}_F$. 
Let $\eta$ be a $\overline{\Q}$-valued narrow ray class character of $F$ with sign $r\in (\Z/2\Z)^{J_F}$, 
whose conductor is denoted by $\ideal{m}_{\eta}$, 
such that $\ideal{m}_{\eta}$ is prime to $\ideal{d}_F[t_i]$ for each $i$.
Let $(\ideal{m}_{\eta}^{-1}\ideal{d}_F^{-1}[t_i]^{-1}/\ideal{d}_F^{-1}[t_i]^{-1})^{\times}$ (resp. $(\ideal{m}_{\eta}^{-1}/\mathfrak{o}_F)^{\times}$) be the subset of $\ideal{m}_{\eta}^{-1}\ideal{d}_F^{-1}[t_i]^{-1}/\ideal{d}_F^{-1}[t_i]^{-1}$ (resp. $\ideal{m}_{\eta}^{-1}/\mathfrak{o}_F$) consisting of elements whose annihilator is $\ideal{m}_{\eta}$. 
We fix a non-canonical isomorphism of $\mathfrak{o}_F$-modules 
$\ideal{m}_{\eta}^{-1}\ideal{d}_F^{-1}[t_i]^{-1}/\ideal{d}_F^{-1}[t_i]^{-1} \simeq \ideal{m}_{\eta}^{-1}/\mathfrak{o}_F \simeq \mathfrak{o}_F/\ideal{m}_{\eta}$ 
and a non-canonical bijection induced from it
$(\ideal{m}_{\eta}^{-1}\ideal{d}_F^{-1}[t_i]^{-1}/\ideal{d}_F^{-1}[t_i]^{-1})^{\times} \simeq (\ideal{m}_{\eta}^{-1}/\mathfrak{o}_F)^{\times}\simeq (\mathfrak{o}_F/\ideal{m}_{\eta})^{\times}$. 
Hence we may canonically identify $(\ideal{m}_{\eta}^{-1}\ideal{d}_F^{-1}[t_i]^{-1}/\ideal{d}_F^{-1}[t_i]^{-1})^{\times}/\mathfrak{o}_{F,+}^{\times}$ with a subgroup of $\text{Cl}_F^+(\ideal{m}_{\eta})$ under the canonical extension 
\begin{align}\label{canonical extension}
1\to (\mathfrak{o}_F/\ideal{m}_{\eta})^{\times}/\mathfrak{o}_{F,+}^{\times} \to \text{Cl}_F^+(\ideal{m}_{\eta}) \to \text{Cl}_F^+\to 1. 
\end{align}
We fix a splitting of the sequence (\ref{canonical extension}).
Let $\eta_i$ be a map on 
$(\ideal{m}_{\eta}^{-1}\ideal{d}_F^{-1}[t_i]^{-1}/\ideal{d}_F^{-1}[t_i]^{-1})^{\times}/\mathfrak{o}_{F,+}^{\times}$ defined by $\eta_i(\bar{b})=\sgn(b)^r\eta(\bar{b}\ideal{m}_{\eta}\ideal{d}_F[t_i])$. 
We note that $\eta_i(\xi \bar{b})=\eta(\xi)\eta_i(\bar{b})$ for any $\bar{b}\in (\ideal{m}_{\eta}^{-1}\ideal{d}_F^{-1}[t_i]^{-1}/\ideal{d}_F^{-1}[t_i]^{-1})^{\times}/\mathfrak{o}_{F,+}^{\times}$ and $0\ll\xi\in [t_i]$ prime to $\ideal{m}_{\eta}$. 

We put $\mathfrak{o}_{F,\ideal{m}_{\eta},+}^{\times}=\{e\in \mathfrak{o}_{F,+}^{\times}\mid e\equiv 1\ (\bmod\ \ideal{m}_{\eta})\}$. 
We fix a complete set $S_i$ (resp. $T$) of representatives of $(\ideal{m}_{\eta}^{-1}\ideal{d}_F^{-1}[t_i]^{-1}/\ideal{d}_F^{-1}[t_i]^{-1})^{\times}/\mathfrak{o}_{F,+}^{\times}$ in $\ideal{m}_{\eta}^{-1}\ideal{d}_F^{-1}[t_i]^{-1}$
(resp. $\mathfrak{o}_{F,+}^{\times}/\mathfrak{o}_{F,\ideal{m}_{\eta},+}^{\times}$ in $\mathfrak{o}_{F,+}^{\times}$).

For a subset $J$ of $J_F$, we put 
$z_{\iota}^J=z_{\iota}$ (resp. $\overline{z}_{\iota}$) if $\iota\in J$ (resp. $\iota \in J_F \backslash J$), and 
\begin{align}\label{dz_J}
dz_{{}_J}=\bigwedge_{\iota\in J_F} dz_{\iota}^J.
\end{align}
\begin{prop}\label{Mellin}
Let $\mathbf{f}=(f_i)_{1\le i \le h_F^+}\in S_2(\ideal{n},\C)$ be a Hecke eignform of all $T(\ideal{q})$ and $U(\ideal{q})$, and let $\eta$ a $\overline{\Q}$-valued narrow ray class character of $F$, whose conductor is denoted by $\ideal{m}_{\eta}$, such that $\ideal{m}_{\eta}$ is prime to $\ideal{d}_F[t_i]$ for each $i$. 
Then, under the above notation, we have
\begin{align*}
\sum_{i=1}^{h_F^+}\sum_{b_i\in S_i}
\eta_i(\bar{b}_i)^{-1}
\int_{\sqrt{-1}(F\otimes \mathbb{R})_+^{\times}/\mathfrak{o}_{F,\ideal{m}_{\eta},+}^{\times}}f_i(z+\bar{b}_i)dz_{{}_{J_F}}
=\tau(\eta^{-1}) \frac{D(1,\mathbf{f}, \eta)}{(-2\pi \sqrt{-1})^{n}}. 
\end{align*}
Here $\bar{b}_i$ denotes the image of $b_i\in S_i$ in $(\ideal{m}_{\eta}^{-1}\ideal{d}_F^{-1}[t_i]^{-1}/\ideal{d}_F^{-1}[t_i]^{-1})^{\times}/\mathfrak{o}_{F,+}^{\times}$ under the canonical map, $\tau(\eta^{-1})$ denotes the Gauss sum $($defined by \cite[(3.9)]{Shi}$)$, and the integrals are independent of the choice of the lift $b_i$ of $\bar{b}_i$. 
\end{prop}

In order to prove the proposition, we compute the following zeta integral:
for $\mathbf{h}\in S_k(\ideal{n},\C)$, 
\[
Z(\mathbf{h},s)=\int_{\mathbb{A}_F^{\times}/F^{\times}}
\mathbf{h}(\begin{pmatrix}t&0\\ 0&1\end{pmatrix})|t|^s d^{\times}t.
\]

\begin{prop}\label{zeta integral unrami}
Let $\mathbf{h}=(h_i)_{1\le i \le h_F^+}\in S_k(\ideal{n},\C)$. 
Then we have
\begin{align*}
Z(\mathbf{h},s)=|D|^{-s} D\left(s+k/2,\mathbf{h}\right)
\left(\frac{\Gamma(s+k/2)}{(2\pi)^{s+k/2}}\right)^n.
\end{align*}
\end{prop}
\begin{proof}
Note that 
\[
\mathbb{A}_F^{\times}
=\coprod_{1 \le i \le h_F^+} (F\otimes \mathbb{R})_+^{\times} \widehat{\mathfrak{o}}_F^{\times} D^{-1}t_i^{-1} F^{\times}.
\]
Then we have 
\begin{align*}
Z(\mathbf{h},s)&=\sum_{i=1}^{h_F^+} 
\int_{(F\otimes \mathbb{R})_+^{\times} \widehat{\mathfrak{o}}_F^{\times} D^{-1}t_i^{-1} F^{\times}/F^{\times}}
\mathbf{h}(\begin{pmatrix}t&0\\ 0&1\end{pmatrix})|t|^s d^{\times}t \\
&=\sum_{i=1}^{h_F^+} 
\int_{(F\otimes \mathbb{R})_+^{\times} \widehat{\mathfrak{o}}_F^{\times} F^{\times}/F^{\times}}
\mathbf{h}(\begin{pmatrix}D^{-1}t_i^{-1}t&0\\ 0&1\end{pmatrix})|D^{-1}t_i^{-1} t|^s d^{\times}t.
\end{align*}
Since $(F\otimes \mathbb{R})_+^{\times} \widehat{\mathfrak{o}}_F^{\times} F^{\times}/F^{\times}=(F\otimes \mathbb{R})_+^{\times}/\mathfrak{o}_{F,+}^{\times} \cdot \widehat{\mathfrak{o}}_F^{\times}$ and $\mathbf{h}$ is right-invariant with respect to the action of $\begin{pmatrix}\widehat{\mathfrak{o}}_F^{\times}&0\\ 0&1\end{pmatrix}$, we have 
\begin{align*}
Z(\mathbf{h},s)&=
\sum_{i=1}^{h_F^+} 
\int_{(F\otimes \mathbb{R})_+^{\times}/\mathfrak{o}_{F,+}^{\times}} \mathbf{h}(\begin{pmatrix}D^{-1}t_i^{-1}t_{\infty}&0\\ 0&1\end{pmatrix})|D^{-1}t_i^{-1} t_{\infty}|^s d^{\times}t_{\infty}.
\end{align*}
Thus we obtain 
\begin{align}\label{zeta=L unramified}
Z(\mathbf{h},s)
&=\sum_{i=1}^{h_F^+} 
\int_{(F\otimes \mathbb{R})_+^{\times}/\mathfrak{o}_{F,+}^{\times}} 
h_i(\sqrt{-1}t_{\infty})t_{\infty}^{k/2}|D^{-1}t_i^{-1} t_{\infty}|^s d^{\times}t_{\infty}\\
\nonumber&=\sum_{i=1}^{h_F^+} 
\int_{(F\otimes \mathbb{R})_+^{\times}/\mathfrak{o}_{F,+}^{\times}} 
\sum_{0 \ll \xi \in [t_i]}a_{\infty}(\xi,h_i)e_F(\sqrt{-1}\xi t_{\infty})|D^{-1}t_i^{-1}|^s t_{\infty}^{s+k/2} d^{\times}t_{\infty}\\
\nonumber&=|D|^{-s} \sum_{i=1}^{h_F^+}
\int_{(F\otimes \mathbb{R})_+^{\times}/\mathfrak{o}_{F,+}^{\times}} 
\sum_{0 \ll \xi \in [t_i]}
\frac{a_{\infty}(\xi,h_i)N([t_i]^{-1})^{k/2}}{N(\xi[t_i]^{-1})^{s+k/2}}e_F(\sqrt{-1}\xi t_{\infty}) (\xi t_{\infty})^{s+k/2} d^{\times}(\xi t_{\infty})\\
\nonumber&=|D|^{-s} \sum_{i=1}^{h_F^+} 
\sum_{\xi \mathfrak{o}_{F,+}^{\times}}
\frac{C(\xi[t_i]^{-1},\mathbf{h})}{N(\xi[t_i]^{-1})^{s+k/2}}
\int_{(F\otimes \mathbb{R})_+^{\times}} 
e_F(\sqrt{-1}t_{\infty}) t_{\infty}^{s+k/2} d^{\times}t_{\infty}\\
\nonumber&=|D|^{-s} D\left(s+k/2,\mathbf{h}\right)
\left(\frac{\Gamma(s+k/2)}{(2\pi)^{s+k/2}}\right)^n,
\end{align}
which proves the assertion. 
Here the first equality follows from the definition of $h_i$ (\cite[(2.7)]{Shi}) and the fourth equality follows from the definition of $C(\xi[t_i]^{-1},\mathbf{h})$ (\S \ref{Diri}).
\end{proof}

By evaluating the equation (\ref{zeta=L unramified}) at $s=0$, we obtain the following:
\begin{prop}\label{Mellin unramified}
Let $\mathbf{h}=(h_i)_{1\le i \le h_F^+}\in S_k(\ideal{n},\C)$. 
Then, under the above notation, we have
\begin{align*}
\sum_{i=1}^{h_F^+} 
\int_{\sqrt{-1}(F\otimes \mathbb{R})_+^{\times}/\mathfrak{o}_{F,+}^{\times}}h_i(z) \mathrm{Im}(z)^{k/2-1}dz_{{}_{J_F}}
=D(k/2,\mathbf{h})\left(\frac{\Gamma(k/2)}{-(2\pi)^{k/2}\sqrt{-1}}\right)^n.
\end{align*}
\end{prop}

Now we prove Proposition \ref{Mellin}.
Let $\mathbf{f}\otimes \eta=((f\otimes \eta)_i)_{1 \le i \le h_F^+}\in S_2(\ideal{n}_{\eta},\C)$ be a Hecke eigenform defined by 
\[
(f\otimes \eta)_i(z)=\sum_{0\ll \xi \in[t_i]}a_{\infty}(\xi,f_i)\eta(\xi [t_i]^{-1})e_F(\xi z)
\]
(\cite[Proposition 4.4, 4.5]{Shi}), where $\ideal{n}_{\eta}$ is the least common multiple of $\ideal{n}$, $\ideal{m}_{\eta}^2$, and $\ideal{m}_{\eta}\ideal{n}$. 
Since $D(s,\mathbf{f}\otimes \eta)=D(s,\mathbf{f},\eta)$, by applying Proposition \ref{Mellin unramified} to $\mathbf{f}\otimes \eta$ instead of $\mathbf{h}$, we have 
\begin{align}\label{Mellin ramified}
\tau(\eta^{-1})\sum_{i=1}^{h_F^+} 
\int_{\sqrt{-1}(F\otimes \mathbb{R})_+^{\times}/\mathfrak{o}_{F,+}^{\times}}(f\otimes \eta)_i(z) dz_{{}_{J_F}}
=\tau(\eta^{-1})\frac{D(1,\mathbf{f},\eta)}{(-2\pi\sqrt{-1})^{n}}.
\end{align}
Thus it suffices to show that the left-hand side of (\ref{Mellin ramified}) is equal to the left-hand side of the equation in Proposition \ref{Mellin}.
By \cite[(3.11)]{Shi}, we have 
\[
\eta(\xi[t_i]^{-1})\tau(\eta^{-1})=\sum_{b \in (\ideal{m}_{\eta}^{-1}\ideal{d}_F^{-1}[t_i]^{-1}/\ideal{d}_F^{-1}[t_i]^{-1})^{\times}}
\eta_i(\bar{b})^{-1}e_F(\xi b).
\]
Then the left-hand side of (\ref{Mellin ramified}) is equal to 
\begin{align*}
&\sum_{i=1}^{h_F^+} 
\int_{\sqrt{-1}(F\otimes \mathbb{R})_+^{\times}/\mathfrak{o}_{F,+}^{\times}}
\sum_{0\ll \xi \in[t_i]}a_{\infty}(\xi,f_i) \sum_{b_i\in S_i}\sum_{u\in T}
\eta_i(\bar{b}_i)^{-1}e_F(\xi b_i u^{-1})e_F(\xi z) dz_{{}_{J_F}}\\
&=\sum_{i=1}^{h_F^+} \sum_{b_i\in S_i}\eta_i(\bar{b}_i)^{-1}
\int_{\sqrt{-1}(F\otimes \mathbb{R})_+^{\times}/\mathfrak{o}_{F,+}^{\times}}
\sum_{u\in T} f_i(z+b_i u^{-1}) dz_{{}_{J_F}}\\
&=\sum_{i=1}^{h_F^+} \sum_{b_i\in S_i}\eta_i(\bar{b}_i)^{-1}
\int_{\sqrt{-1}(F\otimes \mathbb{R})_+^{\times}/\mathfrak{o}_{F,\ideal{m}_{\eta},+}^{\times}}
 f_i(z+b_i) dz_{{}_{J_F}},
\end{align*}
which is the the left-hand side of the equation in Proposition \ref{Mellin} as desired. 
Here we may regard $f_i(z+b_i)$ as a function on $\sqrt{-1}(F\otimes \mathbb{R})_+^{\times}/\mathfrak{o}_{F,\ideal{m}_{\eta},+}^{\times}$ because $f_i(uz+b_i)=f_i(z+b_i)$ for any $u\in \mathfrak{o}_{F,\ideal{m}_{\eta},+}^{\times}$. 
Furthermore, the integrals in the last line of this equation are independent of the choice of the lift $b_i$ of $\bar{b}_i$ (because if $\bar{b}_i=\bar{b}_i'$, then there is $\gamma \in \Gamma_1(\ideal{d}_F[t_i],\ideal{n})$ such that $b=\gamma (b')$ and $\gamma(\infty)=\infty$). 
Hence the integral depends only on the image $\bar{b}_i$ of $b_i$ in $(\ideal{m}_{\eta}^{-1}\ideal{d}_F^{-1}[t_i]^{-1}/\ideal{d}_F^{-1}[t_i]^{-1})^{\times}/\mathfrak{o}_{F,+}^{\times}$ and it shall be denoted by 
\[
\int_{\sqrt{-1}(F\otimes \mathbb{R})_+^{\times}/\mathfrak{o}_{F,\ideal{m}_{\eta},+}^{\times}}f_i(z+\bar{b}_i)dz_{{}_{J_F}}. 
\]

We consider a Mellin transform in the anti-holomorphic case. 
Let $W_G$ denote the Weyl group $K_{\infty}/K_{\infty,+}$, which is identified with the set $\{w_J\mid J\subset J_F\}$, 
where, for each subset $J$ of $J_F$, $w_J \in K_{\infty}$ such that 
$w_{J,\iota}=\begin{pmatrix}1&0\\0&1\end{pmatrix}$ if $\iota\in J$ and 
$w_{J,\iota}=\begin{pmatrix}-1&0\\0&1\end{pmatrix}$ if $\iota\in J_F \backslash J$. 
Then $W_G$ acts on the space $\bigoplus_{J\subset J_F} S_{2,J}(\ideal{n},\C)$ of Hilbert cusp forms via $\textbf{h}\mapsto \textbf{h}_J:=\textbf{h}|[K_{\infty}w_J K_{\infty}]$. 
Since $K_{\infty}w_J K_{\infty}=K_{\infty}w_J $, we have $\textbf{f}_J(x)=\textbf{f}(x w_J)$.
Then $f_{J,i}$ has the Fourier expansion of the form 
\begin{align*}
f_{J,i}(z)
=\sum_{\mu \in[t_i],\,\{\mu\}=J} a(\mu,f_{J,i})e_F(\mu z)
\end{align*}
given by \cite[Theorem 6.1]{Hida94}. 
Here $\{\mu\}=\{\iota\in J_F| \mu^{\iota}> 0 \}$.
Since the action of $W_G$ is compatible with the Hecke operators $T(\ideal{q})$ and $U(\ideal{q})$, $\textbf{f}_J$ is also a Hecke eigenform of all $T(\ideal{q})$ and $U(\ideal{q})$. 
Hence, by \cite[Corollary 6.2]{Hida94}, the $V(\ideal{q})$-eigenvalue $a(\mu,f_{J,i})N([t_i]^{-1})$ is equal to $C(\ideal{q},\mathbf{f})$, 
where $\ideal{q}=\mu[t_i]^{-1}$ and $V(\ideal{q})=T(\ideal{q})$ or $U(\ideal{q})$. 

We define $(\mathbf{f}\otimes \eta)_J=((f\otimes \eta)_{J,i})_{1 \le i \le h_F^+}\in S_{2,J}(\ideal{n}_{\eta},\C)$ by 
\[
(f\otimes \eta)_{J,i}(z)=\sum_{\mu \in[t_i],\,\{\mu\}=J}a(\mu,f_{J,i})\eta_{\infty}(\nu_J)\eta(\mu [t_i]^{-1})e_F(\mu z),
\]
where $\nu_J\in \mathbb{A}_{F,\infty}$ such that $\nu_{J,\iota}=1$ if $\iota\in J$ and $\nu_{J,\iota}=-1$ if $\iota\in J_F\backslash J$. 
Now the same argument as in the proof of Proposition \ref{Mellin} shows the following: 
\begin{prop}\label{anti Mellin}
Under the same notation and assumptions as Proposition \ref{Mellin}, we have
\begin{align*}
\sum_{i=1}^{h_F^+}\sum_{b_i\in S_i}
\eta_i(\bar{b}_i)^{-1}
\int_{\sqrt{-1}(F\otimes \mathbb{R})_+^{\times}/\mathfrak{o}_{F,\ideal{m}_{\eta},+}^{\times}}f_{J,i}(z+\bar{b}_i)dz_{{}_{J}}
=\eta_{\infty}(\nu_J)\tau(\eta^{-1})
\frac{D(1,\mathbf{f}, \eta)}{(-2\pi \sqrt{-1})^{n}}. 
\end{align*}
Here $\nu_J\in \mathbb{A}_{F,\infty}$ such that $\nu_{J,\iota}=1$ if $\iota\in J$ and $\nu_{J,\iota}=-1$ if $\iota\in J_F\backslash J$, and $dz_{{}_{J}}$ is defined by $(\ref{dz_J})$.
\end{prop}

%
\subsection{Relation between cohomology class and Dirichlet series}\label{subsection:modular symbol}
%

In this subsection, we give a cohomological description of the Dirichlet series  defined by (\ref{L-function}). 

We keep the notation in \S \ref{subsection:Mellin transform}.
We fix $i\in \Z$ with $1\le i \le h_F^+$. 
We abbreviate $\Gamma_1(\ideal{d}_F[t_i],\ideal{n})$ to $\Gamma$. 
We consider the Hilbert modular varieties $Y(\ideal{n})$ (defined by (\ref{adele HMV})) and $Y_i$ (defined by (\ref{analytic HMV})). 
Let $C_i$ denote the set of all cusps of $Y_i$. 

We fix $b_i\in S_i$.
We consider the following subset $H_{b_i}$ of $\mathfrak{H}^{J_F}$: 
\[
H_{b_i}:=b_i+\sqrt{-1}(F\otimes\mathbb{R})_+^{\times}\hookrightarrow \mathfrak{H}^{J_F}. 
\]

We define an action of $\mathfrak{o}_{F,\ideal{m}_{\eta},+}^{\times}$ on $H_{b_i}$ by 
$\varepsilon\ast(z_{\iota})_{\iota\in J_F}=(\varepsilon^{\iota}z_{\iota}-(\varepsilon^{\iota}-1)b_i)_{\iota\in J_F}$.
Since $(\varepsilon-1) b_i\in \ideal{d}_F^{-1}[t_i]^{-1}$ for any $\varepsilon\in \mathfrak{o}_{F,\ideal{m}_{\eta},+}^{\times}$, we see that 
$\varepsilon\ast(z_{\iota})_{\iota\in J_F}$ is $\Gamma$-equivalent to $(z_{\iota})_{\iota\in J_F}$. 
Therefore we have
$H_{b_i}/\mathfrak{o}_{F,\ideal{m}_{\eta},+}^{\times} \to Y_i$. 

We extend it to a morphism on their compactifications as follows. 
Let $(\mathfrak{H}^{J_F})^{\text{BS}}$ denote the Borel--Serre compactification of $\mathfrak{H}^{J_F}$, which is a locally compact manifold on which $\GL_2(F)$ acts (see, for example, \cite[\S 2.1]{Ha}, \cite[\S 1.8]{Hida93}, \cite[\S 2.1]{Hira}). 
We can describe the boundary of $(\mathfrak{H}^{J_F})^{\text{BS}}$ at the cusp $\infty$ as follows. 
We put $X=\{(y,x)\in (F\otimes \mathbb{R})_+^{\times} \times (F\otimes \mathbb{R})
\mid \prod_{\iota\in J_F}y_{\iota}=1\}$. 
Then we have 
\begin{align}\label{H^{J_F}}
\mathfrak{H}^{J_F} \xrightarrow{\simeq} X \times \mathbb{R}_+^{\times}; (x_{\iota}+\sqrt{-1}y_{\iota})_{\iota\in J_F}
\mapsto \big{(}
\left(\left(\prod_{\iota\in J_F}y_{\iota}\right)^{-\frac{1}{n}}y_{\iota},x_{\iota}\right)_{\iota\in J_F},
\prod_{\iota\in J_F} y_{\iota}\big{)},
\end{align}
which is compatible with the action of $\Gamma_{\infty}$. 
Here $\Gamma_{\infty}$ denotes the stabilizer of $\infty$ in $\Gamma$, which acts trivially on the second factor of the right-hand side. 
The compactification of $\mathfrak{H}^{J_F}$ at the cusp $\infty$ is given by $X \times (\mathbb{R}_+^{\times}\cup \{\infty\})$. 

Let $Y_i^{\text{BS}}$ denote the Borel--Serre compactification $\overline{\Gamma}\backslash(\mathfrak{H}^{J_F})^{\text{BS}}$ of $Y_i$. 
Then $Y_i^{\text{BS}}$ is a compact manifold and its  boundary at a cusp $s$, which is denoted by $D_{s}$, is given by $\overline{\Gamma_{s}}\backslash \alpha(X \times \{\infty\})$, where $\Gamma_{s}$ denotes the stabilizer of $s$ in $\Gamma$ and $\alpha\in\SL_2(F)$ such that $s=\alpha(\infty)$. 

We define subsets $H_{b_i}^{\mathrm{BS}}$, $\partial_{\infty}$, and $\partial_{b_i}$ of $X \times (\mathbb{R}_{\ge 0}\cup \{\infty\})$ as follows.
Let $X_{b_i}$ denote the image of $H_{b_i}$ in $X$ under the composition of the isomorphism (\ref{H^{J_F}}) and the projection to $X$. 
We have $H_{b_i} \simeq X_{b_i} \times \mathbb{R}_+^{\times}$. 
We define $H_{b_i}^{\mathrm{BS}}$, $\partial_{\infty}$, and $\partial_{b_i}$ by 
\begin{align*}
&H_{b_i}^{\mathrm{BS}}=X_{b_i} \times (\mathbb{R}_{\ge 0}\cup \{\infty\}),& 
&\partial_{\infty}=X_{b_i} \times \{\infty\}, &
&\partial_{b_i}=X_{b_i} \times \{0\}.&
\end{align*}
The action of $\mathfrak{o}_{F,\ideal{m}_{\eta},+}^{\times}$ on $H_{b_i}$ extends canonically to an action on $H_{b_i}^{\mathrm{BS}}$.
We put $\alpha_{b_i}=\small\begin{pmatrix}-b_i&1+b_i^2\\-1&b_i \end{pmatrix}$.
Note that $\alpha_{b_i}(\infty)=b_i$.
The embedding $H_{b_i} \hookrightarrow \mathfrak{H}^{J_F}$
(resp. the composition $H_{b_i} \xrightarrow{\alpha_{b_i}} H_{b_i} \hookrightarrow \mathfrak{H}^{J_F}\xrightarrow{\alpha_{b_i}}\mathfrak{H}^{J_F}$) induces an $\mathfrak{o}_{F,\ideal{m}_{\eta},+}^{\times}$-equivariant map
\begin{align}\label{extend to H_b^{BS}}
&H_{b_i} \cup \partial_{\infty} \to \mathfrak{H}^{J_F}\cup (X\times\{\infty\})&
&\text{(resp. $H_{b_i} \cup \partial_{b_i} \to \mathfrak{H}^{J_F}\cup \alpha_{b_i}(X\times\{\infty\})$)}&
\end{align}
because $\alpha_{b_i}(b_i+\sqrt{-1}y)=b_i+\sqrt{-1}/y$.
Therefore we have $H_{b_i}^{\mathrm{BS}}/\mathfrak{o}_{F,\ideal{m}_{\eta},+}^{\times} \to Y_i^{\mathrm{BS}}$ and it induces
\begin{align}\label{Gysin}
H_c^n(Y(\ideal{n}),A&)
\to H_c^n(Y_i,A)\simeq H^n(Y_i^{\text{BS}},\partial(Y_i^{\text{BS}});A)\\
\nonumber&\to H^n(H_{b_i}^{\mathrm{BS}}/\mathfrak{o}_{F,\ideal{m}_{\eta},+}^{\times},\partial_{b_i}/\mathfrak{o}_{F,\ideal{m}_{\eta},+}^{\times}\cup\partial_{\infty}/\mathfrak{o}_{F,\ideal{m}_{\eta},+}^{\times};A)
\simeq
H_c^n(H_{b_i}/\mathfrak{o}_{F,\ideal{m}_{\eta},+}^{\times},A)
\end{align}
for $A=\integer{}$, $K$, or $\C$. 
Here $H_c^n(X,A)$ denotes the compactly supported cohomology of $X$ with coefficients in $A$, and $\partial(Y_i^{\text{BS}})$ denotes the boundary $\coprod_{s\in C_i}D_{s}$ of $Y_i^{\text{BS}}$. 

We define the evaluation map
\begin{align}\label{ev map}
\mathrm{ev}_{b_i,A} :
\widetilde{H}_c^n(Y(\ideal{n}),A)\to A
\end{align}
by the composition of (\ref{Gysin}) and the trace map 
$H_c^n(H_{b_i}/\mathfrak{o}_{F,\ideal{m}_{\eta},+}^{\times},A)\to A$, where 
\[
\widetilde{H}_c^n(Y(\ideal{n}),A)
:=H_c^n(Y(\ideal{n}),A)/\text{($A$-torsion)}.
\]
Note that the definition of $\mathrm{ev}_{b_i,A}$ depends only on $\bar{b}_i$ (because if $\bar{b}_i=\bar{b}_i'$, then there is $\gamma \in \Gamma$ such that $b=\gamma (b')$ and $\gamma(\infty)=\infty$) and hence it shall be denoted by $\mathrm{ev}_{\bar{b}_i,A}$.

Let us fix a Hilbert cusp form $\mathbf{f}\in S_2(\ideal{n},\C)$. 
Let $[\omega_{\mathbf{f}}]$ denote the associated cohomology class in $H^n(Y(\ideal{n}),\C)$.
Let $[\omega_{\mathbf{f}}]_{c}$ denote the compactly supported cohomology class in $H_c^n(Y(\ideal{n}),\C)$ whose image in $H^n(Y(\ideal{n}),\C)$ is $[\omega_{\textbf{f}}]$. 
By combining these observations and Proposition \ref{Mellin}, we obtain the following cohomological description of the Dirichlet series defined by (\ref{L-function}).
\begin{prop}\label{Mellin hol}
Let $\mathbf{f}\in S_2(\ideal{n},\integer{})$ be a Hecke eigenform for all $T(\ideal{q})$ and $U(\ideal{q})$.
Assume that $a [\omega_{\mathbf{f}}]_{c}\in \widetilde{H}_c^n(Y(\ideal{n}),A)$ for some $a\in A$. 
Then, under the same notation and assumptions as Proposition \ref{Mellin}, we have 
\begin{align*}
A(\eta) \ni \sum_{i=1}^{h_F^+}\sum_{b_i\in S_i}\eta_i(\bar{b}_i)^{-1} \mathrm{ev}_{\bar{b}_i,A}(a[\omega_{\mathbf{f}}]_{c})
=a\tau(\eta^{-1})
\frac{D(1,\mathbf{f}, \eta)}{(-2\pi \sqrt{-1})^{n}}. 
\end{align*}
\end{prop}

In the anti-holomorphic case, we also obtain the following cohomological description by the same argument as in the proof of Proposition \ref{Mellin hol}, using Proposition \ref{anti Mellin} instead of Proposition \ref{Mellin}.
\begin{prop}\label{Mellin anti-hol}
Under the same notation and assumptions as Proposition \ref{anti Mellin} and Proposition \ref{Mellin hol}, we have 
\begin{align*}
A(\eta) \ni \sum_{i=1}^{h_F^+}\sum_{b_i\in S_i}
\eta_i(\bar{b}_i)^{-1} 
\mathrm{ev}_{\bar{b}_i,A}\left(a[\omega_{\mathbf{f}}]_c|[K_{\infty}w_J K_{\infty}]\right)
=a\eta_{\infty}(\nu_J)\tau(\eta^{-1})
\frac{D(1,\mathbf{f},\eta)}{(-2\pi \sqrt{-1})^{n}}. 
\end{align*}
\end{prop}

%
\subsection{Distribution property and interpolation property}\label{subsection:const of p-adic L}
%

In this subsection, we construct a $\C$\nobreakdash-valued distribution attached to a Hilbert cusp form. 

Let $p$ be a prime number such that $(p,6\ideal{d}_F\ideal{n})=1$. 
We decompose $(p)=\ideal{p}_1\cdots \ideal{p}_r$ into prime ideals $\ideal{p}_i$ of $\ideal{o}_F$. 
Let $\mathbf{f}\in S_2(\ideal{n},\integer{})$ be a normalized Hecke eigenform for all $T(\ideal{q})$ and $U(\ideal{q})$ with character $\varepsilon$. 
We assume that $\mathbf{f}$ is $p$-ordinary, that is, for each $j$ with $1\le j \le r$, the $T(\ideal{p}_j)$-eigenvalue $C(\ideal{p}_j,\mathbf{f})$ is prime to $p$. 
Let $\alpha_{\ideal{p}_j}$ be a unit root of the polynomial 
\begin{align}\label{unit root}
X^2-C(\ideal{p}_j,\mathbf{f})X+\varepsilon(\ideal{p}_j)N(\ideal{p}_j)=0.
\end{align}
Note that 
\[
K_0(\ideal{n})
\begin{pmatrix}\varpi_{\ideal{p}_j}&0\\ 0&1\end{pmatrix}
K_0(\ideal{n})
=\coprod_{u\in \ideal{o}_F/\ideal{p}_j}
K_0(\ideal{n})\begin{pmatrix}1 &u\\ 0& \varpi_{\ideal{p}_j}\end{pmatrix}
\coprod K_0(\ideal{n})\begin{pmatrix}\varpi_{\ideal{p}_j}&0\\ 0&1\end{pmatrix}.
\]
Then we have 
\begin{align}\label{Hecke relation of f}
C(\ideal{p}_j,\mathbf{f})\mathbf{f}(x)=\sum_{u\in \ideal{o}_F/\ideal{p}_j} 
\mathbf{f}(x \begin{pmatrix}\varpi_{\ideal{p}_j} &u\\ 0& 1\end{pmatrix})
+\varepsilon(\ideal{p}_j)
\mathbf{f}(x\begin{pmatrix}\varpi_{\ideal{p}_j}^{-1}&0\\ 0&1 \end{pmatrix}).
\end{align}

For $v=(v_j)_{1\le j\le r} \in (\Z_{\ge 0})^r$, 
we simply write $p^v=\prod_{1\le j \le r}\ideal{p}_j^{v_j}$ if there is no risk of confusion.
For a non-zero ideal $\ideal{a}$ of $\ideal{o}_F$, let $v(\ideal{a})=(v(\ideal{a})_j)_{1\le j\le r}\in (\Z_{\ge 0})^r$ such that $\ideal{a}p^{-v(\ideal{a})}\subset \ideal{o}_F$ and $\ideal{a}p^{-v(\ideal{a})}$ is prime to $p$. 
We use the convention $\alpha^{-v(\ideal{a})}=\prod_{1\le j\le r}\alpha_{\ideal{p}_j}^{-v(\ideal{a})_j}$.

For a non-zero ideal $\ideal{m}$ of $\ideal{o}_F$ and $\underline{a}=(a_i)_{1 \le i \le h_F^+} \in \bigoplus_{1\le i\le h_F^+}(\ideal{m}^{-1}\ideal{d}_F^{-1}[t_i]^{-1}/\ideal{d}_F^{-1}[t_i]^{-1})^{\times}/\ideal{o}_{F,+}^{\times}$, under the same notation in \S \ref{subsection:modular symbol}, 
we put 
\begin{align*}
H_{\underline{a}}:=\bigoplus_{1 \le i \le h_F^+}H_{a_i} \ \ \text{and} \ \ 
\mathrm{ev}_{\underline{a},\C}:=\bigoplus_{1 \le i \le h_F^+}\mathrm{ev}_{a_i,\C}. 
\end{align*}
We note that the image of $H_{\underline{a}}$ under the right multiplication of $\begin{pmatrix}\varpi_{\ideal{p}_j}^{-1}&0\\ 0&1 \end{pmatrix}$ is expressed as $H_{\underline{a(\ideal{p}_j)}}$, where the element $\underline{a(\ideal{p}_j)}=(a(\ideal{p}_j)_i)_{1\le i\le h_F^+}\in \bigoplus_{1\le i\le h_F^+}\ideal{p}_j\ideal{m}^{-1}\ideal{d}_F^{-1}[t_i]^{-1}$ is explicitly given as follows: 
under the same notation in \S \ref{Analytic HMV} and a fixed splitting of (\ref{canonical extension}), by computing modulo the left action of $G(F)$ and the right action of $K_1(\ideal{n})$, 
$x_i \begin{pmatrix}y_{\infty}&a_i \\0&1\end{pmatrix}_{\infty}
\equiv \begin{pmatrix}1&-a_i\\0&1\end{pmatrix}_0
\begin{pmatrix}D^{-1}t_i^{-1}&0\\0&1\end{pmatrix}
\begin{pmatrix}y_{\infty}&0\\0&1\end{pmatrix}_{\infty}$ and 
\begin{align*}
x_i \begin{pmatrix}y_{\infty}&a_i \\0&1\end{pmatrix}_{\infty}
\begin{pmatrix}\varpi_{\ideal{p}_j}^{-1}&0\\0&1\end{pmatrix}_0
\equiv &\begin{pmatrix}1&-a_i\\0&1\end{pmatrix}_0 
\begin{pmatrix}D^{-1}t_i^{-1}\varpi_{\ideal{p}_j}^{-1}&0\\0&1\end{pmatrix}
\begin{pmatrix}y_{\infty}&0\\0&1\end{pmatrix}_{\infty}\\
\equiv&\begin{pmatrix}1&-(a')^{-1}a_i\\0&1\end{pmatrix}_0
x_{k} \begin{pmatrix}t'y_{\infty}&0\\0&1\end{pmatrix}_{\infty},
\end{align*}
where $a'\in F^{\times}$ and $t'\in (F\otimes\mathbb{R})_{+}^{\times}$ such that $D^{-1}t_i^{-1}\varpi_{\ideal{p}_j}^{-1}=a'D^{-1}t_k^{-1}u't'$ for some $u' \in \widehat{\ideal{o}}_F^{\times}$.

Let $R(\ideal{p}_j)$ be an operation such that $R(\ideal{p}_j)\mathrm{ev}_{\underline{a},\C}$ is the evaluation map $\mathrm{ev}_{\underline{a(\ideal{p}_j)},\C}$ associated to $H_{\underline{a(\ideal{p}_j)}}$. 
We note that the operations $R(\ideal{p}_j)$ for $1 \le j \le r$ are compatible with each other. 

For a non-zero ideal $\ideal{m}$ of $\ideal{o}_F$, we define a map 
\[
\mu_{\mathbf{f},\alpha}(-,\ideal{m}):\bigoplus_{1 \le i \le h_F^+}(\ideal{m}^{-1}\ideal{d}_F^{-1}[t_i]^{-1}/\ideal{d}_F^{-1}[t_i]^{-1})^{\times}/\ideal{o}_{F,+}^{\times} \to \C
\]
by the following:
for $\underline{a} \in \bigoplus_{1\le i\le h_F^+}(\ideal{m}^{-1}\ideal{d}_F^{-1}[t_i]^{-1}/\ideal{d}_F^{-1}[t_i]^{-1})^{\times}/\ideal{o}_{F,+}^{\times}$,
\begin{align}\label{def of distribution}
\mu_{\mathbf{f},\alpha}(\underline{a},\ideal{m})
&=\alpha^{-v(\ideal{m})}
\left(\prod_{\ideal{q}|p} \left(1-\alpha_{\ideal{q}}^{-1}\varepsilon(\ideal{q})R(\ideal{q})\right)\right)
\mathrm{ev}_{\underline{a},\C}([\omega_{\mathbf{f}}]_c). 
\end{align}

\begin{prop}[distribution property]\label{prop:distribution}
Let $\mathbf{f}\in S_2(\ideal{n},\integer{})$ be a normalized Hecke eigenform for all $T(\ideal{q})$ and $U(\ideal{q})$ with character $\varepsilon$ and $p$-ordinary. 
Then, for a non-zero ideal $\ideal{m}$ of $\ideal{o}_F$ and $\underline{a} \in \bigoplus_{1\le i\le h_F^+}(\ideal{m}^{-1}\ideal{d}_F^{-1}[t_i]^{-1}/\ideal{d}_F^{-1}[t_i]^{-1})^{\times}/\ideal{o}_{F,+}^{\times}$, we have
\[
\sum_{\substack{\underline{b}\in  \bigoplus_{1 \le i \le h_F^+}(\ideal{m}^{-1}p^{-1}\ideal{d}_F^{-1}[t_i]^{-1}/\ideal{d}_F^{-1}[t_i]^{-1})^{\times}/\ideal{o}_{F,+}^{\times} \\ \underline{b}\equiv\underline{a}\,(\mathrm{mod}\,p)}}
\mu_{\mathbf{f},\alpha}(\underline{b},\ideal{m}p)=\mu_{\mathbf{f},\alpha}(\underline{a},\ideal{m}).\ \ \ \ \ \ \ \ \ \ \ \ \ \ \ \ \ \ \ \ \ 
\]
Here $\underline{b}$ runs over a complete set of representatives of $\bigoplus_{i=1}^{h_F^+}(\ideal{m}^{-1}p^{-1}\ideal{d}_F^{-1}[t_i]^{-1}/\ideal{d}_F^{-1}[t_i]^{-1})^{\times}/\ideal{o}_{F,+}^{\times}$, whose image in $\bigoplus_{i=1}^{h_F^+}(\ideal{m}^{-1}\ideal{d}_F^{-1}[t_i]^{-1}/\ideal{d}_F^{-1}[t_i]^{-1})^{\times}/\ideal{o}_{F,+}^{\times}$ under the canonical map is equal to $\underline{a}$.
\end{prop}
\begin{proof}
We claim that, for every prime ideal $\ideal{p}$ of $\ideal{o}_F$ dividing $p$, non-zero ideal $\ideal{m}$ of $\ideal{o}_F$, and $\underline{a}\in  \bigoplus_{1 \le i \le h_F^+}(\ideal{m}^{-1}\ideal{d}_F^{-1}[t_i]^{-1}/\ideal{d}_F^{-1}[t_i]^{-1})^{\times}/\ideal{o}_{F,+}^{\times}$,
\begin{align}\label{claim on distribution}
\sum_{\substack{\underline{b}\in  \bigoplus_{1 \le i \le h_F^+}(\ideal{m}^{-1}\ideal{p}^{-1}\ideal{d}_F^{-1}[t_i]^{-1}/\ideal{d}_F^{-1}[t_i]^{-1})^{\times}/\ideal{o}_{F,+}^{\times} \\ \underline{b}\equiv\underline{a}\,(\mathrm{mod}\,\ideal{p})}}
\mu_{\mathbf{f},\alpha}(\underline{b},\ideal{mp})=\mu_{\mathbf{f},\alpha}(\underline{a},\ideal{m}).
\end{align}
Here $\overline{b}$ runs over a complete set of representatives of $\bigoplus_{i=1}^{h_F^+}(\ideal{m}^{-1}\ideal{p}^{-1}\ideal{d}_F^{-1}[t_i]^{-1}/\ideal{d}_F^{-1}[t_i]^{-1})^{\times}/\ideal{o}_{F,+}^{\times}$, whose image in $\bigoplus_{i=1}^{h_F^+}(\ideal{m}^{-1}\ideal{d}_F^{-1}[t_i]^{-1}/\ideal{d}_F^{-1}[t_i]^{-1})^{\times}/\ideal{o}_{F,+}^{\times}$ under the canonical map is equal to $\underline{a}$.

For the moment, we admit the claim (\ref{claim on distribution}).
Let $T'$ be a complete set of representatives of 
$\bigoplus_{i=1}^{h_F^+}(\ideal{m}^{-1}\ideal{p}_1^{-1}\cdots \ideal{p}_{r-1}^{-1}\ideal{d}_F^{-1}[t_i]^{-1}/\ideal{d}_F^{-1}[t_i]^{-1})^{\times}/\ideal{o}_{F,+}^{\times}$ in $\bigoplus_{i=1}^{h_F^+}\ideal{m}^{-1}\ideal{p}_1^{-1}\cdots \ideal{p}_{r-1}^{-1}\ideal{d}_F^{-1}[t_i]^{-1}$. 
Then the left-hand side in the proposition is equal to 
\begin{align*}
\sum_{\substack{\underline{b}'\in T'\\\underline{b}'\equiv \underline{a}\,(\mathrm{mod}\,\ideal{p}_1\cdots\ideal{p}_{r-1})}}
\sum_{\substack{\underline{b}\in  \bigoplus_{1 \le i \le h_F^+}(\ideal{m}^{-1}\ideal{p}_1^{-1}\cdots \ideal{p}_{r-1}^{-1} \ideal{p}_r^{-1}\ideal{d}_F^{-1}[t_i]^{-1}/\ideal{d}_F^{-1}[t_i]^{-1})^{\times}/\ideal{o}_{F,+}^{\times}\\ \underline{b}\equiv \underline{b}'\,(\mathrm{mod}\,\ideal{p}_r)}}
\mu_{\mathbf{f},\alpha}(\underline{b},\ideal{m}p).
\end{align*}
Hence, by applying (\ref{claim on distribution}) to the triple 
$(\ideal{p}_r,\ideal{m}\ideal{p}_1\cdots \ideal{p}_{r-1}, \underline{b}')$ instead of $(\ideal{p},\ideal{m},\underline{a})$, 
the left-hand side in the proposition is equal to 
\begin{align*}
\sum_{\substack{\underline{b}'\in T'\\ \underline{b}'\equiv \underline{a}\,(\mathrm{mod}\,\ideal{p}_1\cdots\ideal{p}_{r-1})}}
\mu_{\mathbf{f},\alpha}(\underline{b}',\ideal{m}\ideal{p}_1\cdots \ideal{p}_{r-1}).
\end{align*}
Now, by the same argument as above, we can prove the proposition as desired. 

It remains to prove the claim (\ref{claim on distribution}).
The left-hand side of (\ref{claim on distribution}) is equal to 
\begin{align*}
\alpha&^{-v(\ideal{mp})}\sum_{\underline{b}\equiv\underline{a}\,(\mathrm{mod}\,\ideal{p})} 
\left(\prod_{\ideal{q}|p}(1-\alpha_{\ideal{q}}^{-1}\varepsilon(\ideal{q})R(\ideal{q}))\right)
\mathrm{ev}_{\underline{b},\C}([\omega_{\mathbf{f}}]_c)\\
&=\alpha^{-v(\ideal{mp})}\sum_{\underline{b}\equiv\underline{a}\,(\mathrm{mod}\,\ideal{p})}
\left(\prod_{\substack{\ideal{q}|p\\ \ideal{q}\neq \ideal{p}}}(1-\alpha_{\ideal{q}}^{-1}\varepsilon(\ideal{q})R(\ideal{q})) \right)
\mathrm{ev}_{\underline{b},\C}
([\omega_{\mathbf{f}}]_c)\\
&\ \ \ \ -\alpha^{-v(\ideal{m}\ideal{p}^2)}\varepsilon(\ideal{p})
\sum_{\underline{b}\equiv\underline{a}\,(\mathrm{mod}\,\ideal{p})} \left(\prod_{\substack{\ideal{q}|p\\ \ideal{q}\neq \ideal{p}}}(1-\alpha_{\ideal{q}}^{-1}\varepsilon(\ideal{q})R(\ideal{q}))\right)
R(\ideal{p})\mathrm{ev}_{\underline{b},\C}
([\omega_{\mathbf{f}}]_c ).
\end{align*}
We note that the first term above is equal to 
\begin{align}\label{first term}
\alpha^{-v(\ideal{mp})}
\left(C(\ideal{p},\mathbf{f})-\varepsilon(\ideal{p})R(\ideal{p})\right)
\left(\prod_{\substack{\ideal{q}|p\\ \ideal{q}\neq \ideal{p}}}(1-\alpha_{\ideal{q}}^{-1}\varepsilon(\ideal{q})R(\ideal{q})) \right)
\mathrm{ev}_{\underline{a},\C}
([\omega_{\mathbf{f}}]_c)
\end{align}
and the second term above is equal to 
\begin{align}\label{second term}
-\alpha^{-v(\ideal{mp}^2)}
\varepsilon(\ideal{p})N(\ideal{p})
\left(\prod_{\substack{\ideal{q}|p\\ \ideal{q}\neq \ideal{p}}}(1-\alpha_{\ideal{q}}^{-1}\varepsilon(\ideal{q})R(\ideal{q})) \right)
\mathrm{ev}_{\underline{a},\C}
([\omega_{\mathbf{f}}]_c).
\end{align}
Here the former (\ref{first term}) follows from (\ref{Hecke relation of f}) and the following: under the same notation in \S \ref{Analytic HMV}, 
a fixed splitting of (\ref{canonical extension}), 
and the identification $\ideal{o}_F/\ideal{p} \simeq \ideal{d}_F^{-1}[t_i]^{-1}/\ideal{p}\ideal{d}_F^{-1}[t_i]^{-1}$, by computing modulo the left action of $G(F)$ and the right action of $K_1(\ideal{n})$, 
\begin{align*}
\label{cal1}x_i \begin{pmatrix}y_{\infty}&a_i \\0&1\end{pmatrix}_{\infty}
\begin{pmatrix}\varpi_{\ideal{p}}&-u\\0&1\end{pmatrix}_0
\equiv& \begin{pmatrix}1&-a_i\\0&1\end{pmatrix}_0
\begin{pmatrix}D^{-1}t_i^{-1}&0\\0&1\end{pmatrix}
\begin{pmatrix}y_{\infty}&0\\0&1\end{pmatrix}_{\infty}
\begin{pmatrix}\varpi_{\ideal{p}}&-u\\0&1\end{pmatrix}_0\\
\nonumber=&\begin{pmatrix}1&-(a_i+u D^{-1}t_i^{-1})\\0&1\end{pmatrix}_0 
\begin{pmatrix}D^{-1}t_i^{-1}\varpi_{\ideal{p}}&0\\0&1\end{pmatrix}
\begin{pmatrix}y_{\infty}&0\\0&1\end{pmatrix}_{\infty}\\
\nonumber\equiv&\begin{pmatrix}1&-(a'')^{-1}(a_i+u D^{-1}t_i^{-1})\\0&1\end{pmatrix}_0
x_j 
\begin{pmatrix}t''y_{\infty}&0\\0&1\end{pmatrix}_{\infty},
\end{align*}
where $a''\in F^{\times}$ and $t''\in (F\otimes\mathbb{R})_{+}^{\times}$ such that $D^{-1}t_i^{-1}\varpi_{\ideal{p}}=a''D^{-1}t_j^{-1}u''t''$ for some $u'' \in \widehat{\ideal{o}}_F^{\times}$. 
Here we note that the $(1,2)$ entry of the first matrix in the last line is a lift of $a_i$. 
Furthermore the latter (\ref{second term}) follows from that a complete set of such representatives in (\ref{claim on distribution}) 
is given by $\left\{(u+a_i)/\varpi_{\ideal{p}}\mid u\in \ideal{o}_F/\ideal{p} \simeq \ideal{d}_F^{-1}[t_i]^{-1}/\ideal{p}\ideal{d}_F^{-1}[t_i]^{-1} \right\}$ 
and the map $\mathrm{ev}_{(u+a_i)_{i},\C}$ depends only on $\overline{(u+a_i)}_i=\overline{(a_i)}_i$. 
Since $\varepsilon(\ideal{p})N(\ideal{p})=-\alpha_{\ideal{p}}^2 +C(\ideal{p},\mathbf{f})\alpha_{\ideal{p}}$ by (\ref{unit root}), therefore we obtain the claim (\ref{claim on distribution}) as desired.
\end{proof}

\begin{prop}[interpolation property]\label{prop:interpolation}
Let $\mathbf{f}\in S_2(\ideal{n},\C)$ be a normalized Hecke eigenform with character $\varepsilon$ and $p$-ordinary. 
Let $\eta$ be a narrow ray class character of $F$, whose conductor is denoted by $\ideal{m}_{\eta}$, such that $\ideal{m}_{\eta}$ is prime to $\ideal{d}_F[t_i]$ for each $i$, and $\ideal{n}p|\ideal{m}_{\eta}$. 
Then we have 
\[
\sum_{\underline{b}\in \bigoplus_{1 \le i \le h_F^+}\left(\ideal{m}_{\eta}^{-1}\ideal{d}_F^{-1}[t_i]^{-1}/\ideal{d}_F^{-1}[t_i]^{-1}\right)^{\times}/\ideal{o}_{F,+}^{\times}}
\eta_i(\overline{b}_i)^{-1}
\mu_{\mathbf{f},\alpha}(\underline{b},\ideal{m}_{\eta})
=\alpha^{-v(\ideal{m}_{\eta})}\tau(\eta^{-1})
\frac{D(1,\mathbf{f},\eta)}{(-2\pi \sqrt{-1})^{n}},
\]
where $\underline{b}$ runs over a complete set of representatives of $\bigoplus_{i=1}^{h_F^+}\left(\ideal{m}_{\eta}^{-1}\ideal{d}_F^{-1}[t_i]^{-1}/\ideal{d}_F^{-1}[t_i]^{-1}\right)^{\times}/\ideal{o}_{F,+}^{\times}$.
\end{prop}
\begin{proof}
By the definition (\ref{def of distribution}) and Proposition \ref{Mellin hol}, 
it suffices to show that, for each $m$, 
\[
\sum_{\underline{b}}\eta_i(\overline{b}_i)^{-1}
R(\ideal{p}_1)\cdots R(\ideal{p}_m)
\mathrm{ev}_{\underline{b},\C}([\omega_{\mathbf{f}}]_c)=0.
\]
Note that, by the definition of $R(\ideal{p}_j)$, the value $R(\ideal{p}_1)\cdots R(\ideal{p}_m)\mathrm{ev}_{\underline{b},\C}([\omega_{\mathbf{f}}]_c)$ depends on $\underline{b}$ modulo $\bigoplus_{i=1}^{h_F^+}\ideal{p}_1^{-1} \cdots \ideal{p}_m^{-1}\ideal{d}_F^{-1}[t_i]^{-1}$. 
Therefore, our assertion follows from that $\eta$ is a primitive character. 
\end{proof}

%
\section{Integrality of $p$-adic $L$-functions}\label{section:integrality of p-adic L}
%
In this subsection, we prove the integrality of $p$-adic $L$-functions attached to Hilbert cusp forms divided by the canonical periods. 
In order to show it, we need the integrality of the relative cohomology class attached to a Hilbert cusp form. 

%
\subsection{Partial Eichler--Shimura--Harder isomorphism}\label{subsection:Eichler--Shimura--Harder}
%

In this subsection, we recall the Eichler--Shimura--Harder isomorphism (\ref{+,+ decomp}), where we use the assumption $h_F^+=1$ (\cite[\S 4.2]{Hira}).

Let $K$ be a finite extension of the field $\Phi_p$ defined in \S \ref{analytic HMV}, and $\integer{}$ the ring of integers of $K$. 
For $A=\integer{}$, $K$, or $\C$, let $H_c^n(Y(\ideal{n}),A)$ denote the compactly supported cohomology of $Y(\ideal{n})$ with coefficients in $A$, and let $H_{\pa}^n(Y(\ideal{n}),A)$ denote the parabolic cohomology of $Y(\ideal{n})$ with coefficients in $A$, that is, $H_{\pa}^n(Y(\ideal{n}),A)=\im\left(H_c^n(Y(\ideal{n}),A)\to H^n(Y(\ideal{n}),A)\right)$. 
Let $\widetilde{H}_{\pa}^n(Y(\ideal{n}),A)$ denote the torsion-free part of $H_{\pa}^n(Y(\ideal{n}),A)$, that is, $\widetilde{H}_{\pa}^n(Y(\ideal{n}),A)=H_{\pa}^n(Y(\ideal{n}),A)$ for $A=K$ or $\C$ and 
\begin{align*}
\widetilde{H}_{\pa}^n(Y(\ideal{n}),\integer{})
&=\im(H_{\pa}^n(Y(\ideal{n}),\integer{})\to H_{\pa}^n(Y(\ideal{n}),K)).
\end{align*}
As mentioned in \cite[$\S$7]{Hida88}, $H_{\pa}^n (Y(\ideal{n}),A)$ is a $W_G$-module and hence so is $\widetilde{H}_{\pa}^n(Y(\ideal{n}),A)$,
where $W_G$ is the Weyl group explained before Proposition \ref{anti Mellin}.
In the case where $n$ is even, if a character $\epsilon$ of $W_G$ satisfies $\sharp \{\iota\in J_F \mid \epsilon(-1_{\iota})=-1\} \neq n/2$, then we have
\begin{align}\label{+,+ decomp}
H_{\pa}^n (Y(\ideal{n}),\C)[\epsilon]
\simeq S_2(\ideal{n},\C)
\end{align}
as Hecke modules, where we use the assumption that $h_F^+=1$ (\cite[\S 4.2, (4.7)]{Hira}). 
Here $W_G$ is identified with $\{\pm1\}^{J_F}$ by the determinant map, and 
for a $W_G$-module $V$, $V[\epsilon]$ denotes the $\epsilon$-isotypic part $\{v\in V \mid w\cdot v=\epsilon(w)v\ \text{for all}\ w \in W_G \}$. 
Thus the Hecke algbra $\mathscr{H}_2(\ideal{n},\integer{})$ is isomorphic to the $\integer{}$\nobreakdash-subalgebra of 
$\mathrm{End}_{\integer{}}\left(\widetilde{H}_{\pa}^n(Y(\ideal{n}),\integer{})[\epsilon]\right)$.

%
%
\subsection{Canonical periods}\label{subsection:canonical period}
%
%

In this subsection, we recall the definition of the canonical periods (\ref{period}) (\cite[\S 6.1]{Hira}).

We keep the notation in \S \ref{subsection:Eichler--Shimura--Harder}.
Let $\textbf{f}\in S_2(\ideal{n},\integer{})$ be a normalized Hecke eigenform for all $T(\ideal{q})$ and $U(\ideal{q})$ with character $\varepsilon$. 
We put the character $\epsilon_{\mathbf{f}}=\textbf{1}$ or $\text{sgn}^{J_F}$ of $W_G$. 
Let $\mathfrak{p}_{\textbf{f}}$ denote the prime ideal of the Hecke algebra $\mathscr{H}_2(\ideal{n},\integer{})$ 
generated by $T(\mathfrak{q})-C(\ideal{q},\textbf{f})$ and $S(\mathfrak{q})-\varepsilon^{-1}(\mathfrak{q})$ for all non-zero prime ideals $\mathfrak{q}$ of $\mathfrak{o}_F$ prime to $\ideal{n}$, and $U(\ideal{q})-C(\ideal{q},\textbf{f})$ for all non\nobreakdash-zero prime ideals $\ideal{q}$ of $\ideal{o}_F$ dividing $\ideal{n}$. 
The isomorphism (\ref{+,+ decomp}) and the $q$-expansion principle over $\C$ imply that 
$\dim_{\C}\left( H_{\pa}^n(Y(\ideal{n}),\C)[\epsilon_{\mathbf{f}},\mathfrak{p}_{\textbf{f}}]\right)=1$ and $\rank_{\integer{}}\left( \widetilde{H}_{\pa}^n(Y(\ideal{n}),\integer{})[\epsilon_{\mathbf{f}},\mathfrak{p}_{\textbf{f}}]\right)=1$. We choose a generator $[\delta_{\textbf{f}}]^{\epsilon_{\mathbf{f}}}$ of $\widetilde{H}_{\pa}^n(Y(\ideal{n}),\integer{})[\epsilon_{\mathbf{f}},\mathfrak{p}_{\textbf{f}}]$. 
Let $[\omega_{\textbf{f}}]^{\epsilon_{\mathbf{f}}}$ denote the projection of $[\omega_{\textbf{f}}]$ to the $\epsilon_{\mathbf{f}}$\nobreakdash-isotypic part $H_{\pa}^n(Y(\ideal{n}),\C)[\epsilon_{\mathbf{f}},\mathfrak{p}_{\textbf{f}}]$. 
We define the canonical period $\Omega_{\textbf{f}}^{\epsilon_{\mathbf{f}}} \in \C^{\times}$ of $\mathbf{f}$ by
\begin{align}\label{period}
[\omega_{\mathbf{f}}]^{\epsilon_{\mathbf{f}}}=\Omega_{\mathbf{f}}^{\epsilon_{\mathbf{f}}}[\delta_{\mathbf{f}}]^{\epsilon_{\mathbf{f}}}. 
\end{align}

Let $C(\Gamma_{1}(\ideal{d}_F [t_1],\ideal{n}))$ denote the set of all cusps of $Y(\ideal{n})$, and let $C_{\infty}$ denote the subset of $C(\Gamma_{1}(\ideal{d}_F [t_1],\ideal{n}))$ consisting of cusps $\Gamma_{0}(\ideal{d}_F [t_1],\ideal{n})$\nobreakdash-equivalent to the cusp $\infty$. 
Let $D_{C_{\infty}}(\ideal{n})$ denote the union of $D_{s}$ for all $s\in C_{\infty}$, where $D_{s}$ is the boundary of $Y(\ideal{n})^{\text{BS}}$ at a cusp $s$ (\S \ref{subsection:modular symbol}). 
As explained in \cite[\S 5.1]{Hira}, the Hecke correspondence $U(\ideal{q})$ preserves the component $D_{C_{\infty}}(\ideal{n})$.
Let $\mathbb{H}_2(\ideal{n},\integer{})'$ be the commutative $\integer{}$\nobreakdash-subalgebra of $\End_{\integer{}}({H}^{n-1}(D_{C_{\infty}}(\ideal{n}),\integer{})) \oplus \End_{\integer{}}({H}^n (Y(\ideal{n})^{\text{BS}},D_{C_{\infty}}(\ideal{n});\integer{}))\oplus \End_{\integer{}}({H}^n (Y(\ideal{n}),\integer{}))\oplus \End_{\integer{}}({H}^n(D_{C_{\infty}}(\ideal{n}),\integer{}))$ generated by $U(\ideal{q})$ for all non-zero prime ideals $\ideal{q}$ of $\mathfrak{o}_F$ dividing $\ideal{n}$, and 
$\ideal{m}_{\mathbf{f}}'$ the maximal ideal of $\mathbb{H}_2(\ideal{n},\integer{})'$ generated by $\varpi$ and $U(\ideal{q})-C(\ideal{q},\mathbf{f})$ for all non-zero prime ideals $\ideal{q}$ of $\mathfrak{o}_F$ dividing $\ideal{n}$. 

As mentioned in \cite[\S 2.4]{Hira}, by the relative de Rham theory (\cite[Theorem 5.2]{Bo}), we can define the relative cohomology class $[\omega_{\textbf{f}}]_{\text{rel}}$ in $H^n(Y(\ideal{n})^{\text{BS}},D_{C_{\infty}}(\ideal{n});\C)$ whose image in $H^n(Y(\ideal{n}),\C)$ is $[\omega_{\mathbf{f}}]$. 
Let $[\delta_{\mathbf{f}}]_{\mathrm{rel}}^{\epsilon_{\mathbf{f}}}$ denote the class $[\omega_{\mathbf{f}}]_{\mathrm{rel}}^{\epsilon_{\mathbf{f}}}/\Omega_{\mathbf{f}}^{\epsilon_{\mathbf{f}}}$ in $H^n(Y(\ideal{n})^{\mathrm{BS}},D_{C_{\infty}}(\ideal{n});\C)$. 
We put 
\begin{align*}
\widetilde{H}^n(Y(\ideal{n})^{\text{BS}},D_{C_{\infty}}(\ideal{n});\integer{})
&=H^n(Y(\ideal{n})^{\text{BS}},D_{C_{\infty}}(\ideal{n});\integer{})/\text{($\integer{}$-torsion)},\\
\widetilde{H}^n(Y(\ideal{n}),\integer{})
&=H^n(Y(\ideal{n}),\integer{})/\text{($\integer{}$-torsion)}.
\end{align*}

\begin{prop}\label{prop:integral of relative class}
Let $\mathbf{f}\in S_2(\ideal{n},\integer{})$ be a normalized Hecke eigenform for all $T(\ideal{q})$ and $U(\ideal{q})$ with character $\varepsilon$. 
Assume that $H^n(D_{C_{\infty}}(\ideal{n}),\integer{})_{\ideal{m}_{\mathbf{f}}'}$ is torsion-free, and $C(\ideal{q},\mathbf{f})\not\equiv N(\ideal{q})\,(\bmod\,\varpi)$ for some prime ideal $\ideal{q}$ of $\ideal{o}_F$ dividing $\ideal{n}$. 
Then $[\delta_{\mathbf{f}}]_{\mathrm{rel}}^{\epsilon_{\mathbf{f}}}$ is integral, that is,
\[
[\delta_{\mathbf{f}}]_{\mathrm{rel}}^{\epsilon_{\mathbf{f}}}
\in \widetilde{H}^n(Y(\ideal{n})^{\mathrm{BS}},D_{C_{\infty}}(\ideal{n});\integer{}).
\]
\end{prop}
\begin{proof}
We abbreviate $Y(\ideal{n})$ to $Y$ and $D_{C_{\infty}}(\ideal{n})$ to $D_{C_{\infty}}$. 
We have an exact sequence 
\[
H^{n-1}(D_{C_{\infty}},\integer{})_{\ideal{m}_{\textbf{f}}'}\to H^n(Y^{\mathrm{BS}},D_{C_{\infty}};\integer{})_{\ideal{m}_{\textbf{f}}'}\to H^n(Y,\integer{})_{\ideal{m}_{\textbf{f}}'}\to H^n(D_{C_{\infty}},\integer{})_{\ideal{m}_{\textbf{f}}'}.
\]
By \cite[Proposition 5.3]{Hira}, $H^{n-1}(D_{C_{\infty}},\integer{})_{\ideal{m}_{\textbf{f}}'}$ is torsion. 
By the assumption, $H^n(D_{C_{\infty}},\integer{})_{\ideal{m}_{\textbf{f}}'}$ is torsion-free. 
Therefore we obtain an exact sequence
\[
0\to \widetilde{H}^n(Y^{\mathrm{BS}},D_{C_{\infty}};\integer{})_{\ideal{m}_{\textbf{f}}'}\to \widetilde{H}^n(Y,\integer{})_{\ideal{m}_{\textbf{f}}'}\to H^n(D_{C_{\infty}},\integer{})_{\ideal{m}_{\textbf{f}}'}.
\]
Now the assertion follows from this exact sequence. 
\end{proof}

%
\subsection{Integrality of $p$-adic $L$-functions for Hilbert modular forms}\label{subsection:construction of p-adic L}
%

Let $F_{\infty}$ denote the cyclotomic $\Zp$-extension of $F$. 
We put $\Gamma=\Gal(F_{\infty}/F)$. 
Let $\gamma$ be a topological generator of $\Gamma$. 
\begin{thm}\label{thm:integral p-adic L}
Let $\mathbf{f}\in S_2(\ideal{n},\integer{})$ be a normalized Hecke eigenform for all $T(\ideal{q})$ and $U(\ideal{q})$ with character $\varepsilon$. 
Assume that $\mathbf{f}$ is $p$-ordinary. 
Let $\ideal{m}$ be a non-zero ideal of $\ideal{o}_F$ such that $(\ideal{m},p)=1$. 
We put $\mathbf{g}=(\mathbf{f}\otimes\mathbf{1}_{\ideal{n}})\otimes \mathbf{1}_{\ideal{m}}\in S_2(\ideal{n}',\integer{})$ $($cf. \cite[Proposition 4.4, 4.5]{Shi}$)$, where, for a non-zero ideal $\ideal{a}$ of $\ideal{o}_F$, $\mathbf{1}_{\ideal{a}}$ is the trivial character modulo $\ideal{a}$, and $(\ideal{n'},p)=1$.
Let $\chi$ be a narrow ray class character of $F$ whose conductor is $\ideal{n}'$ such that  
$\chi=\epsilon_{\mathbf{f}}$ on $W_G$.
Assume that $H^n(D_{C_{\infty}}(\ideal{n}'),\integer{})_{\ideal{m}_{\mathbf{g}}'}$ is torsion-free, where $\ideal{m}_{\mathbf{g}}'$ is the maximal ideal of $\mathbb{H}_2(\ideal{n}',\integer{})'$ defined before Proposition \ref{prop:integral of relative class}.
If $\chi$ is of type $S$ $($that is, $F_{\chi} \cap F_{\infty}=F$, where $F_{\chi}=\overline{\Q}^{\ker(\chi)}$$)$, then there exists a $p$-adic $L$-function $\mathscr{L}_p(\mathbf{f},\chi,T)\in \integer{}(\chi)[[T]]$ satisfying the following interpolation property:
for each finite order character $\rho$ of $\Gamma$ with conductor $p^{\nu_{\rho}}$, 
\begin{align}\label{p-adic L}
\mathscr{L}_p(\mathbf{f},\chi,\rho(\gamma)-1)
=\alpha^{-\nu_{\rho}}\tau(\chi^{-1}\rho^{-1})\frac{D(1,\mathbf{f},\chi\rho)}{(-2\pi \sqrt{-1})^n\Omega_{\mathbf{g}}^{\epsilon_{\mathbf{f}}}}\in \integer{}(\chi,\rho).
\end{align}
Here $\Omega_{\mathbf{g}}^{\epsilon_{\mathbf{f}}}\in \C^{\times}$ is the canonical period of $\mathbf{g}\in S_2(\ideal{n}',\integer{})$ defined by $(\ref{period})$
\end{thm}
\begin{proof}
By Proposition \ref{prop:integral of relative class}, the class 
$[\delta_{\mathbf{g}}]_{\mathrm{rel}}^{\epsilon_{\mathbf{f}}}=[\omega_{\mathbf{g}}]_{\mathrm{rel}}^{\epsilon_{\mathbf{f}}}/\Omega_{\mathbf{g}}^{\epsilon_{\mathbf{f}}}$ is integral (because $C(\ideal{q},\mathbf{g})=0$ for every prime ideal $\ideal{q}$ of $\ideal{o}_F$ dividing $\ideal{n}$). 
Hence Proposition \ref{Mellin hol} and \ref{Mellin anti-hol} imply that the value in the right-hand side of (\ref{p-adic L}) belongs to $\integer{}(\chi,\rho)$ because $D(s,\mathbf{f},\chi\rho)=D(s,\mathbf{g},\chi\rho)$ and the fact \cite[Proposition 2.5, 2.6]{Hira} (that the evaluation maps factor through $H^n(Y(\ideal{n}')^{\mathrm{BS}},D_{C_{\infty}}(\ideal{n}');\C)$ as explained in \cite[\S 2.4]{Hira}), where we use the assumption that the conductor of $\chi$ is $\ideal{n}'$.

For a non-zero ideal $\ideal{m}$ of $\ideal{o}_F$ and $\underline{a}\in (\ideal{m}^{-1}\ideal{d}_F^{-1}[t_1]^{-1}/\ideal{d}_F^{-1}[t_1]^{-1})^{\times}/\ideal{o}_{F,+}^{\times}$, we define a map $\mu_{\mathbf{f},\alpha}^{\epsilon_{\mathbf{f}}}(-,\ideal{m})$ by 
\begin{align*}
\mu_{\mathbf{f},\alpha}^{\epsilon_{\mathbf{f}}}(\underline{a},\ideal{m})
&=\alpha^{-v(\ideal{m})}
\left(\prod_{\ideal{q}|p} \left(1-\alpha_{\ideal{q}}^{-1}\varepsilon(\ideal{q})R(\ideal{q})\right)\right)
\mathrm{ev}_{\underline{a},\C} ([\omega_{\mathbf{g}}]_{c}^{\epsilon_{\mathbf{f}}}).
\end{align*}
By Proposition \ref{prop:distribution}, $\mu_{\mathbf{f},\alpha}^{\epsilon_{\mathbf{f}}}$ satisfies the distribution property. 
By Proposition \ref{prop:interpolation}, 
the distribution $\mu_{\mathbf{f},\alpha}^{\epsilon_{\mathbf{f}}}/\Omega_{\mathbf{g}}^{\epsilon_{\mathbf{f}}}$ interpolates the value in the right-hand side of (\ref{p-adic L}). 
Now our assertion follows from the connection between $\integer{}(\chi)$-valued measures on $\Gal(F_{\chi}(\mu_{p^{\infty}})/F)$ and elements of $\integer{}(\chi)[[T]]$ (see, for example, \cite{Del--Ri}, \cite{Ri78}, \cite{Se}).
\end{proof}

%
\section{Equality between the Iwasawa $\lambda$-invariants}\label{section:Equality between the Iwasawa invariants}
%

%
\subsection{The Iwasawa $\lambda$-invariants for Hilbert modular forms}\label{subsection:application for alg and anal}
%

In this subsection, we define the analytic Iwasawa $\lambda$-invariant and state the main result of this paper. 

We keep the notation in \S \ref{subsection:construction of p-adic L}.
Let $\mathbf{f} \in S_2(\ideal{n}, \integer{})$ be a normalized Hecke eigenform with character $\varepsilon$. 
We assume that $\mathbf{f}$ is $p$-ordinary and 
$\mathbf{f}$ satisfies the conditions (RR), (Parity), and ($\mu=0$) in \S \ref{subsection:algebraic Iwasawa invariants}. 
Let $\ideal{n}'$ be the least common multiple of $\ideal{n}^2$ and $\ideal{m}_{\varepsilon} \ideal{m}_{\psi}^2$. 
We put $\mathbf{g}=(\mathbf{f}\otimes \mathbf{1}_{\ideal{n}})\otimes \mathbf{1}_{\ideal{m}_{\psi}}\in S_2(\ideal{n}',\integer{})$ (cf. \cite[Proposition 4.4, 4.5]{Shi}). 
Let $\mathcal{H}_2(\ideal{n}',\integer{})$ be the commutative $\integer{}$\nobreakdash-subalgebra of 
$\End_{\integer{}}(H_c^n(Y(\ideal{n}'),\integer{})) \oplus \End_{\integer{}}(H^n (Y(\ideal{n}'), \integer{})) \oplus \End_{\integer{}}(H^n (\partial(Y(\ideal{n}')^{\mathrm{BS}}),\integer{})) \oplus \End_{\integer{}}(H_c^{n+1}(Y(\ideal{n}'),\integer{}))$ generated by $T(\ideal{q})$, $S(\ideal{q})$ for all non\nobreakdash-zero prime ideals $\ideal{q}$ of $\ideal{o}_F$ prime to $\ideal{n}'$, and $U(\ideal{q})$ for all non-zero prime ideals $\ideal{q}$ of $\ideal{o}_F$ dividing $\ideal{n}'$, and 
$\ideal{m}$ the maximal ideal of $\mathcal{H}_2(\ideal{n}',\integer{})$ generated by $\varpi$ and $T(\ideal{q})-C(\ideal{q},\mathbf{g})$, $S(\ideal{q})-\varepsilon^{-1}(\ideal{q})$ for all non-zero prime ideals $\ideal{q}$ of $\ideal{o}_F$ prime to $\ideal{n}'$, and $U(\ideal{q})-C(\ideal{q},\mathbf{g})$ for all non-zero prime ideals $\ideal{q}$ of $\ideal{o}_F$ dividing $\ideal{n}'$. 

Let $\chi$ be an $\integer{}$\nobreakdash-valued totally even narrow ray class character of $F$ whose conductor is $\ideal{n}'$. 
We define the analytic Iwasawa $\lambda$-invariant $\lambda_{\mathbf{f}\otimes \chi}^{\mathrm{an}}$ by 
\begin{align}\label{lambda^an for f and chi}
&\lambda_{\mathbf{f}\otimes \chi}^{\mathrm{an}}
=\lambda(\mathscr{L}_p(\mathbf{f},\chi,T))=\deg(P_{\chi}^{\text{an}}(T)).
\end{align}
Here $\mathscr{L}_p(\mathbf{f},\chi,T) \in \integer{}[[T]]$ is the $p$-adic $L$-function constructed in Theorem \ref{thm:integral p-adic L} and 
$P_{\chi}^{\mathrm{an}}(T)$ is the distinguished polynomial corresponding to 
$\mathscr{L}_p(\mathbf{f},\chi,T)$ via the Weierstrass preparation theorem (see, for example, \cite[\S 7.1, Theorem 7.3]{Was}). 

\begin{thm}\label{thm:equality of Iwasawa invariants}
Let $p$ be a prime number such that $p \ge n+2$ and $p$ is prime to $\mathfrak{n}$ and $6\Delta_F$. 
Assume that $h_F^+=1$. 
Let $\chi$ be an $\integer{}$-valued totally even narrow ray class character of $F$ satisfying the conditions at the beginning of \S \ref{subsection:application for algebraic}.
We assume the following two conditions$:$
\begin{enumerate}[$(a)$] 
\item the local components $H^{n}(\partial \left(Y(\ideal{n}')^{\mathrm{BS}}\right),\integer{})_{\ideal{m}}$ and $H_c^{n+1}(Y(\ideal{n}'),\integer{})_{\ideal{m}}$ are torsion-free$;$

\item the local component $H^{n}(D_{C_{\infty}}(\ideal{n}'),\integer{})_{\ideal{m}_{\mathbf{g}}'}$ is torsion-free, where $\ideal{m}_{\mathbf{g}}'$ is the maximal ideal of $\mathbb{H}_2(\ideal{n}',\integer{})'$ defined before Proposition \ref{prop:integral of relative class}.
\end{enumerate}
Then we have 
\begin{align*}
\lambda_{\mathbf{f}\otimes \chi}^{\mathrm{alg}}
&=\lambda_{\mathbf{f}\otimes \chi}^{\mathrm{an}}(=\lambda_{\varphi\chi,\Sigma_0}+\lambda_{\psi\chi,\Sigma_0}). 
\end{align*}
\end{thm}

%
\subsection{Proof of main result (Theorem \ref{thm:equality of Iwasawa invariants})}\label{subsection:proof of main theorem}
%

For the proof of Theorem \ref{thm:equality of Iwasawa invariants}, by the equality (\ref{alg.lambda}), it suffices to show that 
\begin{align*}
\lambda_{\mathbf{f}\otimes \chi}^{\mathrm{an}}
&=\lambda_{\varphi\chi,\Sigma_0}+\lambda_{\psi\chi,\Sigma_0}.
\end{align*}
It follows from a congruence between our $p$-adic $L$-function for $\mathbf{g}$ and the product of two Deligne\nobreakdash--Ribet $p$-adic $L$-functions (Theorem \ref{thm:congruence of p-adic L}) as explained below.

In order to state Theorem \ref{thm:congruence of p-adic L}, we first define 
the $p$-adic $L$-functions $\mathscr{L}_p(A_{\varphi},\chi,T)$ and $\mathscr{L}_p(A_{\psi},\chi,T)$, and the non-primitive $p$-adic $L$-functions $\mathscr{L}_p^{\Sigma_0}(A_{\varphi},\chi,T)$ and $\mathscr{L}_p^{\Sigma_0}(A_{\psi},\chi,T)$ for the Galois representations $A_{\varphi\chi}$ and $A_{\psi\chi}$ appearing in \S \ref{subsection:algebraic Iwasawa invariants} and \S \ref{subsection:application for algebraic}. 
We put $\Lambda=\integer{}[[\Gamma]] \simeq \integer{}[[T]]\,;\,\gamma\mapsto 1+T$.

(i)\ The $p$-adic $L$-function $\mathscr{L}_p(A_{\varphi},\chi,T)\in \Lambda$ 
is defined by the interpolation property 
\[
\mathscr{L}_p(A_{\varphi},\chi,\rho(\gamma) -1)
=L_F(0,\chi\varepsilon \psi^{-1} \rho)
=L_F(1,\chi\varepsilon\chi_{\mathrm{cyc}} \psi^{-1} \rho)
\]
for every non-trivial finite order character $\rho$ of $\Gamma$. 
Here we remark that, by our assumption, $\chi\varepsilon\chi_{\mathrm{cyc}}\psi^{-1}$ is non-trivial character and hence $L_F(s,\chi\varepsilon\chi_{\mathrm{cyc}}\psi^{-1} \rho)$ is holomorphic for all $s\in \C$. 
Then, $\mathscr{L}_p(A_{\varphi},\chi,T)$ is related to the 
Deligne--Ribet $p$-adic $L$-function $L_p^{\mathrm{DR}}(s,\chi\varepsilon\omega \psi^{-1})$ by
\[
L_p^{\mathrm{DR}}(s,\chi\varepsilon\omega \psi^{-1})
=\mathscr{L}_p(A_{\varphi},\chi,\kappa(\gamma)^{-s}-1)
\]
for any $s \in \Zp$. 
Here, $\kappa(\gamma)$ is the element of  $1+p\Zp$ 
which induces the action of $\gamma$ on $\mu_{p^{\infty}}$ 
under the identification $\Gamma \simeq \Gal(F(\mu_{p^{\infty}})/F(\mu_p))$. 
The non-primitive $p$-adic $L$-function $\mathscr{L}_p^{\Sigma_0}(A_{\varphi},\chi,T)\in \Lambda$ is defined by the interpolation property 
\[
\mathscr{L}_p^{\Sigma_0}(A_{\varphi},\chi,\rho(\gamma) -1)
=L_F^{\{\ideal{n}'\}}(0,\chi\varepsilon \psi^{-1} \rho)
\]
for every non-trivial finite order character $\rho$ of $\Gamma$. 
Here $L_F^{\{\ideal{n}'\}}(s,\ast)$ denotes the complex $L$-functions formed from $L_F(s,\ast)$ by omitting the Euler factors for non-zero prime ideals dividing $\ideal{n}'$. 
Our assumption that $\mu(\Sel(F_{\infty},A_{\varphi\chi})^{\mathrm{PD}})=0$ and the Wiles theorem \cite[Theorem 1.3, Theorem 1.4]{Wil90} (see also \cite[Proposition 9]{Gre12}) assert that 
$\mathscr{L}_p^{\Sigma_0}(A_{\varphi},\chi,T)\notin\varpi \Lambda$ and 
the $\lambda$-invariant of $\mathscr{L}_p^{\Sigma_0}(A_{\varphi},\chi,T)$ is equal to $\lambda_{\chi\varepsilon\omega\psi^{-1},\Sigma_0}=\lambda_{\varphi\chi,\Sigma_0}$.

(ii)\ The $p$-adic $L$-function $\mathscr{L}_p(A_{\psi},\chi,T)\in \Lambda$ 
is defined by the interpolation property 
\[
\mathscr{L}_p(A_{\psi},\chi,\rho(\gamma) -1)
=\tau(\chi^{-1}\psi^{-1} \rho^{-1})
\frac{L_F(1,\chi\psi \rho)}{(-2\pi \sqrt{-1})^n}
=\frac{(-1)^n}{2^n \Delta_F^{1/2}}L_F(0,\chi^{-1}\psi^{-1} \rho^{-1})
\]
for every non-trivial finite order character $\rho$ of $\Gamma$. 
Here we remark that, by our assumption, $\chi^{-1}\omega\psi^{-1}$ is non-trivial character. 
Then, $\mathscr{L}_p(A_{\psi},\chi,T)$ is related to the Deligne--Ribet $p$-adic $L$-function $L_p^{\mathrm{DR}}(s,\chi^{-1} \omega \psi^{-1})$ by
\[
L_p^{\mathrm{DR}}(s,\chi^{-1} \omega \psi^{-1})
=\left(\frac{(-1)^n}{2^n \Delta_F^{1/2}}\right)^{-1}\mathscr{L}_p(A_{\psi},\chi,\kappa(\gamma)^{s} -1)
\]
for any $s \in \Zp$. 
The non-primitive $p$-adic $L$-function $\mathscr{L}_p^{\Sigma_0}(A_{\psi},\chi,T)\in \Lambda$ 
is defined by the interpolation property 
\[
\mathscr{L}_p^{\Sigma_0}(A_{\psi},\chi,\rho(\gamma) -1)
=\tau(\chi^{-1}\psi^{-1} \rho^{-1})
\frac{L_F^{\{\ideal{n}'\}}(1,\chi\psi \rho)}{(-2\pi \sqrt{-1})^n}
\]
for every non-trivial finite order character $\rho$ of $\Gamma$. 
Again by the Wiles theorem, the $\mu$-invariant of $\mathscr{L}_p(A_{\psi},\chi,T)$ is zero and its $\lambda$-invariant is equal to 
$\lambda_{\chi^{-1}\omega \psi^{-1}}
=\lambda_{\psi\chi}$.

Next, we define the $p$-adic $L$-function $\mathscr{L}_p(\mathbf{G},\chi,T)\in \Lambda$ for the Eisenstein series $\mathbf{G}\in M_2(\ideal{n}',\integer{})$ 
with character $\varepsilon$ characterized by 
\[
D(s,\mathbf{G})
=L_F^{\{\ideal{n}'\}}(s,\psi)L_F^{\{\ideal{n}'\}}(s-1,\varepsilon \psi^{-1}).
\]
Note that Theorem \ref{thm:lifting} assures the existence of such $\mathbf{G}$. 

(iii)
The $p$-adic $L$-function $\mathscr{L}_p(\mathbf{G},\chi,T)\in \Lambda$ is defined by the interpolation property
\begin{align*}
\mathscr{L}_p(\mathbf{G},\chi,\rho(\gamma) -1)
&=\tau(\chi^{-1} \psi^{-1}\rho^{-1})
\frac{D(1,\mathbf{G},\chi\rho)}{(-2\pi \sqrt{-1})^n}\\
&=L_F^{\{\ideal{n}'\}}(0,\chi\varepsilon \psi^{-1} \rho) 
\tau(\chi^{-1}\psi^{-1}\rho^{-1})
\frac{L_F^{\{\ideal{n}'\}}(1,\chi\psi \rho)}{(-2\pi \sqrt{-1})^n}
\end{align*}
for every non-trivial finite order character $\rho$ of $\Gamma$. 
Then clearly we have 
\[
\mathscr{L}_p(\mathbf{G},\chi,T)
=\mathscr{L}_p^{\Sigma_0}(A_{\varphi},\chi,T)
\mathscr{L}_p^{\Sigma_0}(A_{\psi},\chi,T). 
\]
Therefore, the $\mu$-invariant of $\mathscr{L}_p(\mathbf{G},\chi,T)$ is zero 
and the $\lambda$-invariant of $\mathscr{L}_p(\mathbf{G},\chi,T)$
is equal to $\lambda_{\varphi\chi,\Sigma_0}+\lambda_{\psi\chi,\Sigma_0}$.
Now Theorem \ref{thm:equality of Iwasawa invariants} follows from the following:
\begin{thm}\label{thm:congruence of p-adic L}
Under the same assumptions as Theorem \ref{thm:equality of Iwasawa invariants}, we have 
\begin{align}\label{congruence of p-adic L}
\mathscr{L}_p(\mathbf{f},\chi,T)\equiv U(T)\mathscr{L}_p(\mathbf{G},\chi,T) 
\,(\bmod\, \varpi \Lambda).
\end{align}
Here $U(T)$ is a unit in $\Lambda^{\times}$ characterized by 
\[
U(\rho(\gamma)-1)=u'\alpha^{-\nu_{\rho}}
\frac{\tau(\chi^{-1}\rho^{-1})}{\tau(\chi^{-1}\psi^{-1}\rho^{-1})}
\]
for every non-trivial finite order character $\rho$ of $\Gamma$ with conductor $p^{\nu_{\rho}}$, and some $p$-adic unit $u'\in \integer{}^{\times}$. 
In particular, we have
\begin{align}\label{equality of the analytic Iwasawa invariants}
\lambda_{\mathbf{f}\otimes \chi}^{\mathrm{an}}
&=\lambda_{\varphi\chi,\Sigma_0}+\lambda_{\psi\chi,\Sigma_0}.
\end{align}
\end{thm}
\begin{proof}
By our assumption, in order to apply \cite[Theorem 6.1]{Hira} to the quadruplet $(\mathbf{g}, \mathbf{G}, \ideal{n}',\chi\rho)$ instead of $(\mathbf{f}, \mathbf{E}, \ideal{n}, \eta)$, 
it suffices to check the congruence $\mathbf{g}\equiv \mathbf{G}\,(\bmod\, \varpi)$, that is, 
\[
C(\ideal{q},\mathbf{g})\equiv C(\ideal{q},\mathbf{G})\,(\bmod\, \varpi)
\]
for every prime ideal $\ideal{q}$ of $\ideal{o}_F$. 
For every prime ideal $\ideal{q}$ of $\ideal{o}_F$ such that $\ideal{q}\nmid \ideal{n}'p$, we have 
\begin{align*}
C(\ideal{q},\mathbf{g})=C(\ideal{q},\mathbf{f})&= \Tr(\rho_{\mathbf{f}}(\Frob_{\ideal{q}}))\\
&\equiv\psi(\Frob_{\ideal{q}})+\varphi(\Frob_{\ideal{q}})\\
&\equiv\psi(\Frob_{\ideal{q}})+\det(\rho_{\mathbf{f}})\psi^{-1}(\Frob_{\ideal{q}})\\
&=\psi(\Frob_{\ideal{q}})+\varepsilon \chi_{\mathrm{cyc}} \psi^{-1}(\Frob_{\ideal{q}})\\
&=\psi(\ideal{q})+\varepsilon(\ideal{q})\psi^{-1}(\ideal{q})N(\ideal{q})\\
&=C(\ideal{q},\mathbf{G}) \ (\bmod\, {\varpi}).
\end{align*}
For every prime ideal $\ideal{p}$ of $\ideal{o}_F$ lying above $p$, by (\ref{rho:decom}), (RR), and (Parity), we obtain 
\begin{align*}
&C(\ideal{p},\mathbf{g})=C(\ideal{p},\mathbf{f})\equiv \psi(\Frob_{\ideal{p}})\equiv C(\ideal{p},\mathbf{G}) \, (\bmod\, {\varpi}).
\end{align*}
Moreover, by the definitions of $\mathbf{g}$ and $\mathbf{G}$, for every prime ideal $\ideal{q}$ of $\ideal{o}_F$ dividing $\ideal{n}'$, we have 
\[
C(\ideal{q},\mathbf{g})=0,\ \ C(\ideal{q},\mathbf{G})=0.
\]
Thus, by applying \cite[Theorem 6.1]{Hira}, we get the congruence 
\[
\tau(\chi^{-1}\rho^{-1})\frac{D(1,\mathbf{g},\chi\rho)}{(-2\pi \sqrt{-1})^n\Omega_{\mathbf{g}}^{\mathbf{1}}}
\equiv u' \tau(\chi^{-1}\rho^{-1})\frac{D(1,\mathbf{G},\chi\rho)}{(-2\pi \sqrt{-1})^n}
\, (\bmod \, \varpi) 
\]
for every non-trivial finite order character $\rho$ of $\Gamma$ with conductor $p^{\nu_{\rho}}$, 
where $u'$ is a unit in $\integer{}^{\times}$. 
Therefore we obtain 
\begin{align*}
\mathscr{L}_p(\mathbf{f},\chi,\rho(\gamma)-1)
&=\alpha^{-\nu_{\rho}}\tau(\chi^{-1}\rho^{-1})
\frac{D(1,\mathbf{g},\chi\rho)}{(-2\pi \sqrt{-1})^n\Omega_{\mathbf{g}}^{\mathbf{1}}}\\
&\equiv u'\alpha^{-\nu_{\rho}}\tau(\chi^{-1}\rho^{-1})\frac{D(1,\mathbf{G},\chi\rho)}{(-2\pi \sqrt{-1})^n}\\
&= u'\alpha^{-\nu_{\rho}}
\frac{\tau(\chi^{-1}\rho^{-1})}{\tau(\chi^{-1}\psi^{-1}\rho^{-1})}
\tau(\chi^{-1}\psi^{-1}\rho^{-1})\frac{D(1,\mathbf{G},\chi\rho)}{(-2\pi \sqrt{-1})^n}\\
&= U(\rho(\gamma)-1)\mathscr{L}_p(\mathbf{G},\chi,\rho(\gamma)-1) 
\, (\bmod \, \varpi)
\end{align*}
for every non-trivial finite order character $\rho$ of $\Gamma$ with conductor $p^{\nu_{\rho}}$.
This proves the congruence (\ref{congruence of p-adic L}). 
Now the equality (\ref{equality of the analytic Iwasawa invariants}) follows from (\ref{congruence of p-adic L}) and $\mu(\mathscr{L}_p(\mathbf{G},\chi,T))=0$.
\end{proof}

\section{Modularity of residually reducible representations}\label{section:Modularity}
%

%
\subsection{Modularity of residually reducible representations}\label{subsection:Example}
%
%
In this subsection, we prove a modularity theorem of residually reducible representations. 

In order to do it, 
we recall the definition and properties of Hilbert Eisenstein series of parallel weight $2$
(see, for more detail, \cite[\S 3]{Shi}, \cite[\S 2]{Da--Da--Po}):
\begin{thm}\label{thm:lifting}
Let $\psi_i$ be $\integer{}$-valued narrow ray class characters of $F$, whose conductor are denoted by $\ideal{m}_{\psi_i}$, with sign $r_i\in (\Z/2\Z)^{J_F}$ for $i=1,2$. 
We put $\ideal{n}=\ideal{m}_{\psi_1} \ideal{m}_{\psi_2}$ and $\varepsilon=\psi_1 \psi_2$. 
If $\varepsilon$ is totally even, then there exists 
$\mathbf{E}_2(\psi_1,\psi_2)=(E_2(\psi_1,\psi_2)_i)_{1 \le i \le h_F^+} \in M_2(\ideal{n},\integer{})$ with character $\varepsilon$, called a Hilbert Eisenstein series, such that 
\begin{align*}
D(s,\mathbf{E}_2(\psi_1,\psi_2))
&=L_F(s,\psi_1)L_F(s-1,\psi_2),\\
C(\ideal{a},\textbf{E}_2(\psi_1,\psi_2))
&=\sum_{\ideal{c}\mid\ideal{a}}\psi_1\left(\frac{\ideal{a}}{\ideal{c}}\right)\psi_2(\ideal{c})N(\ideal{c})\ \,
\text{for each non-zero ideal $\ideal{a}$ of $\ideal{o}_F$}.
\end{align*}
\end{thm}

Hereafter we assume that $h_F^+=1$.
Let $\ideal{n}$ be a non-zero ideal of $\ideal{o}_F$ prime to $6p\Delta_F$ and $\ideal{d}_F[t_1]$. 
We fix narrow ray class characters $\varphi$ and $\psi$ of $F$ satisfying the conditions $(\mu=0)$, $\ideal{m}_{\varphi}\ideal{m}_{\psi}=\ideal{n}$, and 
\begin{align*}
&\text{(Eis)}&
&\text{ $\varphi$ and $\psi$ are $\cal{O}$-valued totally even (resp. totally odd) such that $\varphi$ is non-trivial.}&
\end{align*}
Then $\varphi$ and $\psi$ satisfy the condition \cite[\S 5, (Eis condition)]{Hira}. 

Let $\mathbf{E}$ denote the Hilbert Eisenstein series $\mathbf{E}_2(\varphi,\psi)\in M_2(\ideal{n},\integer{})$ attached to $\varphi$ and $\psi$ as Theorem \ref{thm:lifting}. 
We define the character $\epsilon_{{}_{\mathbf{E}}}$ of the Weyl group $W_G$ by $\epsilon_{{}_{\mathbf{E}}}=\sgn^{J_F}$ (resp. $\epsilon_{{}_{\mathbf{E}}}=\mathbf{1}$) if both $\varphi$ and $\psi$ are totally even (resp. totally odd). 
We put $\varepsilon=\varphi\psi$. 

For each cusp $s \in C(\Gamma_{1}(\ideal{d}_F [t_1],\ideal{n}))$, 
let $a_s(0,E_1)$ denote the constant term of $E_1$ at $s$ (cf. \S \ref{Diri}). 
Let $s_0 \in C(\Gamma_{1}(\ideal{d}_F [t_1],\ideal{n}))$ such that $v_p(a_{s_0}(0,E_{1}))\le v_p(a_{s}(0,E_{1}))$ for each $s\in C(\Gamma_{1}(\ideal{d}_F [t_1],\ideal{n}))$, where $v_p$ denotes the $p$-adic valuation. 
We put 
\[
C=a_{s_0}(0,E_{1}). 
\]

\begin{thm}\label{thm:congruence of modular forms}
Let $p$ be a prime number $>3$ such that $p$ is prime to $\ideal{n}$ and $\Delta_F$. 
We assume the above conditions and the following three conditions$:$ 
\begin{enumerate}[$(a)$] 
\item \label{thm:congruence of modular forms:assumption:torsion-free}
$H^{n}(\partial \left(Y(\ideal{n})^{\mathrm{BS}}\right),\integer{})_{\ideal{m}}$, $H_c^{n+1}(Y(\ideal{n}),\integer{})_{\ideal{m}}$, and $H^{n}(D_{C_{\infty}}(\ideal{n}),\integer{})_{\ideal{m}_{\mathbf{E}}'}$ are torsion-free, 
where $\ideal{m}$ $($resp. $\ideal{m}_{\mathbf{E}}'$$)$ is the Eisenstein maximal ideal of the Hecke algebra $\cal{H}_2(\ideal{n},\integer{})$ $($resp. $\mathbb{H}_2(\ideal{n},\integer{})'$$)$ defined before \cite[Theorem 5.5]{Hira}$;$ $($resp. \cite[Proposition 5.3]{Hira}$)$$;$

\item \label{thm:congruence of modular forms:assumption:Eis eigenvalue}
$C(\ideal{q},\mathbf{E})\not\equiv N(\ideal{q})\, (\bmod \,\varpi)$ for some prime ideal $\ideal{q}$ dividing $\ideal{n}$$;$

\item \label{thm:congruence of modular forms:assumption:ideal}
the ideal $(C)\neq 0, \integer{}$. 
\end{enumerate}
Then there exist a finite extension $K'$ of $K$ with the ring of integers $\integer{}\hookrightarrow \cal{O}'$ and a uniformizer $\varpi'$ such that $(\varpi')\cap \integer{}=(\varpi)$, and a Hecke eigenform $\mathbf{f}\in S_2(\ideal{n},\cal{O}')$ for all $T(\ideal{q})$ and $U(\ideal{q})$ with character $\varepsilon$ such that
$\mathbf{f}\equiv \mathbf{E}\, (\bmod \,\varpi')$, that is, for every non-zero ideal $\ideal{a}$ of $\ideal{o}_F$,
\[
C(\ideal{a},\mathbf{f})\equiv C(\ideal{a},\mathbf{E})\,(\bmod \,\varpi').
\]
\end{thm}
\begin{proof}
We use the same notation as in the proof of \cite[Theorem 5.5]{Hira}. 
By the proof of \cite[Theorem 5.5, (5.8)]{Hira}, 
the assumption $(\ref{thm:congruence of modular forms:assumption:ideal})$ implies that there exists a non-zero class $e_0\in \widetilde{H}_{\pa}^n(Y(\ideal{n}),\integer{})[\epsilon_{{}_{\mathbf{E}}}]\backslash \varpi \widetilde{H}_{\pa}^n(Y(\ideal{n}),\integer{})[\epsilon_{{}_{\mathbf{E}}}]$ such that $e_0$ is cohomologous to $-[\omega_{\mathbf{E}}]$ modulo $\varpi$, and 
the Hecke eigenvalues of $e_0$ are the same as those of $-[\omega_{\mathbf{E}}]$ modulo $\varpi$ for all $t\in \mathscr{H}_2(\ideal{n},\integer{})$. 
Now the Deligne\nobreakdash--Serre lifting lemma (\cite[Lemma 6.11]{Del--Se}) in the case where 
$R=\integer{}$, $M=\widetilde{H}_{\pa}^n(Y(\ideal{n}),\integer{})[\epsilon_{{}_{\mathbf{E}}}]$, and $\mathbb{T}=\mathscr{H}_2(\ideal{n},\cal{O})$ 
says that there exist a finite extension $K'$ of $K$ with the ring of integers $\integer{}\hookrightarrow \cal{O}'$ and a uniformizer $\varpi'$ such that $(\varpi')\cap \integer{}=(\varpi)$, and a non-zero eigenvector 
$e\in \widetilde{H}_{\pa}^n(Y(\ideal{n}),\integer{})[\epsilon_{{}_{\mathbf{E}}}]\otimes \cal{O}'$ 
for all $t\in \mathbb{T}$ with eigenvalues $\lambda(V(\ideal{q}))$ such that 
\[
\lambda(V(\ideal{q}))\equiv C(\ideal{q},\mathbf{E})\, (\bmod\, \varpi')
\]
for all non-zero prime ideals $\ideal{q}$ of $\ideal{o}_F$ prime to $\ideal{n}$ (resp. dividing $\ideal{n}$) and $V(\ideal{q})=T(\ideal{q})$ (resp. $U(\ideal{q})$). 
By the isomorphism (\ref{+,+ decomp}), 
we obtain a Hecke eigenform $\mathbf{f}\in S_2(\ideal{n},\C)$ for all $T(\ideal{q})$ and $U(\ideal{q})$ such that $e=[\omega_{\mathbf{f}}]$. 
By using the relation between Hecke eigenvalues and Fourier coefficients, 
we may assume that $\mathbf{f}\in S_2(\ideal{n},\cal{O}')$ with character $\varepsilon$. 
Therefore we obtain the desired congruence 
$\mathbf{f}\equiv \mathbf{E}\,(\bmod\, \varpi')$.
\end{proof}

Let $\kappa'$ denote the residue field of $\integer{}'$. 
\begin{cor}\label{cor:modularity}
Under the same notation and assumptions as Theorem \ref{thm:congruence of modular forms}, the residual Galois representation 
$\overline{\rho}_{\mathbf{f}}=\rho_{\mathbf{f}}\,(\bmod\,\varpi'): G_{F}\to \GL_2(\kappa')$ associated to $\mathbf{f}$ is reducible and of the form 
\begin{align*}
\overline{\rho}_{\mathbf{f}}  \sim 
\begin{pmatrix} 
\varphi&\ast\\ 
0&\psi\chi_{\mathrm{cyc}}
\end{pmatrix} (\bmod\, \varpi').
\end{align*}
\end{cor}
\begin{proof}
By Theorem \ref{thm:congruence of modular forms}, 
the Hecke eigenform $\mathbf{f}\in S_2(\ideal{n},\cal{O}')$ satisfies the congruences 
\begin{align}\label{f equiv E}
C(\ideal{a},\mathbf{f})&\equiv \sum_{\ideal{c}\mid\ideal{a}}\varphi\left(\frac{\ideal{a}}{\ideal{c}}\right)\psi(\ideal{c})\chi_{\mathrm{cyc}}(\ideal{c})\, (\bmod \,\varpi')
\end{align}
for every non-zero ideal $\mathfrak{a}$ of $\ideal{o}_F$. 
Let $\overline{T}$ denote a $2$-dimensional $\kappa'$-vector space 
with the action of $G_{F}$ via $\overline{\rho}_{\mathbf{f}}$. 
For the proof, it suffices to show that all of the constituents of $\overline{T}$ are isomorphic to $\kappa'(\varphi)$ or $\kappa'(\psi\omega)$. 
We write $\rho$ for $\overline{\rho}_{\mathbf{f}}$ to simplify the notation. 
Let 
\begin{align*}
W&=\overline{T}\oplus \overline{T}{}^{\ast},
\end{align*}
where $\overline{T}{}^{\ast}=\Hom(\overline{T},\kappa'(\varepsilon\omega))$.
We fix a prime ideal $\ideal{q}$ of $\ideal{o}_F$ prime to $\ideal{n}p$. 
We consider the characteristic polynomial of $\Frob_{\ideal{q}}$ acting on $W$. We note that $\rho$ satisfies the relation 
$\rho(\Frob_{\ideal{q}})^2 -C(\ideal{q},\mathbf{f})\rho(\Frob_{\ideal{q}})+\varepsilon(\ideal{q})\chi_{\mathrm{cyc}}(\ideal{q})=0$. 
Let $\alpha(\ideal{q})$ and $\beta(\ideal{q})$ denote the solutions of $X^2-C(\ideal{q},\mathbf{f})X+\varepsilon(\ideal{q})\chi_{\mathrm{cyc}}(\ideal{q})=0$. 
Let $G$ denote a finite quotient of $G_F$ through which the action on $W$ factors. 
Let $N_{\alpha(\ideal{q})}$ and $N_{\beta(\ideal{q})}$ denote the generalized eigenspaces of $\rho(\Frob_{\ideal{q}})$ with respect to $\alpha(\ideal{q})$ and $\beta(\ideal{q})$, respectively.
Then we have $\overline{T}=N_{\alpha(\ideal{q})}\oplus N_{\beta(\ideal{q})}$. 
Since the operation $\Hom(\ast,\kappa'(\varepsilon\omega))$ interchanges the eigenvalues of the action of $\Frob_{\ideal{q}}$, 
the characteristic polynomial of $\Frob_{\ideal{q}}$ acting on $W$ is $(X-\alpha(\ideal{q}))^2(X-\beta(\ideal{q}))^2$. 
On the other hand, the congruence (\ref{f equiv E}) implies that the characteristic polynomial of $\Frob_{\ideal{q}}$ acting on $\kappa'(\varphi)^{\oplus 2}\oplus \kappa'(\psi\omega)^{\oplus 2}$, which is regarded as a $G$\nobreakdash-module, is also $(X-\alpha(\ideal{q}))^2(X-\beta(\ideal{q}))^2$. 
By the Chebotarev density theorem，any element of $G$ is the image of some $\Frob_{\ideal{q}}$ such that $\ideal{q}$ is prime to $\ideal{n}p$. 
Therefore, by the Brauer\nobreakdash--Nesbitt theorem, we obtain
\[
W^{\mathrm{ss}}\simeq \kappa'(\varphi)^{\oplus 2}\oplus \kappa'(\psi\omega)^{\oplus 2}.
\]
Here $W^{\mathrm{ss}}$ denotes the semisimplification of $W$. 
Hence there exists a filtration 
\begin{align*}
0=T_0 \subsetneq T_1 \subsetneq T_2=\overline{T}
\end{align*}
of $\overline{T}$ such that, for each $i$ with $1\le i\le 2$, 
\[
T_i/T_{i-1}\simeq \kappa'(\alpha_i), 
\]
where $\alpha_i=\varphi$ or $\psi\omega$. 
Since $\rho$ is totally odd, we have $\alpha_1 \neq \alpha_2$.
This proves the corollary. 
\end{proof}

%
\subsection{Real quadratic field case}\label{Example}
%
%

In this subsection, we give examples satisfying all the assumption of the Main Theorem \ref{thm:main theorem} by using Corollary \ref{cor:modularity}. 
In order to do it, 
we prove (\ref{thm:congruence of modular forms:assumption:torsion-free}) of Theorem \ref{thm:congruence of modular forms} in certain case (Proposition \ref{prop:ab torsion-free} and \ref{prop:bou torsion-free})
and give a Hilbert Eisenstein series satisfying (\ref{thm:congruence of modular forms:assumption:Eis eigenvalue}) and (\ref{thm:congruence of modular forms:assumption:ideal}) of Theorem \ref{thm:congruence of modular forms} (Example \ref{Example cong}). 

In this subsection, we assume that $F$ is a real quadratic field with $h_F^+=1$. The following proposition is obtained by \cite[Proposition 5.8]{Hira}:
\begin{prop}\label{prop:ab torsion-free}
Assume that $\ideal{n}$ is prime to $6\Delta_F$. 
If $p$ is prime to $6\ideal{n}$ and $\sharp(\mathfrak{o}_{F,+}^{\times}/\mathfrak{o}_{F,\ideal{n}}^{\times 2})$, 
then $H_c^3(Y(\ideal{n}),\integer{})$ is torsion-free.
\end{prop}

Let $u$ denote the fundamental unit of $F$. 
We put $u_+=u$  (resp. $u^2$) if $N(u)=1$ (resp. $N(u)=-1$).
Let $\iota(\ideal{n})$ denote the index $[\overline{\Gamma_{1}(\ideal{d}_F[t_1],\ideal{o}_F)}:\overline{\Gamma_{1}(\ideal{d}_F[t_1],\ideal{n})}]$ of the subgroup $\overline{\Gamma_{1}(\ideal{d}_F[t_1],\ideal{n})}$ in $\overline{\Gamma_{1}(\ideal{d}_F[t_i],\ideal{o}_F)}$, where the bar ${}^{-}$ means image in $\GL_2(F)/(\GL_2(F)\cap F^{\times})$. 

\begin{prop}\label{prop:bou torsion-free}
Assume that $\ideal{n}$ is prime to $6\Delta_F$. 
If $p$ is prime to $\iota(\ideal{n})$ and $u_+^{\iota(\ideal{n})}-1$, then $H^{2}(\partial \left(Y(\ideal{n})^{\mathrm{BS}}\right),\integer{})$ is torsion-free.
\end{prop}
\begin{proof}
We write $\Gamma$ for $\Gamma_{1}(\ideal{d}_F[t_1],\ideal{n})$ to simplify the notation. 
We may assume $\Gamma=\Gamma_1(\mathfrak{o}_F,\ideal{n})$ by taking conjugation. 
For the proof, it suffices to show that 
$H^2(\overline{\Gamma_s},\cal{O})=H^2(\overline{\alpha^{-1}\Gamma \alpha \cap B_{\infty}},\cal{O})$ is torsion-free for each cusp $s\in C(\Gamma)$, 
where $\alpha\in \SL_2(\mathfrak{o}_F)$ such that $\alpha(\infty)=s$, and $B_{\infty}$ denotes the standard Borel subgroup of upper triangular matrices. 
We simply write 
\[
G=\alpha^{-1}\Gamma_1(\ideal{o}_F,\ideal{o}_F) \alpha \cap B_{\infty},\ \,
N=\alpha^{-1}\Gamma \alpha \cap B_{\infty}.
\]
We note that $\alpha^{-1}\Gamma_1(\ideal{o}_F,\ideal{o}_F) \alpha=\GL_2(\ideal{o}_F)$. 
As mentioned in \cite[p.260]{Gha}, for the proof, it suffices to show that $H^1(\overline{N},K/\cal{O})$ is divisible. 
By (the proof of) \cite[\S 3.4.2, Proposition 4]{Gha}, $H^1(\overline{G},K/\integer{})$ is divisible. 
Now we reduce it to showing that $H^1(\overline{G},K/\integer{})=H^1(\overline{N},K/\integer{})$, that is, 
\[
\Hom(\overline{G}^{\mathrm{ab}},K/\integer{})
=\Hom(\overline{N}^{\mathrm{ab}},K/\integer{}), 
\]
where $\overline{N}^{\mathrm{ab}}$ and $\overline{G}^{\mathrm{ab}}$ denote the maximal abelian quotients of $\overline{N}$ and $\overline{G}$, respectively. 
Let $\psi:\overline{N}^{\mathrm{ab}} \to \overline{G}^{\mathrm{ab}}$ denote the canonical morphism induced by $\overline{N} \hookrightarrow \overline{G}$. 
We have a diagram 
\[\xymatrix{
&0 \ar[r] & \overline{N} \ar[d] \ar[r] & \overline{G} \ar[d] \\
0 \ar[r] &\ker(\psi) \ar[r] & \ \overline{N}^{\mathrm{ab}} \ar[r]^{\psi} &\ \overline{G}^{\mathrm{ab}} \ar[r] & \mathrm{cok}(\psi) \ar[r] & 0.
}\]
Then the assumption on $\iota(\ideal{n})$ implies that $\mathrm{cok}(\psi)$ has finite order prime to $p$. 
Hence $\psi$ induces an exact sequence 
\[
0\to \Hom(\overline{G}^{\mathrm{ab}},K/\integer{}) 
\to \Hom(\overline{N}^{\mathrm{ab}},K/\integer{})
\to \Hom(\ker(\psi),K/\integer{})
\to 0.
\]
Thus it is enough to show that 
\begin{align}\label{ker(psi)=0}
\Hom(\ker(\psi),K/\integer{})=0. 
\end{align}

We simply write 
\[
\overline{G}_{T_{\infty}}=\overline{G}/(\overline{G}\cap U_{\infty}),\ \, \overline{N}_{T_{\infty}}=\overline{N}/(\overline{N}\cap U_{\infty}),
\]
where $U_{\infty}$ (resp. $T_{\infty}$) denotes the unipotent radical (resp. the standard torus) of $B_{\infty}$. 
We have the exact sequence 
\begin{align}\label{exact seq N bar}
0\to \overline{N} \cap U_{\infty} \to \overline{N} \to \overline{N}_{T_{\infty}}\to 0.
\end{align}
We note that $\overline{N}_{T_{\infty}}$ acts on $\overline{N} \cap U_{\infty}$ via the exact sequence (\ref{exact seq N bar}). 
Let $\varphi_N : \overline{N}^{\mathrm{ab}}\twoheadrightarrow \overline{N}_{T_{\infty}}$ denote the morphism induced by $\overline{N} \twoheadrightarrow \overline{N}_{T_{\infty}}$, and let $N'$ denote the image of the morphism $(\overline{N} \cap U_{\infty})_{\overline{N}_{T_{\infty}}}\to \overline{N}^{\mathrm{ab}}$ induced by $\overline{N}\cap U_{\infty} \to \overline{N}\twoheadrightarrow \overline{N}^{\mathrm{ab}}$. 
We note that $\ker(\varphi_N)\subset N'$. 
Since the morphism $\overline{N}_{T_{\infty}} \to \overline{G}_{T_{\infty}}$ is injective, we have $\ker(\psi) \subset \ker(\varphi_N)$ and hence $\ker(\psi)\subset N'$.
Thus we obtain 
\begin{align}\label{ex seq ker and N'}
\Hom(\ker(\psi),K/\integer{}) 
\twoheadleftarrow \Hom(N' ,K/\integer{})
\hookrightarrow \Hom((\overline{N} \cap U_{\infty})_{\overline{N}_{T_{\infty}}},K/\integer{}).
\end{align}
Now, for the proof of (\ref{ker(psi)=0}), it suffices to show that the last term
\[
\Hom((\overline{N} \cap U_{\infty})_{\overline{N}_{T_{\infty}}},K/\integer{})=0.\]
Since $\overline{G}_{T_{\infty}} \simeq \mathfrak{o}_{F,+}^{\times}$ (by \cite[\S 3.4.2]{Gha}), we have $(\ideal{o}_{F,+}^{\times})^{\iota(\ideal{n})} 
\hookrightarrow \overline{N}_{T_{\infty}}
\hookrightarrow \ideal{o}_{F,+}^{\times}$. 
It induces $(\overline{N}\cap U_{\infty})_{(\ideal{o}_{F,+}^{\times})^{\iota(\ideal{n})}}\twoheadrightarrow 
(\overline{N} \cap U_{\infty})_{\overline{N}_{T_{\infty}}}$, which implies 
\begin{align*}
\Hom((\overline{N} \cap U_{\infty})_{\overline{N}_{T_{\infty}}},K/\integer{}) 
\hookrightarrow  \Hom((\overline{N}\cap U_{\infty})_{(\ideal{o}_{F,+}^{\times})^{\iota(\ideal{n})}},K/\integer{}).
\end{align*}
Since $\overline{G} \cap U_{\infty} \simeq \ideal{o}_F$ (by \cite[\S 3.4.2]{Gha}), we have 
$\iota(\ideal{n})\ideal{o}_F 
\hookrightarrow \overline{N} \cap U_{\infty}
\hookrightarrow \ideal{o}_F$.
Hence the assumption on $\iota(\ideal{n})$ and the snake lemma for 
\[\xymatrix{
&0 \ar[r] & \overline{N}\cap U_{\infty} \ar[d]^{\times(u_+^{\iota(\ideal{n})}-1)} \ar[r] & \ideal{o}_F \ar[d]^{\times(u_+^{\iota(\ideal{n})}-1)} \ar[r]& \ideal{o}_F/(\overline{N}\cap U_{\infty})\ar[r] \ar[d]^{\times(u_+^{\iota(\ideal{n})}-1)}&0 \\
&0 \ar[r] & \overline{N}\cap U_{\infty} \ar[r] & \ideal{o}_F \ar[r] & \ideal{o}_F/(\overline{N}\cap U_{\infty}) \ar[r] & 0
}\]
implies that 
\begin{align*}
\Hom((\overline{N}\cap U_{\infty})_{(\ideal{o}_{F,+}^{\times})^{\iota(\ideal{n})}},K/\integer{}) 
\simeq \Hom(\ideal{o}_F/(u_+^{\iota(\ideal{n})}-1),K/\integer{}).
\end{align*}
Now, by combining them, 
the assumption on $u_+^{\iota(\ideal{n})}-1$ implies the claim (\ref{ker(psi)=0}) as desired. 
\end{proof}

\begin{ex}\label{Example cong}
We give examples satisfying all the assumptions of the Main Theorem \ref{thm:main theorem} in the case where $F=\Q(\sqrt{2})$. 
We have $\mathfrak{o}_F=\Z[\sqrt{2}]$, $h_F^+=1$, $\Delta_F=8$, $u=1+\sqrt{2}$, and $u_+=3+2\sqrt{2}$. 
We put $G=\Gal(K/\Q)$ and $H=\Gal(K/F)$, where $K=\Q(\sqrt{2},\sqrt{20149})$. 

Let $\varepsilon :H \to \{\pm 1\}$ be the non-trivial character 
whose conductor is a prime ideal $(20149)$ of $\ideal{o}_F$.
Let $\varepsilon_1 :\Gal(\Q(\sqrt{20149})/\Q) \to \{\pm 1\}$ be the non-trivial character whose conductor is a prime ideal $(20149)$ of $\Z$,
and $\varepsilon_2 :\Gal(\Q(\sqrt{2\cdot 20149})/\Q) \to \{\pm 1\}$ the non-trivial character whose conductor is an ideal $(8\cdot 20149)$ of $\Z$.  
Then we have $\mathrm{Ind}^G_H \varepsilon=\varepsilon_1 \oplus \varepsilon_2$ as $G$-modules and hence $L_{\Q}(s,\mathrm{Ind}^G_H \varepsilon)=L_{\Q}(s,\varepsilon_1)\cdot L_{\Q}(s,\varepsilon_2)$.
Since $L_F(s,\varepsilon)=L_{\Q}(s,\mathrm{Ind}^G_H \varepsilon)$, thus we obtain
\begin{align*}
L_F(-1,\varepsilon)
&=2^2 \cdot 5\cdot 281 \cdot 4951 \cdot 13417. 
\end{align*}

Let $p$ be a prime number $\in \{281,4951,13417\}$. 
Let $\textbf{E}$ denote the Eisenstein series $\in M_2((20149)^2,\integer{})$ characterized by 
\[
D(s,\mathbf{E})
=L_F^{\{(20149)\}}(s,\varepsilon)L_F^{\{(20149)\}}(s-1,\mathbf{1}).
\]
In order to apply Theorem \ref{thm:congruence of modular forms} (and Corollary \ref{cor:modularity}) to the pair $(p,\mathbf{E})$, it suffices to check the assumptions of Proposition \ref{prop:ab torsion-free} and \ref{prop:bou torsion-free}. 
The former follows from that $p$ is prime to $\sharp (\ideal{o}_F/(20149)^2)^{\times}$, and the latter follows from that $p$ is prime to $\iota^1((20149)^2)$ and $u_+^{\iota^1((20149)^2)}-1$.
Here, for a non-zero ideal $\ideal{a}$ of $\ideal{o}_F$, $\iota^1(\ideal{a})$ denotes the index $[\overline{\Gamma_{1}^1(\ideal{o}_F,\ideal{o}_F)}:\overline{\Gamma_{1}^1(\ideal{o}_F,\ideal{a})}]$, which is explicitly given by 
\[
\iota^1(\ideal{a})=\frac{1}{2}\sharp (\ideal{o}_F/\ideal{a})^{\times}N(\ideal{a})\prod_{\ideal{q}|\ideal{a}}\left(1+\frac{1}{N(\ideal{q})}\right)\ \ \text{if}\ \ideal{a}\neq (2)
\]
(cf. \cite[Theorem 4.2.5]{Mi}). 
Now, by applying Corollary \ref{cor:modularity}, there exists $\mathbf{f}_0 \in S_2((20149)^2,\integer{}')$ such that $\overline{\rho}_{\mathbf{f}_0}^{\mathrm{ss}} \simeq \varepsilon\oplus \mathbf{1}_{(20149)}\chi_{\mathrm{cyc}} (\bmod\, \varpi')$, where $\mathbf{1}_{(20149)}$ denotes the trivial character modulo $(20149)$. 
Let $\theta$ be the non-trivial character of $\Gal(F(\sqrt{-19})/F)$ whose conductor is a prime ideal $(19)$ of $\ideal{o}_F$. 
We put $\ideal{n}=(20149)^2 (19)^2$ and $\mathbf{f}=\mathbf{f}_0\otimes \theta \in S_2(\ideal{n},\integer{}')$. 
Thus, by the same way as above, $\mathbf{f}$ and $\ideal{n}$ satisfy all the assumptions of the Main Theorem \ref{thm:main theorem}.
\end{ex}



\begin{thebibliography}{999}



\bibitem[Ami--Ve]{Ami--Ve} Y. Amice, J.V\'elu, 
{\em Distributions p-adiques associ\'es aux s\'eries de Hecke}, Journ\'ees Arithm\'etiques de Bordeaux (Conf., Univ. Bordeaux, Bordeaux, 1974),  Asterisque, Nos. \textbf{24-25}, Soc. Math. France, Paris, (1975), 119--131.



\bibitem[Bo]{Bo}A. Borel, 
{\em Stable real cohomology of arithmetic groups. II}, Manifolds and Lie groups (Notre Dame, Ind., 1980), pp. 21--55, 
Progr. Math., \textbf{14}, Birkh\"auser, Boston, Mass., 1981. 



\bibitem[Da--Da--Po]{Da--Da--Po} S. Dasgupta, H. Darmon, R. Pollack, 
{\em Hilbert modular forms and the Gross-Stark conjecture}, 
Ann. of Math. (2) \textbf{174} (2011), no. 1, 439--484. 



\bibitem[Del--Ri]{Del--Ri} P. Deligne, K. A. Ribet, 
{\em Values of abelian $L$-functions at negative integers over totally real fields}, 
Invent. Math. \textbf{59} (1980), no. 3, 227--286. 



\bibitem[Del--Se]{Del--Se} P. Deligne, J. P.  Serre, 
{\em Formes modulaires de poids $1$}, 
Ann. Sci. \'Ecole Norm. Sup. (4) \textbf{7} (1974), 507--530. 



\bibitem[Dim]{Dim}M. Dimitrov, 
{\em Galois representations modulo $p$ and cohomology of Hilbert modular varieties}, 
Ann. Sci. Ecole Norm. Sup. (4) \textbf{38} (2005) no. 4, 505--551. 



\bibitem[Ge--Go]{Ge--Go} J. Getz, M. Goresky, 
{\em Hilbert modular forms with coefficients in intersection homology and quadratic base change}, 
Progress in Mathematics, \textbf{298}. Birkh\"auser/Springer Basel AG, Basel, (2012)



\bibitem[Gha]{Gha} E. Ghate, 
{\em Adjoint $L$-values and primes of congruence for Hilbert modular forms}, 
Compositio Math. \textbf{132} (2002), no. 3, 243--281. 



\bibitem[Gre89]{Gre89} R. Greenberg, 
{\em Iwasawa theory for $p$-adic representations}, in: {\em Algebraic number theory}, Adv.\ Stud.\ Pure Math., \textbf{17}, 
Academic Press, Boston, MA (1989), 97--137.



\bibitem[Gre12]{Gre12} R. Greenberg, 
{\em On $p$-adic Artin $L$-functions II}, in {\em Iwasawa theory 2012} : state of the art and recent advances, Contributions in mathematical and computational sciences \textbf{7}, Springer-Verlag Berlin Heidelberg (2014), 227--245.



\bibitem[Gre--Ste]{Gre--Ste} R.\ Greenberg, G.\ Stevens, 
{\em $p$-adic $L$-functions and $p$-adic periods of modular forms}, 
Invent. Math. \text{111} (1993), no. 2, 407--447. 



\bibitem[Gre--Vat]{Gre--Vat} R.\ Greenberg, V.\ Vatsal, 
{\em On the Iwasawa invariants of elliptic curves}, Invent.\ Math., \textbf{142} (2000) no.\ 1, 17--63. 



\bibitem[Ha]{Ha} G, Harder, 
{\em Eisenstein cohomology of arithmetic groups. The case $\GL_2$}, 
Invent. Math. \textbf{89} (1987) no. 1, 37--118. 



\bibitem[Hida88]{Hida88} H. Hida, 
{\em On $p$-adic Hecke algebras for $\GL_2$ over totally real fields}, 
Ann. of Math. (2) \textbf{128} (1988) no. 2, 295--384. 



\bibitem[Hida91]{Hida91} H. Hida, 
{\em On p-adic L-functions of $\GL(2)\times\GL(2)$ over totally real fields}, 
Ann. Inst. Fourier (Grenoble) \textbf{41} (1991) no. 2, 311--391. 



\bibitem[Hida93]{Hida93} H. Hida, 
{\em $p$-ordinary cohomology groups for $\SL(2)$ over number fields}, 
Duke Math. J. \textbf{69} (1993) no. 2, 259--314. 



\bibitem[Hida94]{Hida94} H. Hida, 
{\em On the critical values of $L$-functions of $\GL(2)$ and $\GL(2)\times \GL(2)$}, 
Duke Math. J. \textbf{74} (1994) no. 2, 431--529.



\bibitem[Hira]{Hira} Y. Hirano, 
{\em Congruences between Hilbert modular forms of weight $2$, and special values of their $L$\nobreakdash-functions}, submitted.



\bibitem[Iwa59]{Iwa59} K. Iwasawa, 
{\em On $\Gamma$-extensions of algebraic number fields}, 
Bull. Amer. Math. Soc. \textbf{65} (1959), 183--226. 



\bibitem[Iwa73]{Iwa73} K. Iwasawa, 
{\em On $\mathbb{Z}_l$-extensions of algebraic number fields}, 
Ann. of Math. (2) \textbf{98} (1973), 246--326. 



\bibitem[Mi]{Mi} T. Miyake, 
{\em Modular forms}, 
Translated from the 1976 Japanese original by Yoshitaka Maeda. Reprint of the first 1989 English edition. Springer Monographs in Mathematics. Springer-Verlag, Berlin, (2006) 



\bibitem[MTT]{MTT} B. Mazur, J. Tate, J. Teitelbaum, 
{\em On $p$-adic analogues of the conjectures of Birch and Swinnerton\nobreakdash-Dyer}, 
Invent. Math. \textbf{84} (1986) no. 1, 1--48. 



\bibitem[NSW]{NSW}J. Neukirch, A. Schmidt, K. Wingberg, 
{\em Cohomology of number fields}, 
Second edition. Grundlehren der Mathematischen Wissenschaften [Fundamental Principles of Mathematical Sciences], \textbf{323}. Springer\nobreakdash-Verlag, Berlin, 2008.



\bibitem[Ri76]{Ri76}K. A. Ribet, 
{\em A modular construction of unramified $p$-extensions of $\Q(\mu_p)$}, 
Invent. Math. \textbf{34} (1976), no. 3, 151--162. 



\bibitem[Ri78]{Ri78} K. A. Ribet, 
{\em Report on $p$-adic $L$-functions over totally real fields}, 
Journ\'ees Arithm\'etiques de Luminy (Colloq. Internat. CNRS, Centre Univ. Luminy, Luminy, 1978), pp. 177--192, 
Ast\'erisque, \textbf{61}, Soc. Math. France, Paris, 1979. 



\bibitem[Se]{Se} J-p. Serre, 
{\em Sur le r\'esidu de la fonction zeta p-adique d'un corps de nombres}, 
C. R. Acad. Sci. Paris S\'er. A-B \textbf{287} (1978), no. 4, A183--A188. 



\bibitem[SE]{SE} J-P. Serre, 
{\em Local fields}, Graduate Texts in Mathematics, \textbf{67}, Springer-Verlag, New York-Berlin, (1979).



\bibitem[Shi]{Shi} G. Shimura, 
{\em The special values of the zeta functions associated with Hilbert modular forms}, Duke Math. J. \textbf{45} (1978), no. 3, 637--679. 



\bibitem[Ski--Ur]{Ski--Ur}C. Skinner, E. Urban, 
{\em The Iwasawa main conjectures for $\GL_2$}, Invent. Math. \textbf{195} (2014), no. 1, 1--277. 



\bibitem[Vat]{Vat} V. Vatsal, 
{\em Canonical periods and congruence formulae}, 
Duke Math. J. \textbf{98} (1999) no. 2, 397--419. 



\bibitem[Vi]{Vi} M. Vishik, 
{\em Non-archimedean measures associated with Dirichlet series}, 
Math. USSR\nobreakdash-Sb. \textbf{28} (1976) no.2, 216--228



\bibitem[Was]{Was} L. C. Washington, 
{\em Introduction to cyclotomic fields}, 
Second edition. Graduate Texts in Mathematics, \textbf{83}. Springer-Verlag, New York, (1997). 



\bibitem[We]{We}  T. Weston, 
{\em Iwasawa invariants of Galois deformations}, Manuscripta Math. \textbf{118} (2005), no. 2, 161--180.



\bibitem[Wil88]{Wil88}A. Wiles, 
{\em On ordinary $\lambda$-adic representations associated to modular forms}, 
Invent. Math. \textbf{94} (1988), no. 3, 529--573. 



\bibitem[Wil90]{Wil90}A. Wiles, 
{\em The Iwasawa conjecture for totally real fields}, 
Ann. of Math. (2) \textbf{131} (1990) no. 3, 493--540.



\end{thebibliography}
\end{document}